\pdfoutput=1

\documentclass[smallcondensed]{svjour3}

\usepackage{eikonal}
\usepackage{subfiles}

\begin{document}

\title{
  Ordered Line Integral Methods for Solving the Eikonal Equation
  \footnote{This work was partially supported by NSF Career Grant
    DMS1554907 and MTECH grant No.\ 6205.}
}

\author{Samuel F. Potter \and Maria K. Cameron}

\institute{
  Samuel F. Potter
  \at
  Department of Computer Science \\
  University of Maryland \\
  \email{sfp@umiacs.umd.edu}
  \and
  Maria K. Cameron
  \at
  Department of Mathematics \\
  University of Maryland \\
  \email{cameron@math.umd.edu}
}

\date{Received: \today}

\maketitle

\begin{abstract}
  We present a family of fast and accurate Dijkstra-like solvers for
  the eikonal equation and factored eikonal equation which compute
  solutions on a regular grid by solving local variational
  minimization problems. Our methods converge linearly but compute
  significantly more accurate solutions than competing first order
  methods. In 3D, we present two different families
  of algorithms which significantly reduce the number of FLOPs needed
  to obtain an accurate solution to the eikonal equation. One method
  employs a fast search using local characteristic directions to prune
  unnecessary updates, and the other uses the theory of constrained
  optimization to achieve the same end. The proposed solvers are more
  efficient than the standard fast marching method in terms of the
  relationship between error and CPU time. We also modify our method
  for use with the additively factored eikonal equation, which can be
  solved locally around point sources to maintain linear
  convergence. We conduct extensive numerical simulations and provide
  theoretical justification for our approach. A library that
  implements the proposed solvers is available on GitHub.
\end{abstract}

\section{Introduction}\label{sec:introduction}

We develop fast, memory efficient, and accurate solvers for the
eikonal equation, a nonlinear hyperbolic PDE encountered in
high-frequency wave propagation~\cite{engquist2003computational} and
the modeling of a wide variety of problems in computational and
applied science~\cite{sethian1999level}, such as photorealistic
rendering~\cite{ihrke2007eikonal}, constructing signed distance
functions in the level set method~\cite{osher2006level}, solving the
shape from shading
problem~\cite{kimmel2001optimal,prados2006shape,durou2008numerical},
traveltime computations in numerical modeling of seismic wave
propagation~\cite{sethian19993,popovici20023,kim20023,van1991upwind,vidale1990finite},
and others. We are motivated primarily by problems in high-frequency
acoustics~\cite{prislan2016ray}, which are key to enabling a higher
degree of verisimilitude in virtual reality simulations
(see~\cite{raghuvanshi2014parametric,raghuvanshi2018parametric} for a
cutting-edge time-domain approach which is useful up to moderate
frequencies). Current approaches to acoustics simulations rely on
methods whose complexity depends on the highest frequency of the sound
being simulated. For moderately high-frequency wave propagation
problems, the eikonal equation comes about as the first term in an
asymptotic WKB expansion of the Helmholtz equation and corresponds to
the first arrival time of rays propagating under geometric optics,
although approaches for computing multiple arrivals
exist~\cite{fomel2002fast}.

In this work, we develop direct solvers for the eikonal equation which
are fast and accurate, particularly in 3D. We develop a family of
algorithms which approach the problem of efficiently computing updates
in 3D in different ways. This family of algorithms is analyzed and
extensive numerical studies are carried out. Our algorithms are
semi-Lagrangian, using information about local characteristics to
reduce the work necessary to get an accurate result. They are
competitive with existing direct solvers for the eikonal equation and
generalize to higher dimensions and related equations, such as the
static Hamilton-Jacobi equation. In fact, this research was done in
tandem with research on the ordered line integral methods for the
quasipotential of nongradient stochastic differential equations
(SDEs)~\cite{dahiya2017ordered,dahiya2018ordered,yang2019computing}.
Due to the relative simplicity of the eikonal equation, the algorithms
presented here are more amenable to analysis, allowing us to obtain
theoretical results that justify our experimental findings.

\subsection{Results}

\begin{figure}
  \centering
  \includegraphics[width=\linewidth]{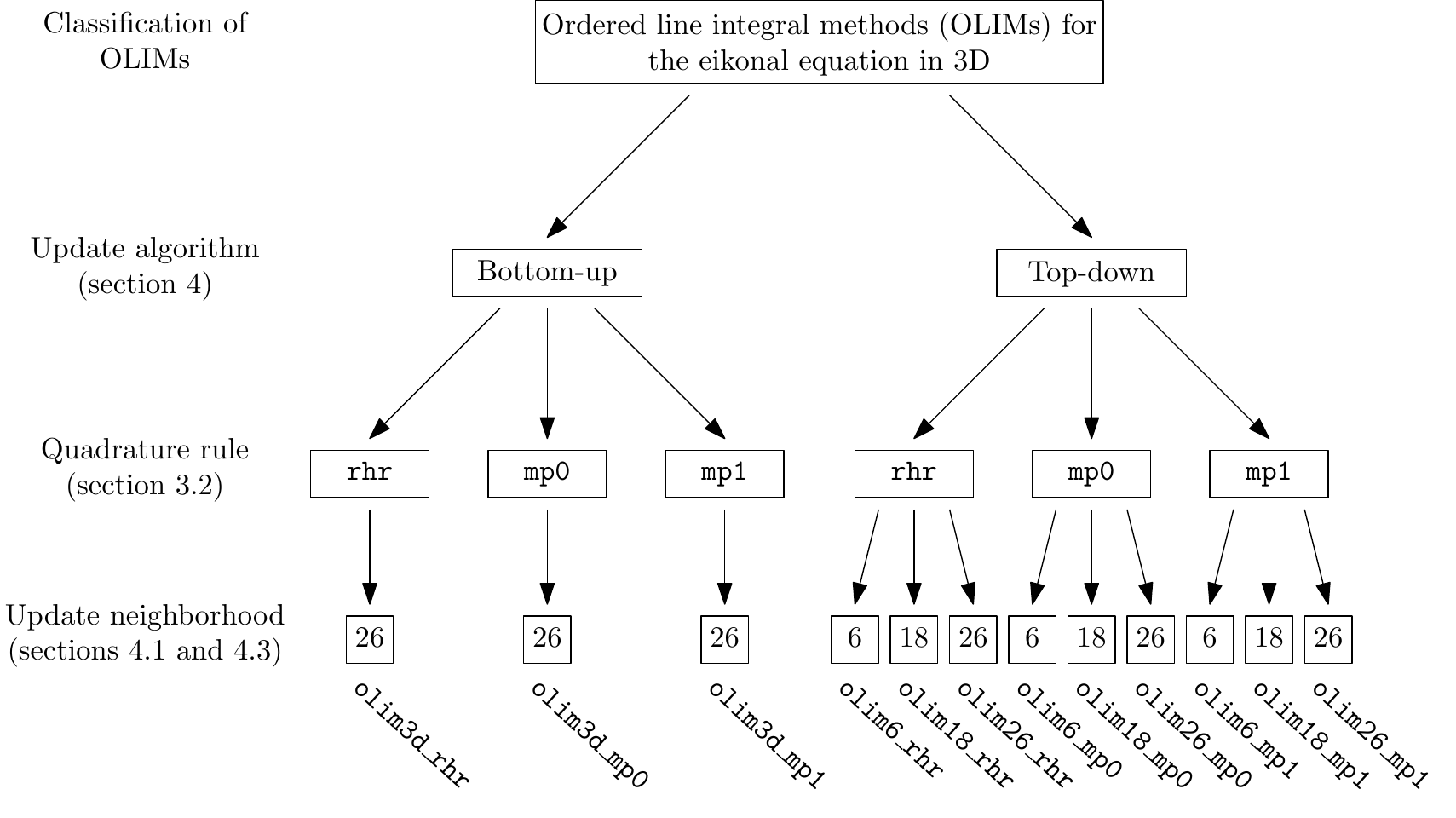}
  \caption{
    \emph{The family of Dijkstra-like solvers designed and studied in
      this work.}  We refer to these as ordered line integral methods
    (OLIMs).  There are three ways of parametrizing the family: by
    selecting an update algorithm, by selecting a quadrature rule, and
    by (in the case of the \emph{top-down} update algorithm, by
    selecting a neighborhood size. Sections in the text that explain
    these choices in detail are indicated. A shorthand notation for
    referring to each parametrized algorithms is listed for each
    algorithm that is involved in numerical tests (e.g.,
    \texttt{olim3d\_mp0}).
  }\label{fig:classification}
\end{figure}

Different numerical methods have been proposed for the solution of the
eikonal equation; generally, there are direct solvers and iterative
solvers. The most popular direct solvers are based on Dijkstra's
algorithm (``Dijkstra-like''
solvers)~\cite{tsitsiklis1995efficient,sethian1996fast}, and the most
popular iterative method is the fast sweeping
method~\cite{tsai2003fast,zhao2005fast}. In this work, we develop a
family of Dijkstra-like solvers for the eikonal equation in 2D and 3D,
similar to the fast marching method (FMM) or ordered upwind methods
(OUMs)~\cite{sethian1996fast,sethian2003ordered}. In constrast to the
FMM and OUMs that use finite difference schemes, our solvers come
about by discretizing and minimizing the action functional for the
eikonal equation (see section\@
\ref{sec:minimum-action-integral}). The proposed family of algorithms
is parameterized by a choice of update algorithm (\emph{bottom-up} or
\emph{top-down}), quadrature rule (a righthand rule (\texttt{rhr}), a
simplified midpoint rule (\texttt{mp0}), and a midpoint rule
(\texttt{mp1})), and, in the case of \emph{bottom-up}, a neighborhood
size (6, 18, or 26 points in 3D). See figure \ref{fig:classification}.

Our \emph{bottom-up} and \emph{top-down} update algorithms
represent two separate approaches to minimizing the number of triangle
and tetrahedron updates that need to be done in 3D, while the
quadrature rules represent a trade-off between speed and accuracy of
the solver. The simplified midpoint rule (\texttt{mp0}) is a sweet
spot that requires extra theoretical justification, which we
provide. Overall, our goal is to explore the relevant algorithm design
trade-offs in 3D and find which solver performs best. Our conclusion
is that, in 3D, \texttt{olim3d\_mp0} is the best overall, and our
results are oriented towards supporting this claim.

Additionally, we modify our algorithms to solve the additively
factored eikonal equation~\cite{luo2012fast}: to enhance accuracy, we
solve the locally factored eikonal equation near point sources, which
recovers the global $O(h)$ error convergence expected from a
first-order method, where $h > 0$ is the uniform spacing between grid
points. This fixes the degraded $O(h \log h^{-1})$ convergence often
associated with point source eikonal problems~\cite{qi2018corner}
(see~\cite{zhao2005fast} for a proof of this error bound).

Our main results follow:
\begin{itemize}
\item \textbf{For 3D problems, we develop two separate update
    algorithms:} a \emph{bottom-up} (\texttt{olim3d}) algorithm, and a
  \emph{top-down} algorithm (\texttt{olim\emph{K}}, where
  \texttt{\emph{K}} \hspace{-0.1em}$=6,18,26$ is the size of
  neighborhood used). Each algorithm locally updates a grid point by
  performing a minimal number of triangle or tetrahedron
  updates. Depending on the quadrature rule, each update is calculated
  by solving a system of nonlinear equations either directly
  (\texttt{rhr} and \texttt{mp0}) or iteratively (\texttt{mp1}).
\item \textbf{We prove theorems relating our quadrature rules,
    rigorously justifying the \texttt{mp0} rule.} These results
  support our case that it is superior to the \texttt{mp1} rule. We
  note that this work was done in tandem with research on ordered line
  integral methods for computing the quasipotential 3D for nongradient
  SDEs~\cite{dahiya2017ordered,yang2019computing,dahiya2018ordered}. Unlike
  the quasipotential, the eikonal equation is simple enough to allow
  us to analyze and justify our algorithms. We are also able to obtain
  simpler solution methods and established performance guarantees.
\item \textbf{We conduct numerical experiments on test problems with
    analytic solutions.} The test problems include point source
  problems for different slowness (index of refraction) functions, and
  multiple point source problems with a linear speed function. All of
  these have analytical solutions, which we use as a ground truth. We
  also test our
\item \textbf{We perform tests involving the Marmousi model in 2D and
    3D.} We demonstrate that the improved directional coverage of
  \texttt{olim26} and \texttt{olim3d} leads to a gain in accuracy.
\item \textbf{We show that a significant improvement in accuracy is
    gained over the equivalent of the standard fast marching method in
    3D, \texttt{olim6\_rhr}.} Only a modest slowdown is incurred using
  our general framework, indicating that our approach is
  competitive. See figure \ref{fig:intro} to see the improvement of
  \texttt{olim3d\_mp0} over \texttt{olim6\_rhr}, and see section\@
  \ref{sec:numerical-results} for more details.
\item \textbf{We use Valgrind~\cite{nethercote2007valgrind} to profile
    our implementation.} Our results indicate that the time spent
  sorting the heap used to order nodes on the front is negligible for
  all practical problem sizes. Since our solvers otherwise run in
  $O(N^n)$ time, where $n$ is the dimension of the domain, we suggest
  that the $O(N^n \log N)$ cost of the algorithm is
  pessimistic. Memory access patterns play a much more significant
  role in scaling.
\end{itemize}

\begin{figure}
  \centering
  \includegraphics[width=\linewidth]{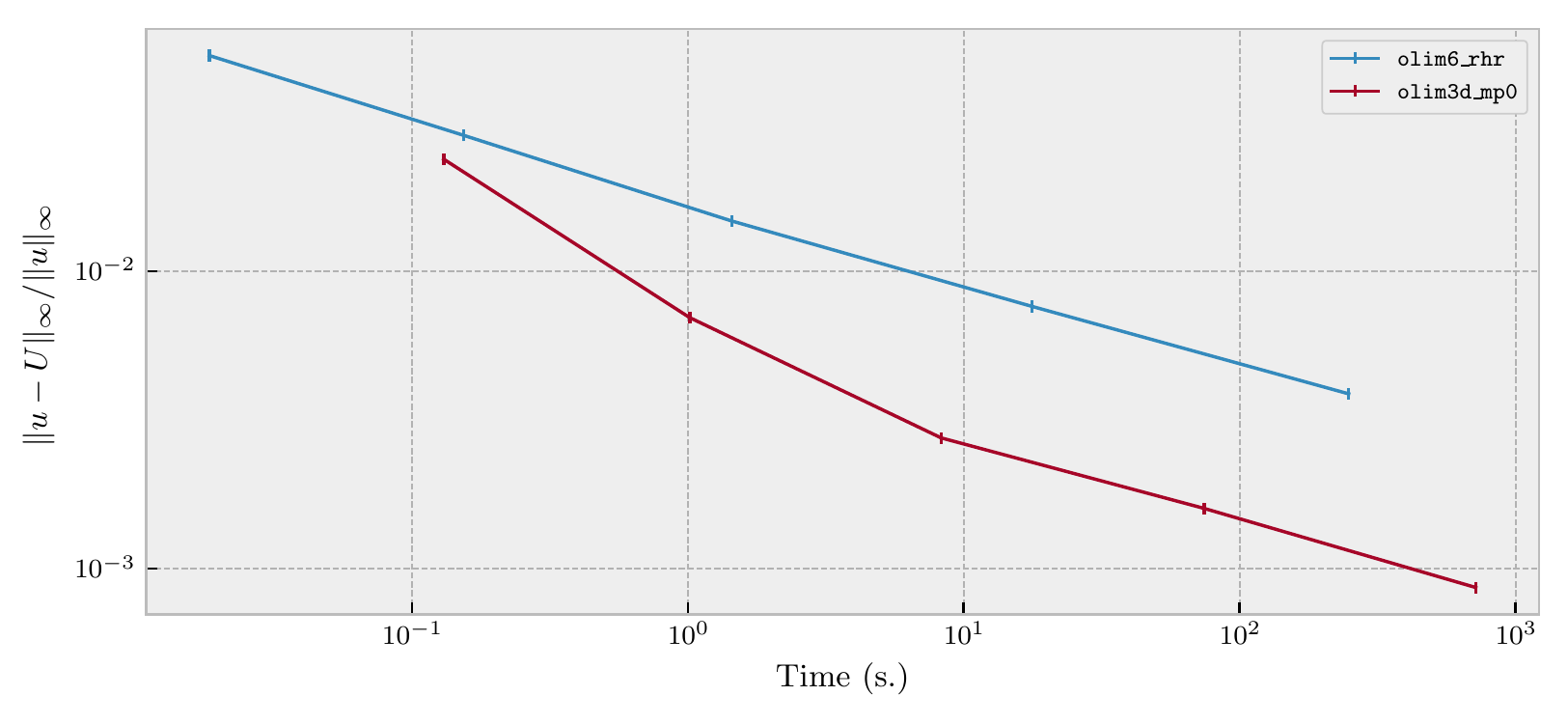}
  \caption{\emph{Comparing the relative $\ell_\infty$ errors of
      \texttt{olim3d\_mp0} and \texttt{olim6\_rhr}.} For the multiple
    point source problem in section\@ \ref{ssec:slotnick} with the
    domain $\Omega = [0, 1]^3$ discretized in each direction into
    $N = 2^p + 1$ (where $p = 5, \hdots, 9$), the total number of grid
    points is $N^3$.}\label{fig:intro}
\end{figure}

\subsection{Accessing our library} Our implementation is in
C++~\cite{stroustrup2013c++}. A link to a project website on GitHub
can be accessed from S.\ Potter's
website~\cite{sfp-umiacs-homepage}. The GitHub site includes
instructions for downloading and running basic
examples~\cite{libolim-github}. The code used to generate plots in
this paper is also available with
instructions~\cite{libolim-github-plotting}.

\section{Background}\label{sec:background}

In this section we provide a brief overview of the eikonal equation
and its numerical solution on a regular grid and sketch a generic
Dijkstra-like algorithm which we will refer to throughout the paper to
organize our results.

\subsection{The eikonal equation}

With $n \geq 2$, and given a domain $\Omega \in \R^n$, the eikonal
equation is:
\begin{equation}\label{eq:eikonal}
  \norm{\nabla u(x)} = s(x), \qquad x \in \Omega,
\end{equation}
where $\|\cdot\|$ denotes the $\ell_2$ norm unless otherwise stated,
and $s : \Omega \to (0, \infty)$ is a fixed, positive slowness
function, which forms part of the data. We solve for
$u : \Omega \to \R_+$. The boundary data is provided on a subset
$D \subset \Omega$ where $u$ has been fixed; i.e.,
$\left. u \right|_D = g$ for some $g : D \to \R_+$. As an example, if
$s \equiv 1$ and $g \equiv 0$, then the solution of eq.\@
\ref{eq:eikonal} is:
\begin{equation}
  \label{eq:distance-to-Omega}
  u(x) = d(x, D) = \min_{y \in D} \norm{x - y}.
\end{equation}
That is, $u$ is the distance to $D$ at each point in
$\Omega$.

To numerically solve eq.\@ \ref{eq:eikonal}, first let
$\calG = \{p_i\} \subseteq \Omega$ be the set or grid of nodes where
we would like to approximate the true solution $u$ with a numerical
solution $U : \calG \to \R_+$. Additionally, for each node
$p \in \calG$, define a set of neighbors,
$\neib(p) \subseteq \calG \backslash \set{p}$. Typically---for the
FMM, for instance---$\calG$ is taken to be a subset of a lattice in
$\R^n$ and $\neib(p)$ to be each node's $2n$ nearest neighbors. We
also define the set of boundary nodes, $\boundary \subseteq \calG$. It
may happen that the set $\boundary$ and $D$ do not coincide (e.g., $D$
could be a curve which does not intersect any points in $\calG$); to
reconcile this difference, the initial value of $U(p)$ for each
$p \in \boundary$ must take $g = \left. u \right|_D$ into account in
the best way possible. This problem has been approached in different
ways, and is not the focus of the present work~\cite{chopp2001some}.

Throughout, we make several simplifying assumptions.
\begin{itemize}
\item All boundary nodes coincide with grid points:
  $\boundary = D \subseteq \calG$.
\item The grid $\calG$ is a regular, uniform grid (a subset of a
  regular, uniform square lattice in 2D or cubic lattice in 3D). We
  denote grid nodes by $x \in \calG$.
\item When numerically computing a new value at a grid point
  $\hat{x} \in \calG$, we transform the neighborhood to the origin and
  scale the vertices so that they have integer values. The transformed
  update node is labeled $\hat{p}$. See section\@ \ref{ssec:quadrature}
  for a detailed explanation.
\end{itemize}

\subsection{Dijkstra-like algorithms}\label{ssec:dijkstra-like}

\begin{figure}[t]
  \centering
  \includegraphics[width=0.9\linewidth]{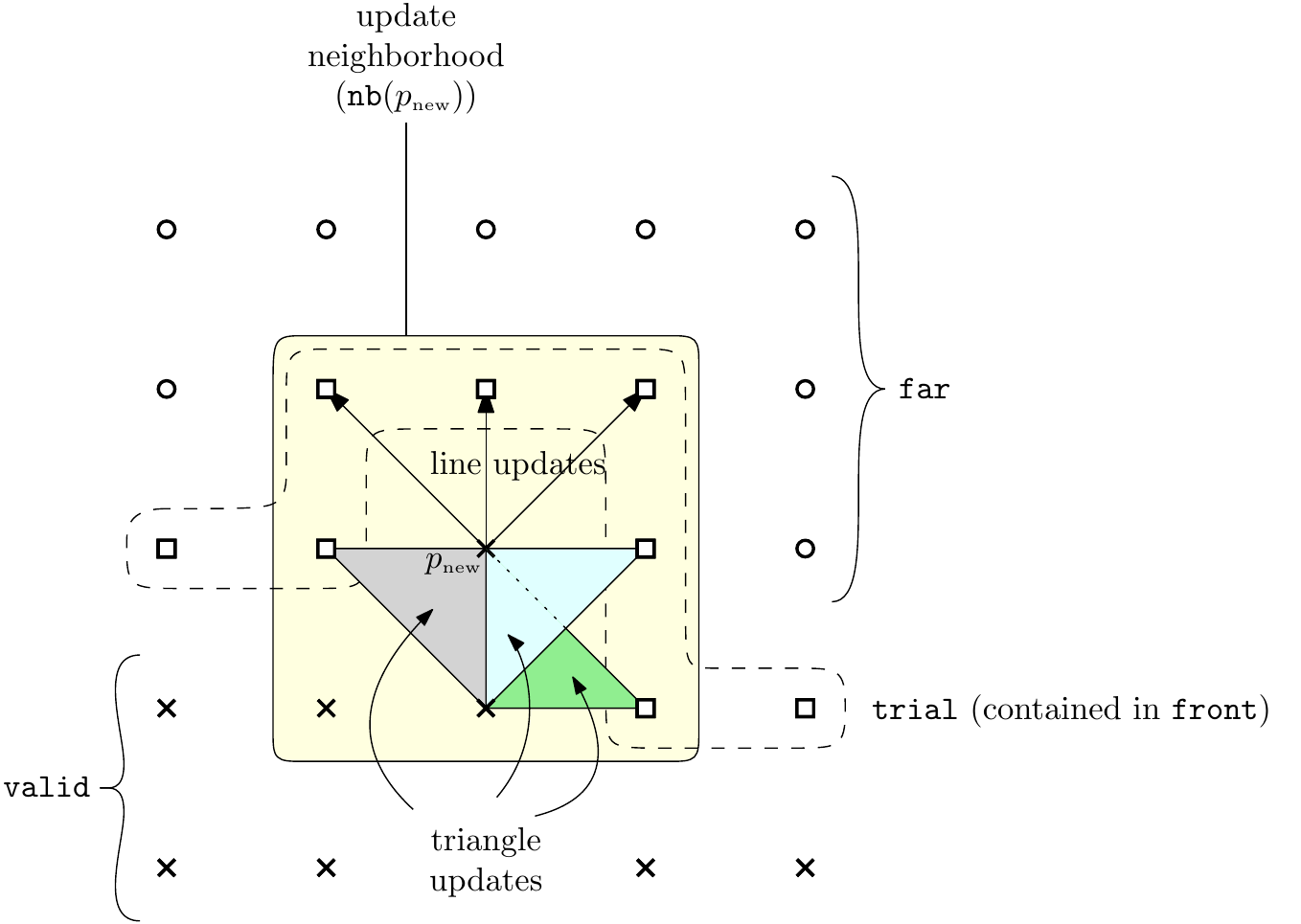}
  \caption{\emph{An overview of a Dijkstra-like algorithm for solving
      the eikonal equation (eq.\@ \ref{eq:eikonal}) in 2D.} See alg.\
    \ref{alg:dijkstra-like} for details. Nodes are labeled by state so
    that $\circ = \mathtt{far}$, $\square = \mathtt{trial}$, and
    $\times = \mathtt{valid}$. In this diagram, the node $\pnew$ has
    been removed from \texttt{front} and had its state set to
    \texttt{valid}. All \texttt{far} nodes in $\neib(\pnew)$ are set
    to \texttt{trial}, and then all \texttt{trial} nodes in
    $\neib(\pnew)$ are updated. The updates are depicted: there are
    three line updates and three triangles, since it is only necessary
    to perform updates that involve $\pnew$. The OLIM shown here is
    \texttt{olim8}. In 3D, there would also be tetrahedron updates.}
  \label{fig:overview}
\end{figure}

If we order nodes in $\mathcal{G}$ so that new solution values are
only computed using upwind nodes, the eikonal equation can be solved
directly; i.e., without the use of an iterative solver. This is done
using a continuous version of Dijkstra's algorithm for finding
shortest paths in a network. Other algorithms which solve similar
network flow problems can also be used, but have different complexity
guarantees~\cite{chacon2012fast}. In particular, Dijkstra's algorithm
is a type of label-setting method for finding shortest paths in a
network; there are also label-correcting
methods~\cite{bertsekas1998network}.

Using Dijkstra's algorithm to solve a ``continuous shortest path''
problem has been discovered in several contexts. The earliest such
development is a theoretical result in computational geometry due to
Mitchell, Mount, and Papadimitriou, who used this idea to compute
exact polyhedral shortest paths (``discrete geodesics'') on
triangulated surfaces~\cite{mitchell1987discrete}. This was followed
by Tsitsiklis who developed a first-order semi-Lagrangian method for
solving isotropic optimal control problems on a uniform
grid~\cite{tsitsiklis1995efficient}. Finally, the fast marching
method, which uses a first-order upwind finite difference scheme was
developed by Sethian to model isotropic front
propagation~\cite{sethian1996fast}. Many variations of these methods
have since been
developed~\cite{sethian2003ordered,kao2008legendre}. Our own
development resembles Tsitsisklis's, but extends it past its original
formulation. In particular, Tsitsiklis considers what we call
\texttt{olim8\_rhr} and \texttt{olim26\_rhr}, and does not treat the
3D case in depth. We generalize this approach, considering higher
accuracy quadrature rules (\texttt{mp0} and \texttt{mp1}) and
algorithms which make 3D solvers fast (our \emph{bottom-up} and
\emph{top-down} algorithms). We also note that Bornemann and Rasch
have investigated the local variational approach (\`{a} la
Tsitsiklis), and have present detailed theoretical results---their
approach is complementary to our own, but has a much different
emphasis~\cite{bornemann2006finite}.

To write down a generic Dijkstra-like algorithm, there are several
pieces of information which need to be kept track of. A data structure
called \texttt{front} tracks \texttt{trial} nodes while the solver
runs (typically an array-based heap). For each node $p$, apart from
the current value of $U(p)$, the most salient piece of information is
its state, written $p$\texttt{.state}
$\in \set{\texttt{valid}, \texttt{trial}, \texttt{far}}$. To fix
ideas, consider the following high-level Dijkstra-like algorithm:

\begin{algorithm}
  \caption{A generic Dijkstra-like algorithm for solving the eikonal
    equation.}\label{alg:dijkstra-like}
  \begin{enumerate}[nolistsep]
  \item For each $p \in \calG$, set $p$\texttt{.state} $\gets$
    \texttt{far} and $U(p) \gets \infty$.
  \item For each $p \in \boundary$, set $p$\texttt{.state} $\gets$
    \texttt{trial}, and set $U(p)$ to a user-defined value.
  \item While there are \texttt{trial} nodes left in $\calG$:
    \begin{enumerate}[nolistsep]
    \item Let $\pnew$ be the \texttt{trial} node in \texttt{front}
      with the smallest value $U(\pnew)$.\label{enum:get-node}
    \item Set $\pnew$\texttt{.state} $\gets$ \texttt{valid} and remove
      $\pnew$ from \texttt{front}.
    \item For each $\hat{p} \in \neib(\pnew)$, set
      $\hat{p}$\texttt{.state} $\gets$ \texttt{trial} if
      $\hat{p}$\texttt{.state} $=$ \texttt{far}.\label{enum:set-trial}
    \item For each $\hat{p} \in \neib(\pnew)$ such that
      $\hat{p}$\texttt{.state} $=$ \texttt{trial}, update
      $\hat{U} = U(\hat{p})$ and merge $\hat{p}$ into
      \texttt{front}.\label{enum:update-U}
    \end{enumerate}
  \end{enumerate}
\end{algorithm}

Specifying how item \ref{enum:update-U} is to be performed is the crux
of developing a Dijkstra-like algorithm and is left intentionally
vague here. This step involves indicating how nodes in
$\neib(\hat{p})$ are used to compute $\hat{U}$, and how they are
organized into the \texttt{front} data structure. The FMM uses an
upwind finite difference scheme where only \texttt{valid} nodes are
used to compute $\hat{U}$, and where nodes on the front are sorted
using an array-based heap implementing a priority
queue~\cite{sethian1996fast}. As an example, Tsitsiklis's algorithm
combines nodes in \texttt{valid} into sets whose convex hulls
approximate the surface of the expanding wavefront and then solves
local functional minimization problems. The method presented here is
more similar to Tsitsiklis's algorithm (see figure
\ref{fig:overview}). For specific details, a general reference should
be consulted~\cite{sethian1999level}.

In addition to item \ref{enum:update-U}, algorithm
\ref{alg:dijkstra-like} is generic in the following ways:
\begin{itemize}
\item As we mentioned before, there are different ways of initializing
  the boundary data $\boundary$ if only off-grid boundary data is
  provided~\cite{chopp2001some}.
\item How we keep track of the node with the smallest value is
  variable: most frequently, as in Dijkstra's algorithm, a heap
  storing pointers to the nodes is used, leading to $O(N^n \log N)$
  update operations overall, where $N^n$ is the number of nodes. In
  fact, there are $O(N^n)$ variations using Dial's algorithm (a
  bucketed version of Dijkstra's algorithm), but these have not been
  used as extensively as Dijkstra-like
  algorithms~\cite{tsitsiklis1995efficient,kim2001calo,yatziv2006n}.
\item The arrangement of the nodes into a grid or otherwise varies, as
  do the neighborhoods of each node. This affects the update
  procedure. A regular grid is simple to deal with, but Dijkstra-like
  methods have been extended to manifolds and unstructured meshes,
  where the situation is more
  involved~\cite{kimmel1998computing,sethian2000fast,bronstein2008numerical}.
\end{itemize}
Other problems can be solved using Dijkstra-like algorithms: the
static Hamilton-Jacobi equation, an anisotropic generalization of the
eikonal equation, can be solved using the ordered upwind
method~\cite{sethian2003ordered} or other recently introduced
methods~\cite{mirebeau2014efficient,mirebeau2014anisotropic}.  The
quasipotential of a nongradient stochastic differential equation can
also be computed using the ordered line integral method, although the
considerations are more
involved~\cite{dahiya2017ordered,dahiya2018ordered,yang2019computing}.

\subsection{Fast sweeping methods} Another approach to solving a
discretized version of eq.\@ \ref{eq:eikonal} is the fast sweeping
method~\cite{tsai2003fast,zhao2005fast}. Unlike Dijkstra-like methods,
which are direct solvers, the fast sweeping method is an iterative
solver using an upwind scheme and rotating sweep directions, which
obtains $O(N^n)$ complexity---however, the constant in this asymptotic
estimate depends heavily on how frequently the characteristics change
direction, and how complicated the geometry of the domain is. Fast
sweeping methods have been extended to Hamilton-Jacobi
equations~\cite{tsai2003fast,kao2008legendre}, and hybrid methods
combining the fast sweeping method with a Dijkstra-like method have
been introduced
recently~\cite{chacon2012fast,chacon2015parallel}. Additionally,
higher-order fast sweeping methods based on ENO and WENO schemes have
been developed~\cite{zhang2006high}. Luo and Zhao also provide a
detailed investigation of the convergence properties of fast sweeping
methods for static convex Hamilton-Jacobi equations, but with a focus
on the 2D case~\cite{luo2016convergence}.

\section{Ordered line integral methods for the eikonal equation}\label{sec:olim}

The fast marching method~\cite{sethian1996fast} solves a discretized
eikonal equation (eq.\@ \ref{eq:eikonal}) in an upwind
fashion. Throughout, we distinguish between the exact solution $u$ and
the numerical solution $U$, where $\hat{U}$ will always denote the
current value to be computed; likewise, any quantity with a hat (\^{})
will denote a quantity evaluated at the node being updated. The
ordered line integral method locally and approximately minimizes the
minimum action integral of eq.\@ \ref{eq:eikonal}:
\begin{equation}
  \label{eq:action-functional}
  \hat{u} = \min_{\alpha} \set{u_0 + \int_\alpha s(x) dl},
\end{equation}
where $\alpha$ is a ray parametrized by arc length, $\hat{x}$ is a
target point, $\hat{u} = u(\hat{x})$, and $u_0 = u(\alpha(0))$ (see
section\@ \ref{sec:minimum-action-integral} for a derivation of eq.\
\ref{eq:action-functional}). By constrast, Lagrangian methods (i.e.,
raytracing methods) trace a bundle of rays from a common locus by
integrating Hamilton's equations for the eikonal equation for
different initial conditions.

In this section, we describe how we discretize and minimize an
approximation of eq.\@ \ref{eq:action-functional}. As we mentioned in
section\@ \ref{ssec:dijkstra-like}, to compute $\hat{U} = U(\hat{p})$
in item \ref{enum:update-U} of algorithm \ref{alg:dijkstra-like}, we
need to approximately minimize several instances of an approximation
to eq.\@ \ref{eq:action-functional}; details of this procedure are
discussed in section\@ \ref{sec:implementation}. In this section, we
focus on a single instance of the discretized version of eq.\@
\ref{eq:action-functional}. We present our notation, derive prelimary
results, and describe the quadrature rules \texttt{mp0}, \texttt{mp1},
and \texttt{rhr}. We also show how the functional minimization problem
can be solved exactly using a QR decomposition for the \texttt{rhr}
and \texttt{mp0} rules. Finally, we present theoretical results
justifying our approach.

\begin{figure}
  \centering
  \includegraphics[width=.39864\textwidth]{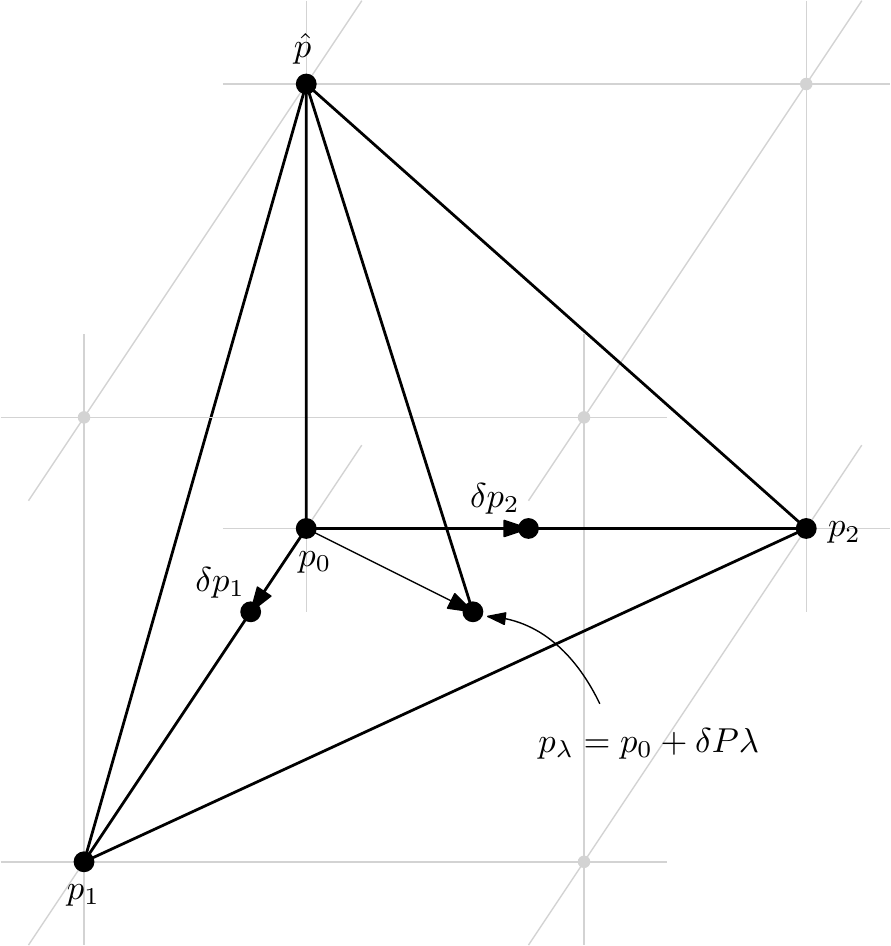}
  \hspace{3em}
  \includegraphics[width=.43912\textwidth]{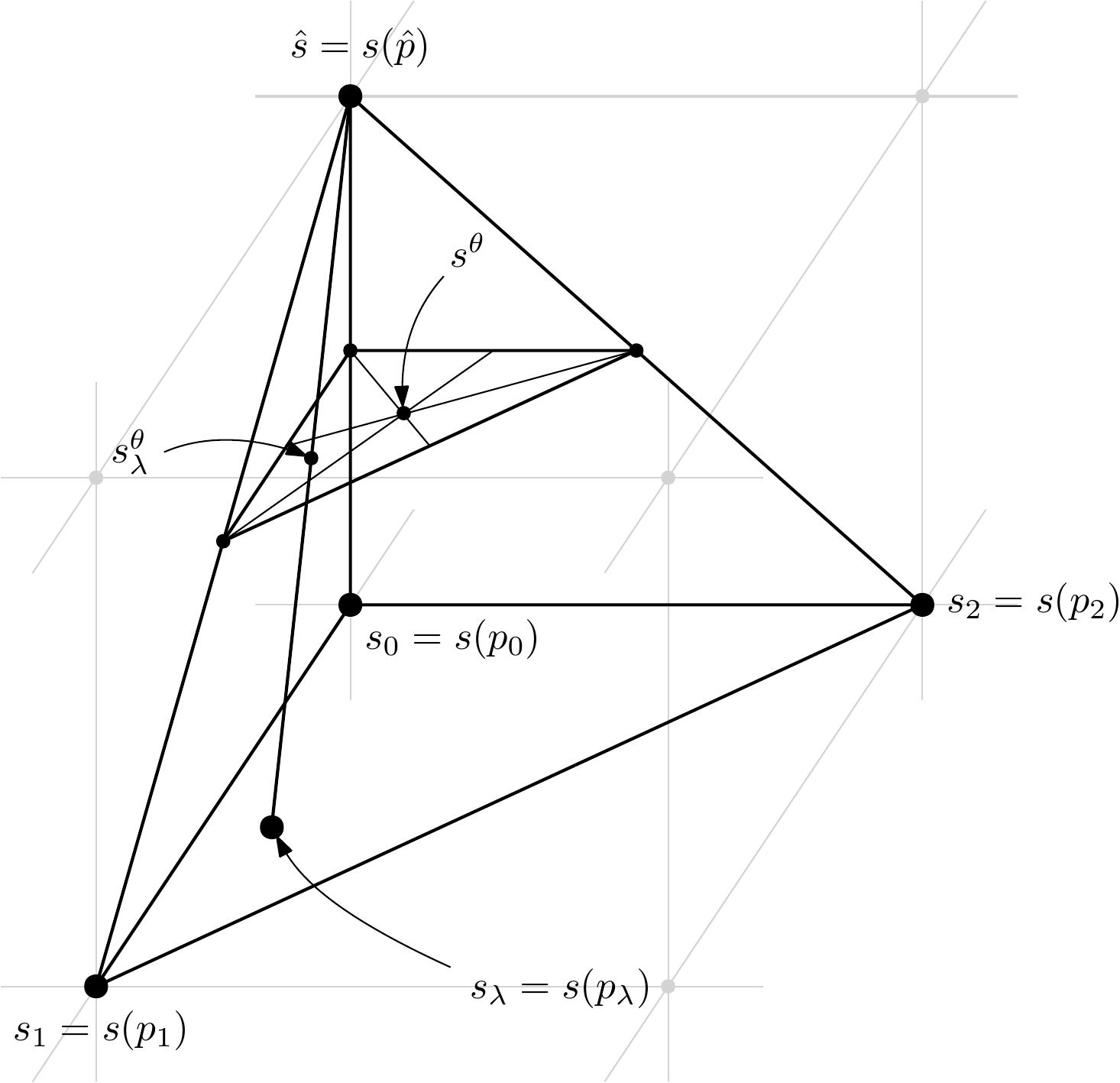}
  \caption{\emph{Overview of a tetrahedron update, showing the
      notation in section\@ \ref{sec:olim}}. Left: a point being
    updated, $\hat{p}$, which is identified with the origin, and three
    neighboring points $p_0, p_1$, and $p_2$ that are assumed to be
    \texttt{valid}. The grid $\mathcal{G}$, which contains other
    points in the discretized domain, is sketched in light grey. The
    domain of the minimization problem eq.\@
    \ref{eq:constrained-minimization} is the convex hull of
    $p_0, p_1$, and $p_2$. The path minimizing eq.\@
    \ref{eq:action-functional} is assumed to be the line segment
    connecting $\hat{p}$ and $p_\lambda$. Not pictured is the newly
    \texttt{valid} point $\pnew$, although it is assumed that $\pnew$
    equals one of $p_0, p_1$, or $p_2$. Right: the same update
    tetrahedron, but this time with quantities related to the slowness
    function depicted.}\label{fig:simplex-diagrams}
\end{figure}

\subsection{Approximating the action functional}\label{ssec:quadrature}

In this section we describe how we approximate eq.\@
\ref{eq:action-functional} and reformulate it as a constrained
optimization problem which we solve to update the \texttt{trial} nodes
surrounding a node $\pnew$ which has just become \texttt{valid}.

First, we assume that $\Omega \subseteq \mathbb{R}^n$, where
$n = 2, 3$. The methods presented here work for general $n$. We refer
to each update as a ``simplex update'', since for a dimension $n$, we
need to consider updates of dimensions $d$ where $0 \leq d < n$. For
$n = 3$, we have line updates ($d = 0$), triangle updates ($d = 1$),
and tetrahedron updates ($d = 2$). So, $d$ refers to the dimension of
the base of each simplex, which is the dimension of the domain of the
optimization problem that we will formulate.

We assume that each update simplex is nondegenerate and that the
convex hull of the update point $\hat{p}$ and $d+1$ points
$p_0, \hdots, p_{d} \in \neib(\hat{p})$. Since we assume that our grid
$\mathcal{G}$ is uniform and rectilinear, we scale and translate
$\mathcal{G}$ so that $\hat{p} = 0$ and $\|p_i\|_\infty = 1$ for
$i = 0, \hdots, d$. \emph{Throughout the rest of the paper, we always
  shift the node $\hat{p}$ to the origin, as this simplifies our
  calculations.}

\paragraph{Approximating the integration path with a straight line
  segment.} To approximately minimize eq.\@ \ref{eq:action-functional}
we assume that the minimizing path is a straight line segment
connecting $\hat{p}$ and a point in the convex hull of
$\{p_0, \hdots, p_d\}$, and numerically approximate the action over
this integral path using quadrature. We discuss each part of this
approximation in turn. First, some notation.

We parametrize the ``base of the update simplex'' (the convex hull of
$p_0, \hdots, p_d$) over the set:
\begin{equation}
  \label{eq:lambda}
  \Delta^d = \set{\lambda_i \geq 0 \mbox{ for } i = 1, \hdots, d \mbox{ and } \sum_{i=1}^d \lambda_i \leq 1}.
\end{equation}
If we let $\lambda_0 = 1 - \sum_{i=1}^d \lambda_i$, then
$(\lambda_0, \hdots, \lambda_d)$ is a vector of convex coefficients.
We let $\delp_i = p_i - p_0$ and define:
\begin{equation}
  \delP = \begin{bmatrix}
    \delp_1 & \cdots & \delp_d
  \end{bmatrix} \in \mathbb{R}^{n \times d}.
\end{equation}
We write a point in the base of the update simplex as:
\begin{equation}
  p_\lambda = p_0 + \sum_{i=1}^d (p_i - p_0) \lambda_i = p_0 + \sum_{i=1}^d \delp_i \lambda_i = p_0 + \delP \lambda.
\end{equation}

We will use the ``$\delta$'' notation for differences and $\lambda$ as
a subscript to denote convex combinations in other contexts, as
well. E.g., $\delU_i = U_i - U_0 = U(x_i) - U(x_0)$ and
$U_\lambda = U_0 + \delU^\top \lambda$. Likewise,
$\dels_i = s_i - s_0 = s(x_i) - s(x_0)$. By an abuse of notation, we
will think of, e.g., $s_i$ and $s(x_i)$ in the context of an update as
``the same'', preferring the notation $s_i$.

\paragraph{Quadrature rules.} We consider a righthand rule
(\texttt{rhr}), a simplified midpoint rule (\texttt{mp0}), and a
midpoint rule (\texttt{mp1}). Recall that $\hat{s} = s(\hat{x})$. The
cost functions being minimized in eq.\@ \ref{eq:action-functional}
are:
\begin{align}\label{eq:quadrature-rules}
  \Frhr(\lambda) &= U_\lambda + \hat{s} h \|p_\lambda\|, \\
  \Fmpzero(\lambda) &= U_\lambda + \parens{\frac{\hat{s} + \tfrac{1}{d+1} \sum_{i=0}^d s_i}{2}} h \|p_\lambda\|, \\
  \Fmpone(\lambda) &= U_\lambda + \parens{\frac{\hat{s} + s_\lambda}{2}} h \|p_\lambda\|.
\end{align}
The difference in the quadrature rules, of course, lies in how we
incorporate the slowness $s$. For $\Frhr$, we evaluate $s$ at the
righthand side of the integral, yielding $\hat{s}$. For $\Fmpone$, we
evaluate $s$ at the midpoint of the integral, approximating $s$
linearly with the convex combination $s_\lambda$ on the base of the
simplex. Finally, for $\Fmpzero$, we approximate $s_\lambda$ itself
with the arithmetic mean of the $s_i$'s. It will turn out that
$\Fmpzero$ will lead to an inconsistent numerical scheme unless extra
care is taken, which is discussed in the rest of the work.

Note that in the above we use $s_\lambda$ and not $s(p_\lambda)$
because we do not want to assume that we have access to a continuous
functional form for $s$; in most cases, we assume that $s$ will be
provided as gridded data, and that interpolation will be used to
approximate $s$ off-grid.

\paragraph{More general quadrature rules.} The quadrature rules above
are specializations of the following more general quadrature
rules. With $\theta$ such that $0 \leq \theta \leq 1$, we define:
\begin{align}
  F_0(\lambda) &= U_\lambda + \squareb{(1 - \theta)\hat{s} + \frac{\theta}{d+1} \sum_{i=0}^d s_i} h \|p_\lambda\|, \\
  F_1(\lambda) &= U_\lambda + \bigg[ (1 - \theta)\hat{s} + \theta s_\lambda \bigg] h \|p_\lambda\|.
\end{align}
Then, $\Frhr = F_0 = F_1$ with $\theta = 0$, $\Fmpzero = F_0$ with
$\theta = \tfrac{1}{2}$, and $\Fmpone = F_1$ with
$\theta = \tfrac{1}{2}$. To simplify notation in our proofs, we also
define:
\begin{equation}
  s^\theta = {(1 - \theta)} + \frac{\theta}{d + 1} \sum_{i=0}^d s_i, \qquad s^\theta_\lambda = (1 - \theta) \hat{s} + \theta s_\lambda.
\end{equation}
Then, the $\theta$-rules can be written more compactly as:
\begin{equation}
  F_0(\lambda) = U_\lambda + s^\theta h \|p_\lambda\|, \qquad F_1(\lambda) = U_\lambda + s^\theta_\lambda h \|p_\lambda\|.
\end{equation}

We introduce this more general
``$\theta$-rule'' for two reasons:
\begin{itemize}
\item This is a natural geometric generalization, and we wish
  to contextualize our results properly.
\item Our proofs are written in terms of the $\theta$-rules, which
  allows us to provide proofs for $F_0$ and proofs for $F_1$, instead
  of separate proofs for each of $\Frhr$, $\Fmpzero$, and
  $\Fmpone$. As a bonus, our proofs apply to other $\theta$-rules not
  considered here. For instance, schemes using $\theta = 1$ or
  $\theta$ chosen adaptively may be of interest.
\end{itemize}

\subsection{The minimization problem}\label{ssec:minimization-problem}

With $F_0$ and $F_1$ so defined, the minimization problem which
approximates eq.\@ \ref{eq:action-functional} is:
\begin{equation}
  \label{eq:constrained-minimization}
  \hat{U} = \min_{\lambda \in \Delta^d} F(\lambda),
\end{equation}
where $F = \Frhr, \Fmpzero$, or $\Fmpone$. This is a nonlinear,
constrained optimization problem with linear inequality constraints
and no equality constraints. We require the gradient and Hessian of
$F_0$ and $F_1$ for our algorithms and analysis. These are easy to
compute, but we have found a particular form for them to be convenient
for both implementation and analysis. The proofs of all propositions
and lemmas in this section can be found in section\@
\ref{sec:minimization-proofs}.

In what follows, we will use the notation:
\begin{equation}
  \proj_p = \frac{p p^\top}{p^\top p}, \qquad \proj_p^\perp = I - \proj_p = I - \frac{p p^\top}{p^\top p}
\end{equation}
for orthogonal projection matrices. Here, $\proj_p$ projects
orthogonally onto $\operatorname{span}(p)$, and $\proj_p^\perp$ onto
its orthogonal complement.

\begin{proposition}\label{prop:F0-grad-and-Hess}
  The gradient and Hessian of $F_0(\lambda)$ are given by:
  \begin{align}
    \nabla F_0(\lambda) &= \delU + s^{\theta} h \delP^\top \nu_\lambda,\label{eq:F0-grad} \\
    \nabla^2 F_0(\lambda) &= \frac{s^{\theta} h}{\|p_\lambda\|} \delP^\top \proj^\perp_{p_\lambda} \delP,
  \end{align}
  where $\nu_\lambda = p_\lambda/\|p_\lambda\|$ is the unit vector in the
  direction of $p_\lambda$.
\end{proposition}

\begin{proposition}\label{prop:F1-grad-and-Hess}
  The gradient and Hessian of $F_1(\lambda)$ satisfy:
  \begin{align}
    \nabla F_1(\lambda) &= \delU + \theta h \|p_\lambda\| \dels + s^{\theta}_\lambda h \delP^\top \nu_\lambda, \\
    \nabla^2 F_1(\lambda) &= \set{\delP^\top \nu_\lambda, \theta h \dels} + \frac{s^\theta_\lambda h}{\|p_\lambda\|} \delP^\top \proj^\perp_{p_\lambda} \delP, \label{eq:hess-F1}
  \end{align}
  where $\set{a, b} = ab^\top + ba^\top$ is the anticommutator of two
  vectors.
\end{proposition}

Our task is to minimize $F_0$ and $F_1$ over the convex set
$\Delta^n$; so, we need to determine whether $F_0$ and $F_1$ are
convex functions. The next two lemmas address this point.

\begin{lemma}\label{lemma:dPt-cprojp-dP-pd}
  Let $p_0, \hdots, p_d$ form a nondegenerate simplex (i.e.,
  $p_0, \hdots, p_d$ are linearly independent) together with $\hat{p}$
  and assume that $s$ is positive. Then, $\nabla^2 F_0$ is positive
  definite and $F_0$ is strictly convex.
\end{lemma}

For $F_1$, we can only obtain convexity (let alone strict convexity)
for $h$ sufficiently small. For large enough $h$, we will encounter
nonconvex updates. To obtain convexity, we need to stipulate that the
slowness function $s$ is Lipschitz continuous on $\Omega$ with a
Lipschitz constant that is independent of $h$. In practice, we have
not found this to be a particularly stringent restriction.

\begin{lemma}\label{lemma:F-strictly-convex}
  In the setting of lemma \ref{lemma:dPt-cprojp-dP-pd}, additionally
  assume that $s$ is Lipschitz continuous with Lipschitz constant
  $K \leq C$ on $\Omega$, for some constant $C > 0$ independent of
  $h$. Then, $\nabla^2 F_1$ is positive definite (hence, $F_1$ is
  strictly convex) for $h$ small enough.
\end{lemma}

We have found that all \texttt{mp1} updates become strictly convex
problems rapidly as $h \to 0$. The reason for this is discussed at the
end of section\@ \ref{ssec:causality}.

\subsection{Validation of \texttt{mp0}}\label{ssec:validation}

\begin{figure}
  \centering
  \includegraphics[width=0.7\textwidth]{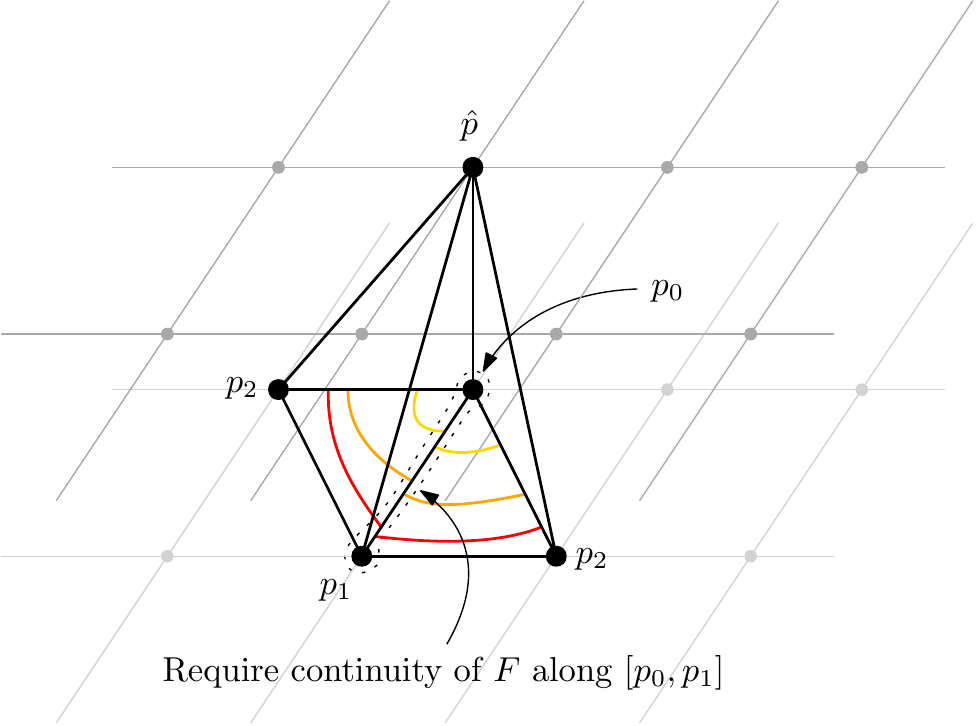}
  \caption{\emph{A problem with \texttt{mp0} for which we provide a
      simple solution.} The cost function $F$ must be continuous
    across the boundaries of adjacent update simplexes, otherwise an
    inconsistent solver can come about. Here, the colorered level sets
    depict the discontinuity of $\Fmpzero$ across the bases of the
    update simplexes. In this case, two adjacent update simplexes
    share a common boundary on their base, shown here as the line
    segment $[p_0, p_1]$. Two layers of surrounding grid points from
    $\mathcal{G}$ are shown. The point $\hat{p}$ is in the top layer
    and the points $p_0$, $p_1$, and $p_2$ are in the bottom
    layer.}\label{fig:continuity}
\end{figure}

If we use $\Fmpzero$ directly, then we run into a situation where the
cost function $\Fmpzero$ is not continuous between the bases of
adjacent update simplices (see fig.\@ \ref{fig:continuity}). We
require $F$ to be continuous across simplex boundaries to avoid an
inconsistent or divergent solver. Now, if we first use $\Fmpzero$ to
compute the minimizer $\lambda_0^*$ of eq.\@
\ref{eq:constrained-minimization} with $F = \Fmpzero$, and then set
$\hat{U} = \Fmpone(\lambda_0^*)$, we will recover continuity, and
indeed, as we will show, the scheme is convergent. The motivation this
is that eq.\@ \ref{eq:constrained-minimization} can be solved exactly
for the $\theta$-rule $F_0$ using a QR decomposition instead of an
iterative solver, making it very cheap. In the next section, we will
show how this can be done.

Let $\lambda_0^*$ and $\lambda_1^*$, denote the optima of eq.\@
\ref{eq:constrained-minimization}, where $F = F_0$ and $F = F_1$ (the
general $\theta$-rules), respectively. We could imagine using Newton's
method to minimize $F_1$, starting from $\lambda_0^*$ (to be clear,
this is not the approach we will ultimately take numerically). This
would allow us to use the convergence theory of Newton's method to
bound the distance between $\lambda_0^*$ and $\lambda_1^*$, thereby
bounding the error incurred by using \texttt{mp0} instead of
\texttt{mp1} to find the minimizing argument of eq.\@
\ref{eq:constrained-minimization}. We follow this idea now.

\begin{theorem}\label{thm:mp0-newton}
  Using lemma \ref{lemma:F-strictly-convex}, let $h$ be sufficiently
  small so that $F_1$ is strictly convex. Then, the error
  $\dellam^* = \lambda_1^* - \lambda_0^*$ satisfies
  $\norm{\dellam^*} = O(h)$. Further, if we let
  $\lambda_0 = \lambda_0^*$ in the following Newton iteration:
  \begin{equation}
    \label{eq:lam0-iter-to-lam1}
    \lambda_{k+1} \gets \lambda_k - \nabla^2 F_1(\lambda_k)^{-1} \nabla F_1(\lambda_k), \qquad k = 0, 1, \hdots,
  \end{equation}
  then this iteration is well-defined, and converges quadratically to
  $\lambda_1^*$. This immediately implies that the error incurred by
  \texttt{mp0} is $O(h^3)$ per update compared to \texttt{mp1}; i.e.:
  \begin{equation}
    \label{eq:mp0-error}
    \abs{F_1(\lambda_1^*) - F_1(\lambda_0^*)} = O(h^3).
  \end{equation}
\end{theorem}

\begin{proof}
  The proof of theorem \ref{thm:mp0-newton} is detailed in appendix
  \ref{app:validation-proofs}.
\end{proof}

We will show in the next section how $\lambda_0^*$ can be computed
directly using a QR decomposition and without using an iterative
solver.

We can provide some intuition for why this bound is satisfactory. If
we assume that our domain is spanned along a diameter by $O(N)$ nodes,
and that $h \sim N^{-1}$, then we can anticipate $O(N)$ downwind
updates, starting from $\boundary$ and extending to the boundary of
$\calG$ in any direction. Accumulating the error over these nodes, we
can expect the maximum pointwise error between a solution to eq.\@
\ref{eq:eikonal} computed by using \texttt{mp0} and \texttt{mp1} to be
$O(h^2)$, which is dominated by the $O(h)$ discretization error coming
from the linear convergence of the method itself. Hence, using
\texttt{mp0} instead of \texttt{mp1} only to find the parameter
$\lambda$, and then evaluate $\hat{U}$ using $F_1$, should introduce
no significant extra error.


\subsection{Exact solution for \texttt{rhr} and \texttt{mp0} using a
  QR decomposition}\label{ssec:exact-soln}

Since $F_0$ is strictly convex, $\nabla F_0(\lambda) = 0$ is
sufficient for the optimality of $\lambda$, ignoring the constraint
$\lambda \in \Delta^d$. The unconstrained system of nonlinear
equations defined by $\nabla F_0(\lambda) = 0$ can be solved exactly
without an iterative solver. We can compute the solution using the
reduced QR decomposition of $\delP$ and by considering the problem's
geometry (see also figure \ref{fig:f0-exact}, in the appendix). This
is captured in the following theorem. We will discuss how to use this
theorem efficiently in a solver in section\@
\ref{ssec:algorithms-and-skipping}.

\begin{theorem}\label{thm:f0-exact}
  Let $\delP = QR$ be the reduced QR decomposition of $\delP$; i.e.,
  where
  $Q \in \mathbb{R}^{n \times d}, R \in \mathbb{R}^{d \times d},
  Q^\top Q = I_d$, and with $R$ upper triangular. For $s^\theta, h$,
  and $U$ fixed, if
  $\lambda^* = \argmin_{\lambda \in \mathbb{R}^n}
  F_0(\lambda)$, then:
  \begin{align}
    \|p_{\lambda^*}\| &= \sqrt{\frac{p_0^\top {(I - QQ^\top)} p_0}{1 - \norm{R^{-\top} \frac{\delU}{s^\theta h}}^2}},\label{eq:l-star-expression} \\
    \lambda^* &= -R^{-1} \parens{Q^\top p_0 + \|p_{\lambda^*}\| R^{-\top} \frac{\delU}{s^\theta h}},\label{eq:f0-exact-lambda} \\
    \hat{U} &= U_0 + \frac{s^\theta h}{\|p_{\lambda^*}\|} p_0^\top p_{\lambda^*}.\label{eq:f0-exact}
  \end{align}
\end{theorem}

\begin{proof}
  See section\@ \ref{sec:exact-soln-proofs}.
\end{proof}

\subsection{Equivalence of the upwind finite difference scheme and
  $F_0$}\label{ssec:equivalence}

If we linearly approximate $U$ near $\hat{p}$, then for
$i = 0, \hdots, n - 1$, we find that $\hat{U}$ satisfies:
\begin{equation}
  \label{eq:finite-differences}
  U_i - \hat{U} = \nabla \hat{U}^\top p_i.
\end{equation}
This finite difference approximation to eq.\@ \ref{eq:eikonal} can be
solved exactly and is a known generalization of the upwind finite
difference scheme used in the fast marching method on an unstructured
mesh~\cite{kimmel1998computing,sethian2000fast}. Computing $\hat{U}$
using this approximation is equivalent to solving:
\begin{equation}
  \hat{U} = \min_{\lambda \in \Delta^n} F_0(\lambda)
\end{equation}
in a sense made precise by the following theorem. As we have pointed
out, this theorem is not new, but we present it here for the sake of
continuity and because it dovetails with our other theorems, providing
context.

\begin{theorem}[Equivalence of upwind finite difference scheme and $F_0$]\label{thm:equivalence}
  Let $\hat{U}$ by the solution of eq.\@
  \ref{eq:finite-differences} and let
  $\hat{U}' = \min_{\lambda \in \R^n} F_0(\lambda)$. Then, $\hat{U}$
  exists if and only if $\|R^{-\top} \delU\| \leq s^\theta h$, and can
  be computed from:
  \begin{equation}
    \label{eq:U-finite-diff}
    \hat{U} = U_i - p_i^\top Q R^{-\top} \delU + \norm{\pmin} \sqrt{{(s^\theta h)}^2 - \norm{R^{-\top} \delU}^2},
  \end{equation}
  where $\pmin = (I - QQ^\top) p_i$ (for all
  $i = 0, \hdots, n - 1$---see figure \ref{fig:f0-exact} in section\@
  \ref{sec:exact-soln-proofs}). Additionally, the following hold:
  \begin{enumerate}
  \item The finite difference solution and line integral solution
    coincide: i.e., $\hat{U} = \hat{U}'$ can be computed from:
    \begin{equation}
      \label{eq:U-from-Ui-exact}
      \hat{U} = U_i + s^\theta h p_i^\top \nu_{\lambda^*},
    \end{equation}
    where $\lambda^* = \argmin_{\lambda \in \R^n} F_0(\lambda)$ and
    $\nu_{\lambda^*} = p_{\lambda^*}/\|p_{\lambda^*}\|$.
  \item The characteristics found by solving the finite difference
    problem and minimizing $F_0$ coincide and are given by
    $[p_{\lambda^*}, \hat{p}] = [p_{\lambda^*}, 0]$.
  \item The approximated characteristic passes through
    $\conv(\set{p_0, \hdots, p_{n-1}})$ if and only if
    $\lambda^* \in \Delta^n$.
  \end{enumerate}
\end{theorem}

\begin{proof}
  See section\@ \ref{sec:equivalence-proofs}.
\end{proof}

\subsection{Causality}\label{ssec:causality} Dijkstra-like
methods are based on the idea of monotone causality, similar to
Dijkstra's method itself. To compute shortest paths in a network,
Dijkstra's method uses dynamic programming to compute globally optimal
shortest paths using local information~\cite{dijkstra1959note}. In
this way, the distance to each downwind vertex must be greater than
its upwind neighboring vertices. To ensure convergence to the correct
viscosity solution, our scheme must be consistent and
monotone~\cite{crandall1983viscosity}. Our OLIMs using the
\texttt{rhr} quadrature rule inherit the consistency and causality of
the finite difference methods which they are equivalent to if they use
the same 4 (in 2D) or 6 (in 3D) point neighborhoods. Since we consider
many different update neighborhoods involving distinct simplexes, we
provide a simple way of checking whether each simplex is causal.

The causality of an update depends on the underlying simplex and the
problem data. In particular, an update is causal for $F_i$ if:
\begin{equation}
  \hat{U} = F_i(\lambda_i^*) \geq \max_i U_i.
\end{equation}
It is enough to determine whether or not each type of update simplex
admits only causal updates, which relates to whether the simplex is
acute.

We also consider something we refer to here as the ``update gap'': the
difference $\hat{U} - \max_i U_i$. As discussed in Tsitsiklis's
original paper~\cite{tsitsiklis1995efficient}, an alternative to
Dijkstra's algorithm is Dial's algorithm---a bucketed version of
Dijkstra's algorithm which runs in $O(N^n)$ time, where the constant
depends on the bucket size~\cite{dial1969algorithm,kim2001calo}. In
this case, the size of the buckets is determined by the update gap. It
is unclear whether there is any real advantage of a Dial-like solver
(see~\cite{jeong2008fast} for a discussion), although a new numerical
study suggests that there may indeed be some advantage in using a
Dial-like solver~\cite{kim2001calo,gomez2019fast}. Despite this, the
update gap is of fundamental importance and limits the number of nodes
that can be processed in parallel without violating causality.

\begin{theorem}\label{thm:causality}
  For $\nu_i = p_i/\|p_i\|$, an update simplex is causal for $F_0$ if
  and only if $\nu_i^\top \nu_j \geq 0$ for all $i$ and $j$ such that
  $0 \leq i < n$ and $0 \leq j < n$. The update gap is given
  explicitly by:
  \begin{equation}
    \hat{U} - \max_i U_i = s^\theta h \min_{i, j} \frac{\nu_i^\top \nu_j}{\norm{p_i}}.
  \end{equation}
  If we assume that $s$ is Lipschitz continuous, then for $h$ small
  enough, the simplex is also causal for $F_1$, and the term in $F_1$
  which prevents an update from being causal decays with order
  $O(h^2)$.
\end{theorem}

\begin{proof}
  See section\@ \ref{sec:causality-proofs}.
\end{proof}

In section\@ \ref{ssec:update-gaps} we present a table of update gaps
for all update simplices considered in this paper.

The fact that our methods are causal for all practical problem sizes
follows from the fact that the term preventing causality decays
rapidly---see eq.\@ \ref{eq:F1-in-terms-of-F0}. This can be seen
easily by rewriting $F_1(\lambda)$ as $F_0(\lambda)$ plus a small
perturbation (which is $O(h^2)$) and using the Lipschitz continuity of
$s$.

\subsection{Local factoring}

Near rarefaction fans, for example if $D$ is a point source or the
domain contains obstacles with corners, the rate of convergence of the
eikonal equation is diminished. For the eikonal equation with point
source data and constant slowness, this degrades the rate of
convergence to $O(h \log
h^{-1})$~\cite{qi2018corner,zhao2005fast}. Fast sweeping methods for
the factored eikonal equations have been
developed~\cite{fomel2009fast,luo2012fast}; likewise, fast marching
methods have been developed, and have also been used in the context of
travel time tomography~\cite{qi2018corner,treister2016fast}.

In this section, we show how the ordered line integral method can be
easily adapted to additive factoring, and provide numerical tests that
show that it recovers the expected linear rate of convergence for
factored point source problems. Our focus is locally factored point
sources, but this approach can be applied to the globally factored
equation and other types of rarefaction fans occuring at corners or
discontinuities~\cite{qi2018corner}.

\begin{figure}
  \centering
  \includegraphics[width=\linewidth]{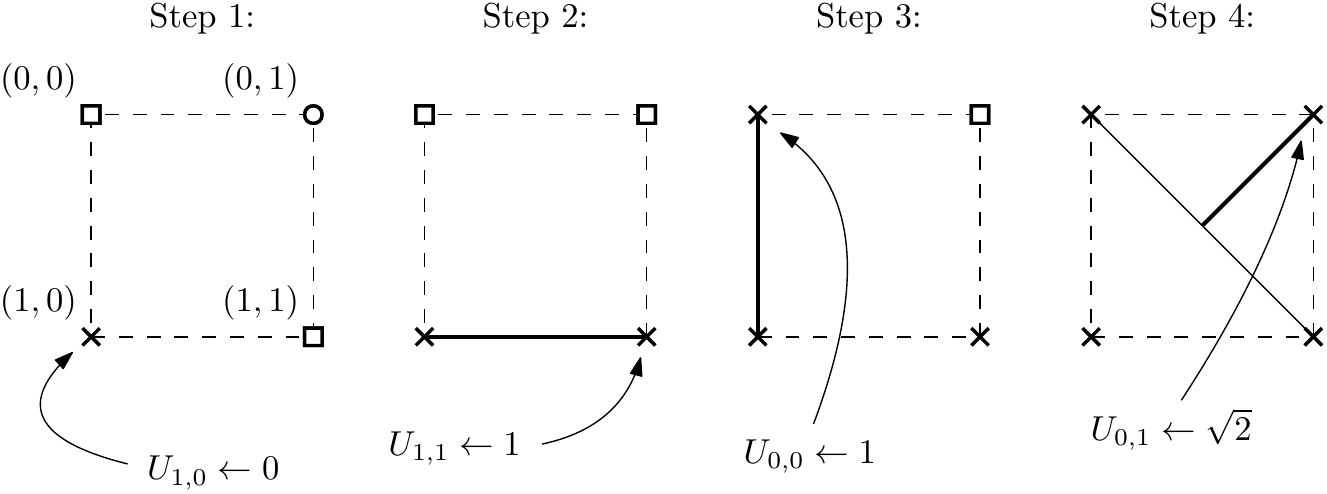}
  \caption{\emph{An example of running \texttt{olim4\_rhr} with local
      factoring for three steps.} Note that \texttt{olim4\_rhr} is
    equivalent to the standard 2D fast marching method. We assume
    $s \equiv 1$. Initially, $p^\circ = p_{1, 0}$ is the only node in
    \texttt{bd} and is factored. The first two steps proceed exactly
    the same as for the unfactored method, since the steps between the
    source nodes and the updated nodes lie along characteristics. In
    the third step, the unfactored method would set
    $U_{0,1} \gets 1 + \nicefrac{\sqrt{2}}{2}$. However, for the
    factored solver, we find that
    $U_{0,1} \gets \nicefrac{\sqrt{2}}{2} + \nicefrac{\sqrt{2}}{2} =
    \sqrt{2}$.}\label{fig:factoring-example}
\end{figure}

Let $x^\circ \in \Omega$ be the location of a point source so that
$\boundary = \set{x^\circ}$, define $p^\circ$ to be the image of
$x^\circ$ under the same transformation that takes $\hat{x}$ to
$\hat{p}$, and let $s^\circ = s(x^\circ)$. The additive factorization
of $U$ around $x^\circ$ is~\cite{luo2012fast,qi2018corner}:
\begin{align}
  U(x) = T(x) + \tau(x), \qquad \text{where} \qquad T(x) = s^\circ \norm{x - x^\circ},
\end{align}
i.e.\ $u_\lambda = T_\lambda + \tau_\lambda$ where
$T_\lambda = s^\circ h \|p_\lambda - p^\circ\|$. Our original definition of
$F_i$ was such that $\hat{U} = F_i(\lambda^*)$. We will
define $G_i$ analogously. Letting
$\tau_\lambda = \tau_0 + \deltau^\top \lambda$, where $\tau_i$ and
$T_i$ are the values of $\tau$ and $T$ at $p_i$ for each $i$, we
define:
\begin{align}
  \label{eq:Gi}
  G_0(\lambda) &= \tau_\lambda + T_\lambda + s^\theta h \|p_\lambda\|, \\
  G_1(\lambda) &= \tau_\lambda + T_\lambda + s^\theta_\lambda h \|p_\lambda\|.
\end{align}
Like with $F_0$ and $F_1$, the only difference between $G_0$ and $G_1$
is between the terms containing $s^\theta$ and $s^\theta_\lambda$. We
do not explicitly refer to them in the rest of this paper, but the
cost functions $G_{\texttt{rhr}}$, $G_{\texttt{mp0}}$, and
$G_{\texttt{mp1}}$ are defined analogously to $\Frhr$, $\Fmpzero$, and
$\Fmpone$ from the $\theta$-rules $G_0$ and $G_1$.

To solve the factored eikonal equation, we choose a factoring radius
$\rfac$, replacing $F_i$ with $G_i$ in eq.\@
\ref{eq:constrained-minimization} for nodes which lie within a
distance $\rfac$ of $x^\circ$. For constant slowness, the effect of
this is to solve eq.\@ \ref{eq:eikonal} exactly inside of the
locally factored region. For clarity, this is depicted in figure
\ref{fig:factoring-example}. Algorithm \ref{alg:top-down} can be
applied to solve eq.\@ \ref{eq:constrained-minimization} for
factored nodes. The gradient and Hessian of $G_i$ are simple
modifications of the gradient and Hessian for $F_i$.

\begin{lemma}
  The gradient and Hessian of $G_i$ for $i = 0, 1$ are given
  by:
  \begin{align}
    \label{eq:G-grad-hess}
    \nabla G_i(\lambda) &= \nabla F_i(\lambda) - \deltau + \frac{s^\circ h}{\|p_\lambda - p^\circ\|} \delP^\top {(p_\lambda - p^\circ)}, \\
    \nabla^2 G_i(\lambda) &= \nabla^2 F_i(\lambda) + \frac{s^\circ h}{\|p_\lambda - p^\circ\|} \delP^\top \proj^\perp_{p_\lambda - p^\circ} \delP.
  \end{align}
\end{lemma}

\section{Implementation of the ordered line integral
  method}\label{sec:implementation}

\begin{figure}
  \centering
  \includegraphics[width=0.95\linewidth]{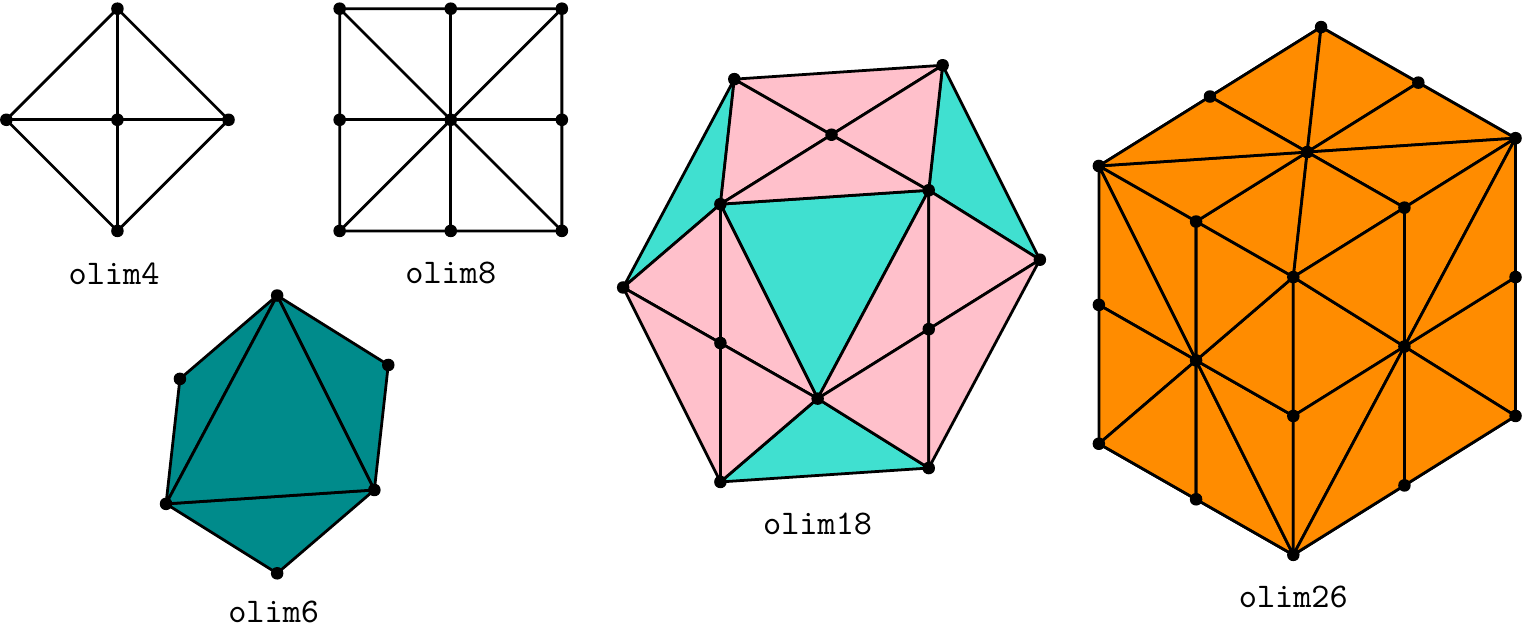}
  \caption{\emph{Neighborhoods for the \emph{top-down} family of algorithms.}
    Algorithms \texttt{olim4} and \texttt{olim8} are 2D solvers and
    the rest are 3D solvers. The color coding of tetrahedron updates
    is the same for this figure and figure \ref{fig:octant-numbering}
    below.}\label{fig:neighborhoods}%
  \includegraphics[width=0.95\linewidth]{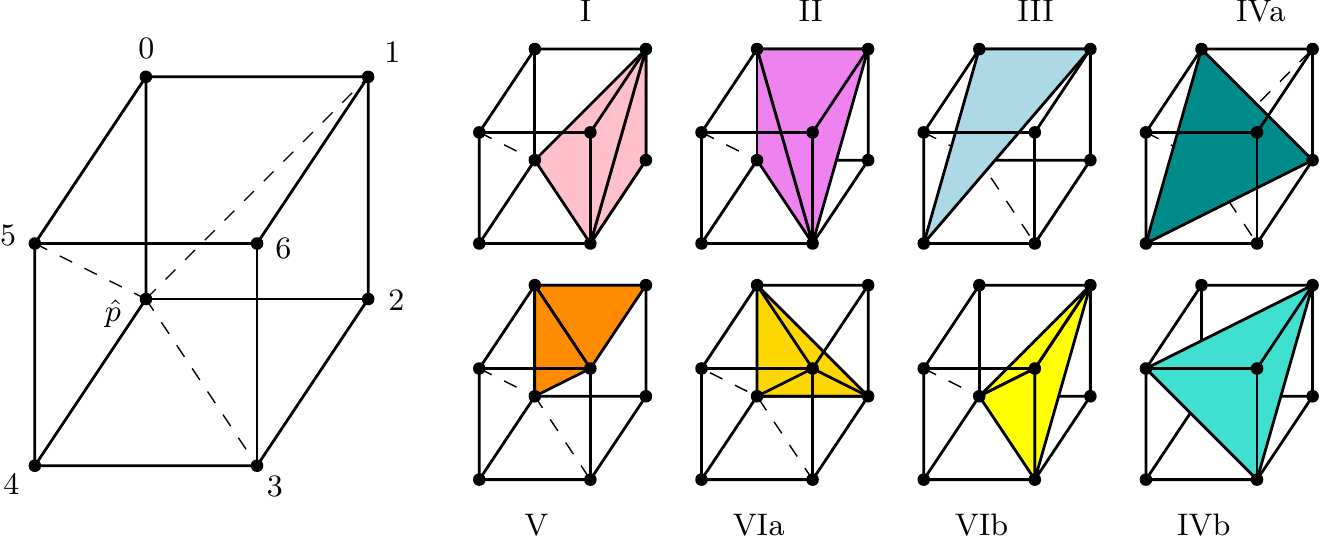}
  \caption{\emph{Numbering scheme for update groups for the \emph{top-down} solvers.} In this
    diagram, $\hat{p}$ is being updated. The diagonally opposite node
    is the sixth (last) node, with the other six nodes numbered 0--5
    cyclically.}\label{fig:octant-numbering}
  { \footnotesize
    \vspace{1em}
    \begin{tabular}{c|cccccc|cccccc|cccccc|cc}
      0 & $\groupmarker$ & & & & $\groupmarker$ & $\groupmarker$ & $\groupmarker$ & & & $\groupmarker$ & & $\groupmarker$ & $\groupmarker$ & & $\groupmarker$ & & & $\groupmarker$ & $\groupmarker$ & \\
      1 & $\groupmarker$ & $\groupmarker$ & & & & $\groupmarker$ & $\groupmarker$ & $\groupmarker$ & & & $\groupmarker$ & & $\groupmarker$ & $\groupmarker$ & & $\groupmarker$ & & & & $\groupmarker$ \\
      2 & $\groupmarker$ & $\groupmarker$ & $\groupmarker$ & & & & & $\groupmarker$ & $\groupmarker$ & & & $\groupmarker$ & & $\groupmarker$ & $\groupmarker$ & & $\groupmarker$ & & $\groupmarker$ & \\
      3 & & $\groupmarker$ & $\groupmarker$ & $\groupmarker$ & & & $\groupmarker$ & & $\groupmarker$ & $\groupmarker$ & & & & & $\groupmarker$ & $\groupmarker$ & & $\groupmarker$ & & $\groupmarker$ \\
      4 & & & $\groupmarker$ & $\groupmarker$ & $\groupmarker$ & & & $\groupmarker$ & & $\groupmarker$ & $\groupmarker$ & & $\groupmarker$ & & & $\groupmarker$ & $\groupmarker$ & & $\groupmarker$ & \\
      5 & & & & $\groupmarker$ & $\groupmarker$ & $\groupmarker$ & & & $\groupmarker$ & & $\groupmarker$ & $\groupmarker$ & & $\groupmarker$ & & & $\groupmarker$ & $\groupmarker$ & & $\groupmarker$ \\
      \multicolumn{1}{c}{} & \multicolumn{6}{c}{I} & \multicolumn{6}{c}{II} & \multicolumn{6}{c}{III} & \multicolumn{2}{c}{IV}
    \end{tabular} \\
    \vspace{1em}
    \begin{tabular}{c|cccccc|cccccc|ccc}
      0 & $\groupmarker$ & & & & & $\groupmarker$ & $\groupmarker$ & & & & $\groupmarker$ & & $\groupmarker$ & & \\
      1 & $\groupmarker$ & $\groupmarker$ & & & & & & $\groupmarker$ & & & & $\groupmarker$ & & $\groupmarker$ & \\
      2 & & $\groupmarker$ & $\groupmarker$ & & & & $\groupmarker$ & & $\groupmarker$ & & & & & & $\groupmarker$ \\
      3 & & & $\groupmarker$ & $\groupmarker$ & & & & $\groupmarker$ & & $\groupmarker$ & & & $\groupmarker$ & & \\
      4 & & & & $\groupmarker$ & $\groupmarker$ & & & & $\groupmarker$ & & $\groupmarker$ & & & $\groupmarker$ & \\
      5 & & & & & $\groupmarker$ & $\groupmarker$ & & & & $\groupmarker$ & & $\groupmarker$ & & & $\groupmarker$ \\
      6 & $\groupmarker$ & $\groupmarker$ & $\groupmarker$ & $\groupmarker$ & $\groupmarker$ & $\groupmarker$ & $\groupmarker$ & $\groupmarker$ & $\groupmarker$ & $\groupmarker$ & $\groupmarker$ & $\groupmarker$ & $\groupmarker$ & $\groupmarker$ & $\groupmarker$ \\
      \multicolumn{1}{c}{} & \multicolumn{6}{c}{V} & \multicolumn{6}{c}{VI} & \multicolumn{3}{c}{VII}
    \end{tabular}%
  }
  \caption{\emph{Tables of update groups.} These tables should be
    scanned columnwise: each column of dots selects a different
    tetrahedron. Tetrahedra (0, 1, 2), (2, 3, 4), and (4, 5, 0) in
    group I and all tetrahedra in group VII are degenerate and can be
    omitted.}\label{fig:tetrahedra-groups}
\end{figure}

In this section, we describe our \emph{bottom-up} and \emph{top-down}
algorithms (see fig.\@ \ref{fig:classification}). We focus on 3D
solvers, since in 2D the distinction between the two is less
important. Each algorithm reduces the number of updates that are done
without degrading solution accuracy by using an efficient enumeration
or search for update simplexes. The primary difference between the two
algorithms is the ordering of the updates' dimensions: \emph{top-down}
proceeds from $d = n - 1$ down to $d = 0$, skipping lower-dimensional
updates when possible, and \emph{bottom-up} proceeds from $d = 0$ up
to $d = n - 1$, skipping higher-dimensional updates when it can.

\subsection{Simplex enumeration for the \emph{top-down}
  algorithm}\label{ssec:simplex-enumeration}

When a node is first removed from \texttt{front} and has just become
\texttt{valid} (item\@ \ref{enum:get-node} in algorithm\@
\ref{alg:dijkstra-like}), an isotropic solver must do updates
involving, at the very least, the node's $2n$ nearest neighbors. We
can use larger neighborhoods to improve the accuracy of the
result. Doing so does not necessarily improve the order of convergence
of the solver, but can significantly improve the error constant of the
solution. For all of the solvers considered in this paper, in 3D, we
only consider neighborhoods with at most 26 neighbors. See fig.\@
\ref{fig:neighborhoods} for neighborhoods for the \emph{top-down}
solver.

For the \emph{top-down} solver, we simplify things by only doing
updates which have $p_i = \pnew$ for some $i$. To iterate over all
update simplexes, by symmetry, we enumerate all simplexes in a single
octant. This means enumerating, e.g., ${7 \choose 3} = 35$
tetrahedra. Some choices lead to degenerate tetrahedra, so the number
of nondegenerate update tetrahedra is fewer than 35 per octant. This
makes it reasonable to write out the update procedure as straight-line
code.

We enumerate the tetrahedra in a type of ``shift-order'' (see, e.g.,
\cite{arndt2010matters})---that is, we start with an unseen bit
pattern, and group this pattern together with all of its shifts (with
rotation). This groups the tetrahedra into sets that are rotationally
symmetric about the diagonal of the octant. In our implementation, we
conditionally compile different groups so that no unnecessary
branching is done. This is done using C++
templates~\cite{stroustrup2013c++}. Example stencils for the versions
of \texttt{olim6}, \texttt{olim18}, and \texttt{olim26} that are used
for our numerical test are shown in figure
\ref{fig:neighborhoods}. The tetrahedron groups are shown in figures
\ref{fig:octant-numbering} and \ref{fig:tetrahedra-groups}.

\subsection{Update gaps for tetrahedron
  groups}\label{ssec:update-gaps}

If we apply theorem \ref{thm:causality} to the tetrahedron groups
enumerated in figures \ref{fig:octant-numbering} and
\ref{fig:tetrahedra-groups}, we get the following update gaps,
enumerated in the same order as fig.\@ \ref{fig:octant-numbering}
(ignoring the $s^\theta h$ factor): \vspace{0.5em}
\begin{center}
  \begin{tabular}{lc|lc|lc|lc}
    Group I & $\nicefrac{1}{\sqrt{2}}$ & Group II & $\nicefrac{1}{\sqrt{2}}$ & Group III & $\nicefrac{1}{\sqrt{2}}$ & Group IVa & 0 \\
    \midrule
    Group V & $\nicefrac{1}{\sqrt{3}}$ & Group VIa & 0 & Group VIb & $\boldsymbol{\nicefrac{2}{\sqrt{3}}}$ & Group IVb & $\nicefrac{1}{\sqrt{2}}$
  \end{tabular}
\end{center}
\vspace{0.5em} The update gap is first explored in Tsitsiklis's
original paper~\cite{tsitsiklis1995efficient}; in this work, the fact
that Group IVa has no update gap and that the update gap of Group V is
$1/\sqrt{3}$ is noted and an $O(N^n)$ algorithm based on Dial's
algorithm is presented using Group V for the update tetrahedra. This
same observation is made in a more recent paper about a method based
on Dial's algorithm~\cite{kim2001calo}. A method based on a
combination of tetrahedra groups will have an update gap that is the
minimum of each of the individual groups' gaps. We note here that a
solver based on a combination of Groups I and VIb has a larger update
gap than a solver based on Group V. This should have a positive impact
on the performance of any parallel Dijkstra-like method.

\subsection{The search procedure used by the \emph{bottom-up}
  algorithm}\label{ssec:bottom-up-search}

\begin{figure}[t]
  \centering
  \includegraphics[width=0.85\linewidth]{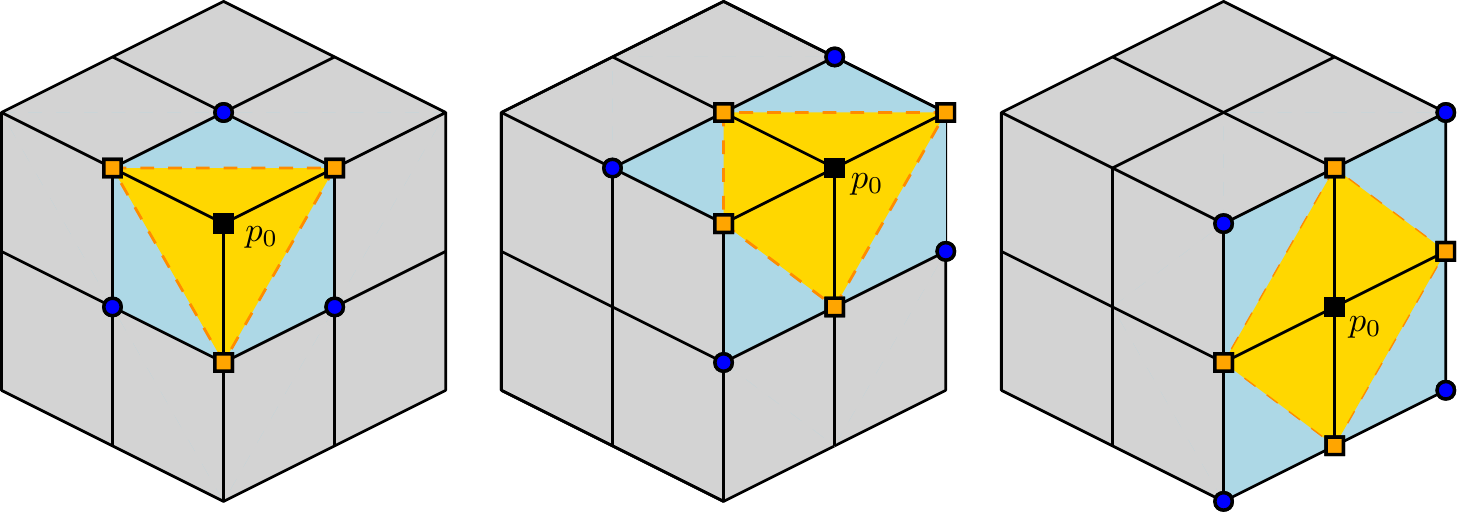}
  \caption{\emph{The three types of neighborhoods for the
      \emph{bottom-up} algorithm.} The yellow and blue regions
    indicate where triangle and tetrahedron updates may be performed,
    respectively. For instance, with $p_0$ the minimizing line update
    vertex, candiates for $p_1$ consist of the yellow nodes: triangle
    updates involving these candidates and $p_0$ will be
    performed. Once a yellow node ($p_1$) has been selected,
    tetrahedron updates involving the neighboring blue nodes
    (candidates for $p_2$) will be performed. Note that the updates
    performed correspond roughly to a combination of groups I, V, VIa,
    and VIb.}\label{fig:hu-neighborhoods}
\end{figure}

\begin{figure}
  \centering
  \includegraphics[width=0.4\linewidth]{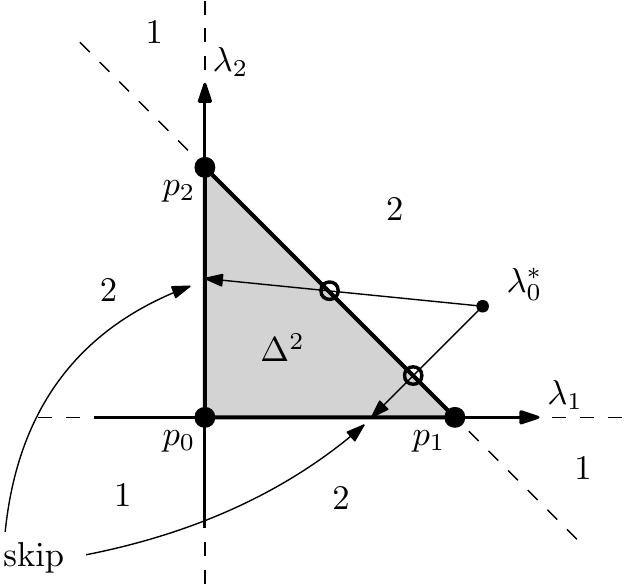}
  \caption{\emph{Skipping lower-dimensional updates when solving the
      unconstrained minimization problem.} For $d = 2$, if
    $\lambda_0^* \in \Delta^2$, all three triangle updates can be
    skipped. On the other hand, when minimizing $F_0$ using theorem
    \ref{thm:f0-exact}, if $\lambda^*_0 \notin \Delta^2$ and depending
    on where $\lambda_0^*$ lies, it is possible to skip one or two
    triangle updates. In this case, we label the different regions by
    the number of updates that it is possible to skip: $\lambda^*$
    here lies in ``zone 2'', since it is possible to skip the two
    triangle updates on the opposite side of $\Delta^2$. Along the
    same lines, if $\lambda^*$ were to lie in ``zone 1'', two triangle
    updates would be ``visible'', and it would only be possible to
    skip one update.}\label{fig:skip-zones}
\end{figure}

Another approach is to use local characteristic information obtained
by performing lower-dimensional updates to help us avoid performing
unnecessary higher-dimensional updates. Intuitively, if we find the
minimum line update ($d = 0$), then we can avoid triangle updates
($d = 1$) that don't include the minimizing line update. Since we only
perform updates for $d > 0$ that include the newly \texttt{valid} node
$\pnew$, we can start our search with $d = 1$.

After doing an update of dimension $d$, we find neighboring updates of
dimension $d + 1$. For $n = 3$ and $d = 1$, for each $p_0$, we find
points $p_1$ that satisfy $\|p_0 - p_1\|_1 \leq 1$. For $d = 2$, for
each pair $(p_0, p_1)$, we find points $p_2$ that satisfy
$\|p_0 - p_2\|_1 \leq 2$ and $\|p_1 - p_2\|_1 \leq 2$
simultaneously. In practice, we only find these neighbors once and
precompute an array of indices to be used later. The neighborhoods
computed in this way are shown in fig.\@ \ref{fig:hu-neighborhoods}.

\subsection{Minimization algorithms and skipping
  updates}\label{ssec:algorithms-and-skipping}

When $F_i = F_0$, we can use theorem \ref{thm:f0-exact} or
\ref{thm:equivalence} to compute $\lambda_0^*$. If $F_i = F_1$, we
need to use an algorithm that can solve the constrained optimization
problem defined by eq.\@ \ref{eq:constrained-minimization}. Our
approach has been to use sequential quadratic programming (SQP),
although there are many other
options~\cite{bertsekas1999nonlinear,nocedal2006numerical}. It is
possible to skip some updates; and, indeed, the performance of our
algorithms comes about largely because of this happening.

We skip updates in three different ways. The first two approaches are
used by the \emph{top-down} algorithms, and the third by the
\emph{bottom-up} algorithms.

\paragraph{Top-down constrained skipping.} When computing an update
using a constrained solver, we can rule out all incident
lower-dimensional updates, since we have computed the global
constrained optimum on $\Delta^d$.

\paragraph{Top-down unconstrained skipping.} If we do an update using
an unconstrained solver, then depending on where the optimum
$\lambda_0^*$ lies, we can skip some or all lower-dimensional
updates. The idea is simple and is best depicted visually, see fig.\@
\ref{fig:skip-zones}.

\paragraph{Bottom-up KKT skipping.} We can also skip
higher-dimensional updates. For example, if we do the three triangle
updates on the boundary of a tetrahedron update, we can use the
Karush-Kuhn-Tucker necessary conditions for optimality of a
constrained optimization problem~\cite{nocedal2006numerical} to
determine if the minimizer on the boundary is also a global minimizer
for the constrained minimization problem given by eq.\@
\ref{eq:constrained-minimization}. See appendix\@
\ref{sec:kkt-skipping}. We note that a modified version of this
strategy for skipping updates was used in the work on computing the
quasipotential for nongradient SDEs in 3D~\cite{yang2019computing}.

\subsection{The \emph{bottom-up} and \emph{top-down} algorithms}\label{ssec:top-down-and-bottom-up}

\begin{algorithm}[t]
  \caption{The \emph{top-down} algorithm.}\label{alg:top-down}
  \begin{enumerate}[nolistsep]
  \item Set $\hat{U} \gets \infty$.
  \item Initialize $\calV_d$ according to eq.\@ \ref{eq:calU} for
    each $d = 0, \hdots, n - 1$.
  \item For $d = n - 1$ down to $0$:
    \begin{enumerate}
    \item For each $(p_0, \hdots, p_d) \in \calV_d$:
      \begin{enumerate}
      \item If $F_i = F_0$ (\texttt{mp0} or \texttt{rhr}):
        \begin{enumerate}
        \item Compute $U$ for $(p_0, \hdots, p_{d})$ using theorem
          \ref{thm:f0-exact} or \ref{thm:equivalence}.
        \item Remove updates from $\calV_0, \hdots, \calV_{d-1}$ by
          visibility (see figure \ref{fig:skip-zones}).
        \end{enumerate}
      \item Otherwise, if $F_i = F_1$ (\texttt{mp1}):
        \begin{enumerate}
        \item Compute $U$ by solving eq.\@
          \ref{eq:constrained-minimization} numerically (we use SQP).
        \item Remove all lower-dimensional updates from
          $\calV_0, \hdots, \calV_{d-1}$.
        \end{enumerate}
      \item Set $\hat{U} \gets \min(\hat{U}, U)$.
      \end{enumerate}
    \end{enumerate}
  \end{enumerate}
\end{algorithm}

\begin{algorithm}[t]
  \caption{The \emph{bottom-up} algorithm.}\label{alg:bottom-up}
  \begin{enumerate}[nolistsep]
  \item Set $\hat{U} \gets \infty$ and $p_0 \gets \pnew$.
  \item For $i = 1, \hdots, n - 1$:
    \begin{enumerate}
    \item For each \texttt{valid} $p_i$ close enough to
      $p_0, \hdots, p_{i-1}$ (see section\@
      \ref{ssec:bottom-up-search}), do the update corresponding to
      $(p_0, \hdots, p_i)$ and keep track of the minimizing
      $\lambda^* \in \Delta^i$. This update can optionally be skipped
      by first computing $\mu^*$ corresponding to the optimum of the
      incident lower-dimensional update $(p_0, \hdots, p_{i-1})$ and
      checking if $\mu^* \geq 0$.\label{enum:bottom-up-for}
    \item Let $p_i$ be the node which forms the update with the
      minimum value.
    \item If $F_i = F_0$ (\texttt{mp0} or \texttt{rhr}), compute $U$
      for $(p_0, \hdots, p_{i})$ using theorem \ref{thm:f0-exact} or
      \ref{thm:equivalence}.
    \item Otherwise, if $F_i = F_1$ (\texttt{mp1}), compute $U$ for
      $(p_0, \hdots, p_{i})$ by solving eq.\@
      \ref{eq:constrained-minimization}.
    \item Set $\hat{U} \gets \min(\hat{U}, U)$.
    \end{enumerate}
  \end{enumerate}
\end{algorithm}

In this section, we describe our \emph{bottom-up} and \emph{top-down}
algorithms. We note for clarity that these algorithms correspond to
the ``compute $\hat{U}$'' part of item \ref{enum:update-U} of alg.\@
\ref{alg:dijkstra-like}.

To describe our \emph{top-down} algorithm (see algorithm
\ref{alg:top-down}), we define the set of admissible $d$-dimensional
updates neighboring the point $\hat{p}$ by:
\begin{equation}\label{eq:calU}
  \begin{aligned}
    \calV_d = \big\{\set{p_0, \hdots, p_d}: \; &p_i\texttt{.state}=\texttt{valid} \mbox{ for } i = 0, \hdots, d, \\
    &\mbox{and} \set{p_0, \hdots, p_d} \mbox{are in a selected update group}, \\
    &\mbox{and } \pnew \in \set{p_0, \hdots, p_d} \big\}
  \end{aligned}
\end{equation}
for $d = 0, \hdots, n - 1$. The update set $\calV_d$ collects all
admissible simplex updates of a given dimension: i.e., updates which
both belong to a group as defined in section\@
\ref{ssec:simplex-enumeration} and are \texttt{valid}. The third
condition is an important optimization. To see why it is correct, fix
an update set $\calV_d$. If $\set{p_0, \hdots, p_d}$ satisfies the
first two conditions but not the third, we can see that $\hat{p}$
would have already been updated from it in a previous iteration. All
new information affecting $\hat{U}$ during this iteration must be
computed from an update involving $\pnew$.

The \emph{bottom-up} algorithm (algorithm \ref{alg:bottom-up}) builds up each
update $(p_0, \hdots, p_d)$ one vector at a time by searching for
adjacent minimizing updates of higher dimension. The optimization
involving $\pnew$ described above can be incorporated by initially
setting $p_0 \gets \pnew$.

\section{Numerical Results}\label{sec:numerical-results}

We do tests involving several different slowness functions with
analytic solutions for point source data, and a linear speed function
(i.e., $1/s$) which has been shown to be amenable to local
factoring. For each quadrature rule described in section\@
\ref{ssec:quadrature} (\texttt{mp0}, \texttt{mp1}, or \texttt{rhr}),
we have two 2D algorithms, \texttt{olim4} and \texttt{olim8},
corresponding to 4-\ and 8-point stencils, respectively. Since there
is no advantage in 2D, we don't apply the \emph{top-down} or
\emph{bottom-up} approaches. In 3D, we have three \emph{top-down}
algorithms: \texttt{olim6} (group IVa), \texttt{olim18} (groups I,
IVa, and IVb), and \texttt{olim26} (group V). We also test the
\emph{bottom-up} algorithm \texttt{olim3d} (see figure
\ref{fig:hu-neighborhoods}).

\subsection{Implementation Notes}\label{ssec:impl-notes}

\begin{figure}
  \centering \includegraphics[width=\linewidth]{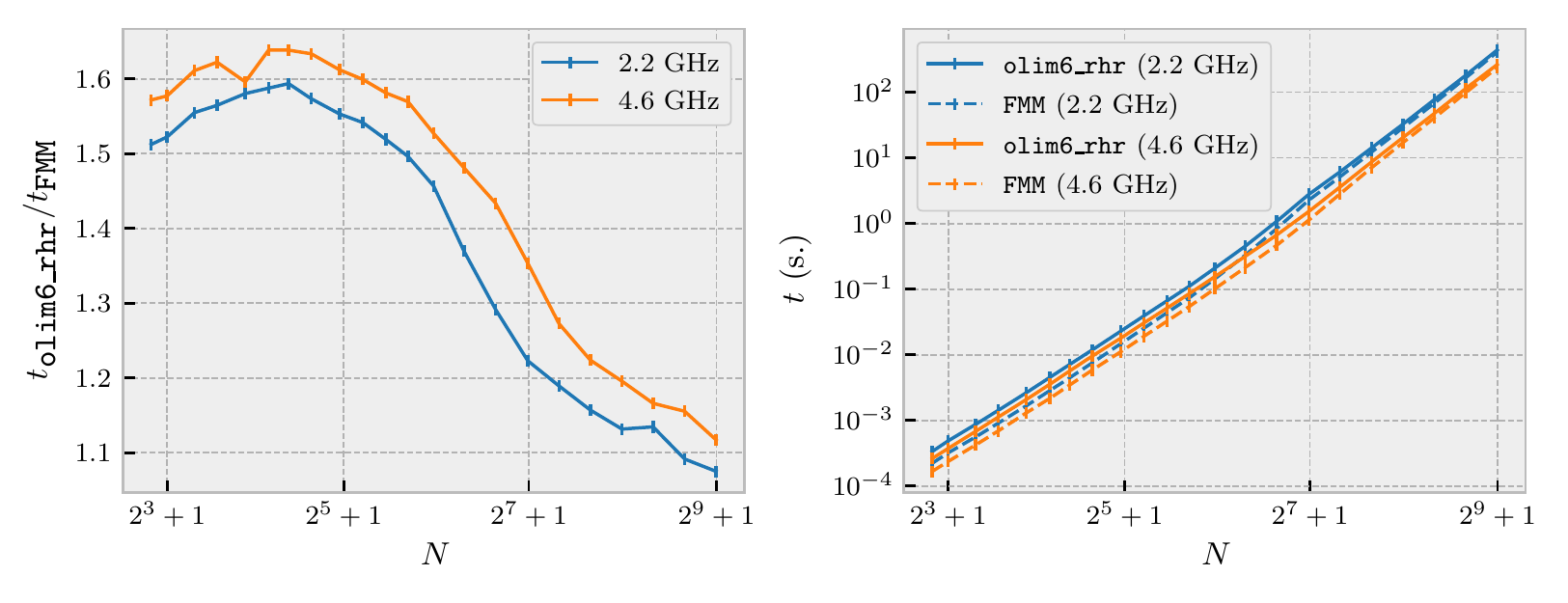}%
  \caption{\emph{Slowdown incurred by using \texttt{olim6\_rhr}
      instead of the FMM.} From left to right: 1) the ratio of
    runtimes versus $N$, 2) the total CPU runtime of each solver. We
    compare results on two different computers: ``2.2 GHz'' is a 2015
    MacBook Air with a 2.2 GHz Intel Core i7 CPU, 8 GB of 1600 MHz
    DDR3 RAM, a 256 KiB L2 cache, and a 4 MiB L3 cache; ``4.6 GHz'' is
    a custom built workstation running Linux with a 4.6 GHz Intel Core
    i7 CPU, 64 GB of 2133 MHz DDR4 RAM, a 1536 KiB L2 cache, and 12
    MiB L3 cache. Both computers have 32 KiB L1 instruction caches and
    data caches. The plots here use our standard $\Omega = [-1, 1]^3$
    domain discretized into $N = 2^p + 1$ nodes in each direction,
    with $s \equiv 1$ and a point source at the
    origin.}\label{fig:speed-comparison}
\end{figure}

\begin{figure}
  \includegraphics[width=\linewidth]{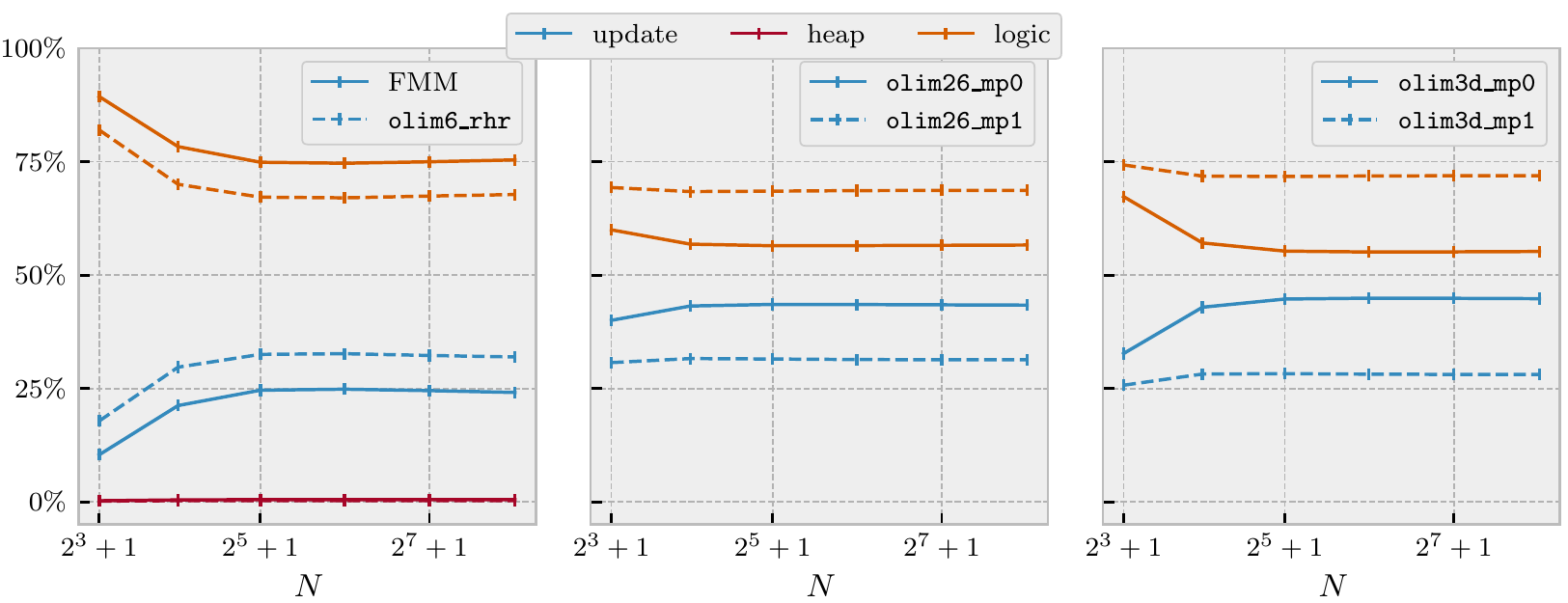}%
  \caption{\emph{Percentage of time spent on different tasks as
      determined by profiling.} The ``update'' tasks and ``heap''
    tasks are clearly defined, while the ``logic'' task contains a
    variety of things related to control-flow, finding neighbors, and
    memory movement---basically, the parts of algorithm
    \ref{alg:dijkstra-like} that don't clearly pertain to computing
    new $\hat{U}$ values or keeping \texttt{front} updated. From these
    plots, it is clear that memory speed plays a large role in
    determining efficiency. To some extent, even though the more
    complicated update procedures are slower, their slowness is hidden
    somewhat by memory latency as problem sizes grow. For large $N$
    and some solvers (the middle and right plots), ``heap'' takes too
    little time, and is not picked up by the
    profiler.}\label{fig:tasks}
\end{figure}

Before describing our numerical tests, we briefly comment on our
implementation and make some observations about its performance. A
discussion of some of the choices that we made in our implementation
follows:
\begin{itemize}
\item We precompute and cache all values of $s$ on the grid $\calG$,
  as opposed to reevaluating $s$, because we assume that $s$ will be
  provided as gridded data (consider, e.g., the shape from shading
  problem~\cite{kimmel2001optimal}, where the input data is an image).
\item We maintain \texttt{front} using a priority queue implemented
  using an array-based binary heap, which is updated using the
  \texttt{sink} and \texttt{swim} functions described in Sedgewick and
  Wayne~\cite{sedgewick2011algorithms}.
\item We store \texttt{front} as a dense grid of states: for each node
  in $p \in \calG$, we track $p$\texttt{.state} for all time for every
  node. We could implement a sparse \texttt{front} using a hash map or
  a quadtree or octree, which would save space, but would also be much
  slower to update.
\end{itemize}

We use a policy-based design~\cite{alexandrescu2001modern} written in
C++. This allows us to conditionally compile different features and
reuse logic to implement different Dijkstra-like algorithms. In
particular, we implement the standard FMM~\cite{sethian1996fast} and
make a direct comparison between it and the ordered line integral
method which it is equivalent to, \texttt{olim6\_rhr} (see figure
\ref{fig:speed-comparison}). We have found that only a modest slowdown
is incurred by using \texttt{olim6\_rhr} for problems of moderate
size. The disparity between the two is greater for smaller problem
sizes, which is due to cache effects. In general, the difference in
speed is due to the fact that the FMM's update is extremely simple
since it requires only the solution of quadratic equations.

Using Valgrind~\cite{nethercote2007valgrind}, we profiled running our
solver on the numerical tests below for different problem sizes and
categorized the resulting profile data. See figure
\ref{fig:tasks}. The ``update'' task corresponds to time spent
actually computing updates, the ``logic'' task is a grab bag category
for time spent on program logic, and ``heap'' corresponds to updating
the array-based heap which implements \texttt{front}. Since the
asymptotic complexity of the ``update'' and ``logic'' sections is
$O(N^n)$, and since ``heap'' is $O(N^n \log N)$, we can see from
figure \ref{fig:tasks} that since so little time is spent updating the
heap, \emph{the algorithm's runtime is better thought of as $O(N^n)$
  for practical problem sizes (although this obviously not technically
  correct!).} This is a consequence of using an array-based heap whose
updates are cheap and cache friendly, and a dense grid of states,
which can be read from and written to in $O(1)$ time.

As an aside, we mention that thorough numerical studies of eikonal
solvers in the literature have been scarce, but we can point out a
recent study which seeks to close this gap~\cite{gomez2019fast}. Our
studies profiling our implementation using Valgrind are carried out in
the same spirit as this work.

\subsection[Single point source]{Slowness functions with an analytic
  solution for a point source}\label{ssec:point-source-problems}

\begin{figure}
  \centering \includegraphics[width=\linewidth]{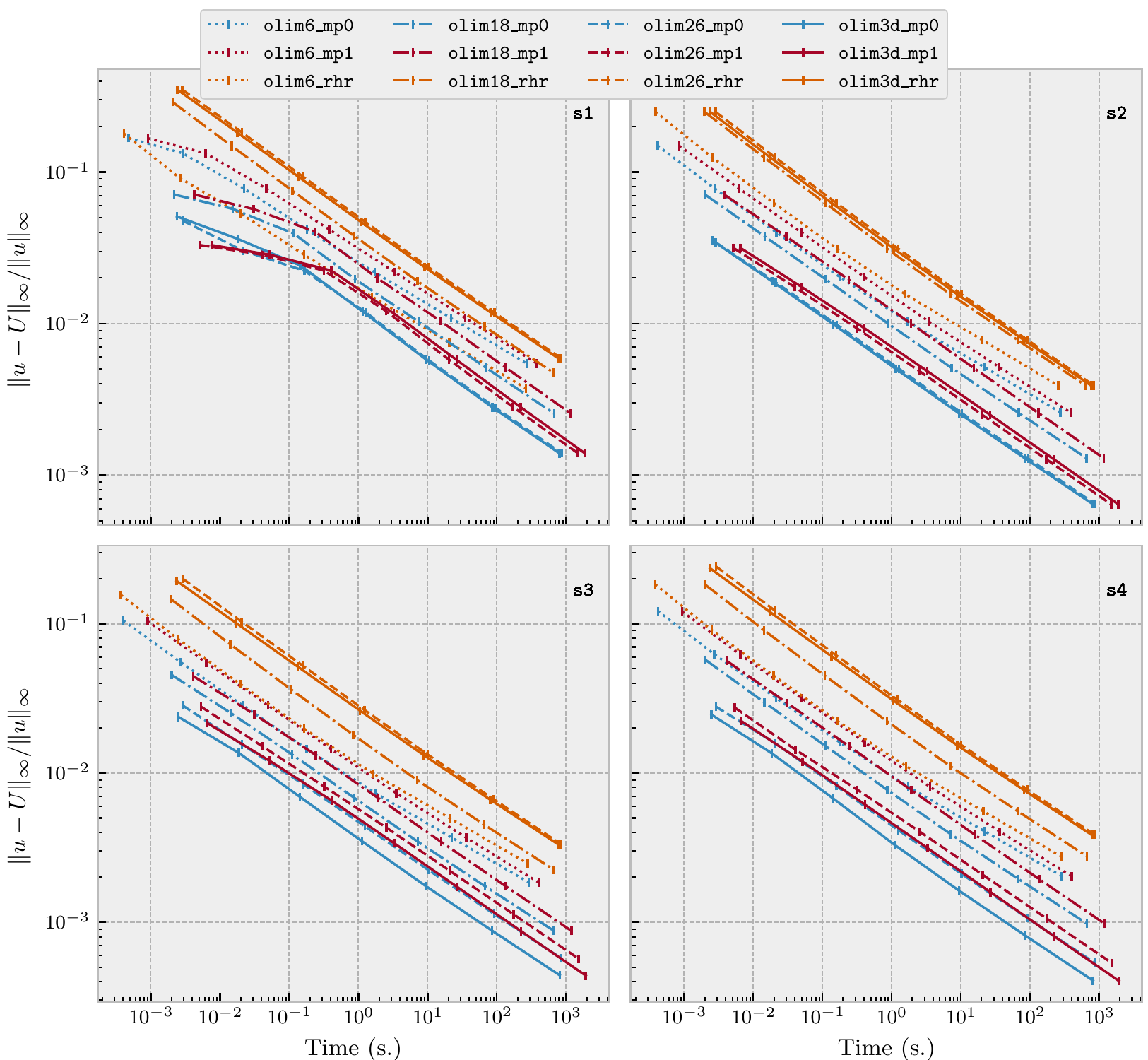}
  \caption{\emph{Relative $\ell_\infty$ error plotted against CPU
      runtime in seconds.} The domain is $\Omega = [-1, 1]^3$
    discretized uniformly in each direction into $N = 2^p + 1$ points,
    where $p = 3, \hdots, 9$, so that there are $N^3$ points
    overall. The slowness functions used are listed in section\@
    \ref{ssec:point-source-problems}. We note that the horizontal and
    vertical axes of each subplot are the
    same.}\label{fig:time-vs-error}
\end{figure}

\begin{figure}
  \centering \includegraphics[width=\linewidth]{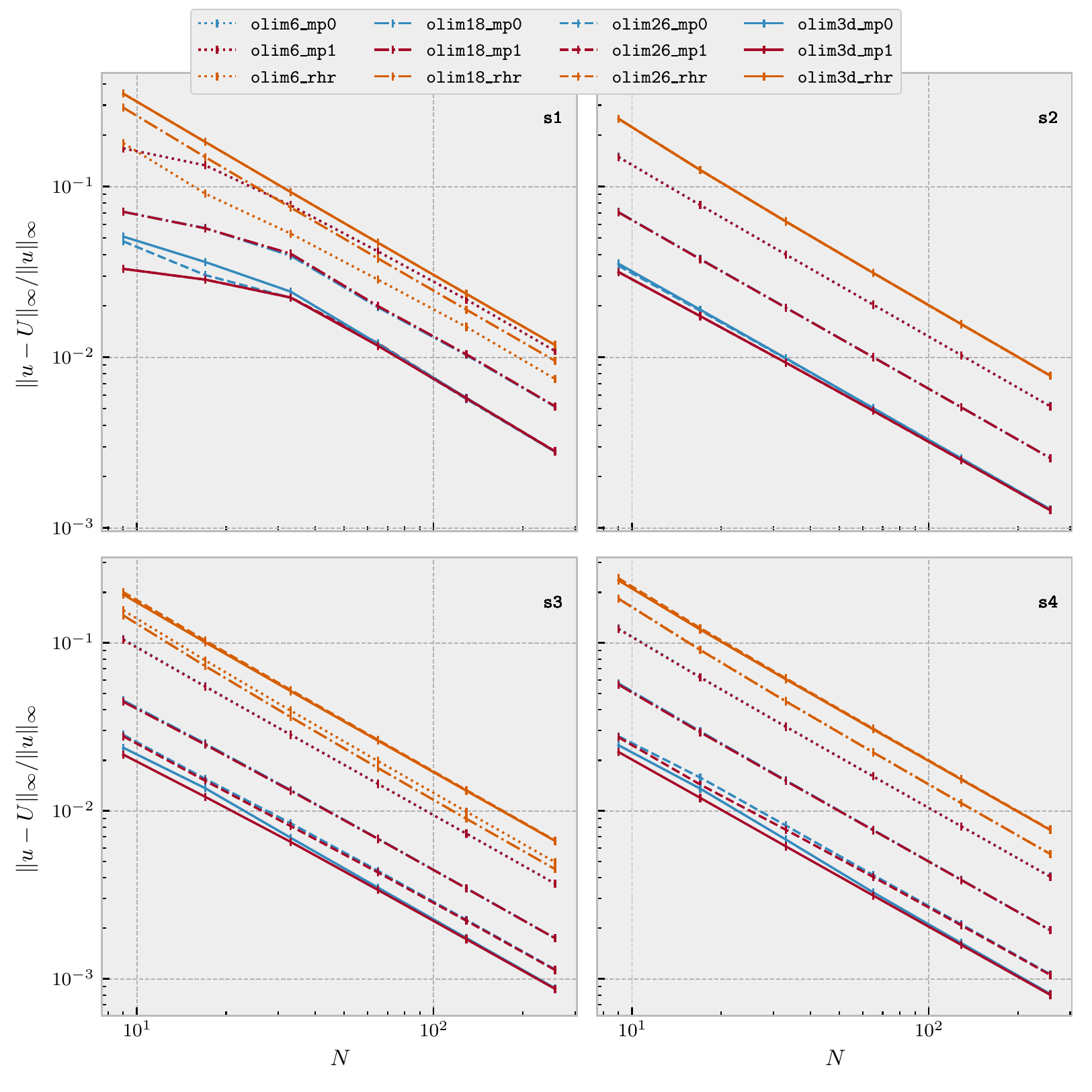}
  \caption{\emph{Relative $\ell_\infty$ error plotted versus $N$.} The
    setup is the same as in figure \ref{fig:time-vs-error}, except
    that $p = 3, \hdots, 8$, so that the largest $N$ is $257$ instead
    of $513$. For \texttt{olim26} and \texttt{olim3d}, we can see that
    \texttt{mp0} is initially less accurate than \texttt{mp1} but
    quickly attains parity, in accordance with theorem
    \ref{thm:mp0-newton}. For \texttt{olim6} and \texttt{olim18}, the
    error is the same between \texttt{mp0} and \texttt{mp1} for all
    slowness functions, so these plots
    overlap.}\label{fig:size-vs-error}
\end{figure}

Using eq.\@ \ref{eq:eikonal} directly, a simple recipe to create pairs
of slowness functions and solutions is to prescribe a continuous
function $u$ whose level sets are each homeomorphic to a ball and
compute $s(x) = \norm{\nabla u(x)}_2$ analytically, which is valid for
a single point source at the origin. Such tests allow us to observe
the effect of local factoring, and to see how \texttt{mp0},
\texttt{mp1}, and \texttt{rhr} compare. The following table lists our
test functions: \vspace{0.5em}
\begin{center}
  \begin{tabular}{cccc}
    Name & $u(x)$ & $s(x)$ \\
    \midrule
    \texttt{s1} & $\cos(r) + r - 1$ & $1 - \sin(r)$ \\
    \texttt{s2} & $r^2/2$ & $r$ \\
    \texttt{s3} & $S(x)^\top A S(x)$ & $\alpha\norm{\operatorname{diag}(C(x)){(A + A^\top)}S(x)}$ \\
    \texttt{s4} & $\tfrac{1}{2} x^\top A^{1/2} x$ & $\norm{x}_A = \sqrt{x^\top A x}$
  \end{tabular}
\end{center}
\vspace{0.5em} We assume that $x \in \Omega = [-1, 1]^3$. We also
define $r = \norm{x}$, and vector fields
$S(x) = (\sin(\alpha x_i))_{i=1}^3$ and
$C(x) = (\cos(\alpha x_i))_{i=1}^3$; we take $\alpha = \pi/5$. For
\texttt{s3} and \texttt{s4}, we assume that $A$ is symmetric positive
definite. In 3D, the matrices we use for \texttt{s3} and \texttt{s4}
are:\begin{equation} A_{\texttt{s3}} = \begin{bmatrix}
    1 & \nicefrac{1}{4} & \nicefrac{1}{8} \\
    \nicefrac{1}{4} & 1 & \nicefrac{1}{4} \\
    \nicefrac{1}{8} & \nicefrac{1}{4} & 1
  \end{bmatrix} = A_{\texttt{s4}}^{1/2}
\end{equation}

Our results are displayed in figures \ref{fig:time-vs-error} and
\ref{fig:size-vs-error}. We include the relative $\ell_\infty$ error
versus problem size and time, as well as the $\ell_\infty$ error
versus $N$. We summarize our observations:
\begin{itemize}
\item Using either of the midpoint rules (\texttt{mp0} or
  \texttt{mp1}) allows improved directional coverage to translate into
  an improved error constant. See figs.\@ \ref{fig:size-vs-error}
  and\@ \ref{fig:time-vs-error}.
\item For \texttt{rhr}, increased directional coverage
  (\texttt{olim6\_rhr} $\to$ \texttt{olim18\_rhr} $\to$
  \texttt{olim26\_rhr}) does not lead to an improved error. In fact,
  for \texttt{s1}, \texttt{s3}, and \texttt{s4}, increasing the
  directional coverages causes the accuracy to deteriorate (see figure
  \ref{fig:size-vs-error}). This may be due to the fact that
  quadrature error and linear interpolation error have different
  signs, and may partially compensate each other (e.g., in
  \texttt{olim6}). This effect may get worse with increased
  directional coverage.
\item If we scan each graph horizontally, focusing on the plot
  markers, we can see that the difference in error between
  \texttt{mp0} and \texttt{mp1} is minimal. For each \texttt{mp1}
  graph, the corresponding \texttt{mp0} graph has the same error, but
  is shifted to the left, reflecting the fact that the \texttt{mp0}
  OLIMs are substantially faster. This is consistent with theorem
  \ref{thm:mp0-newton}, which justifies the use of \texttt{mp0}. See
  fig.\@ \ref{fig:time-vs-error}.
\item With respect to the choice of neighborhood, \texttt{olim6} is
  the fastest; and, for each choice of neighborhood, \texttt{mp0}
  provides the best combination of speed and accuracy. See fig.\@
  \ref{fig:time-vs-error}. If we are willing to pay somewhat in speed,
  we can dramatically improve the error constant by improving the
  directional coverage and using a solver like
  \texttt{olim3d\_mp0}. This tradeoff is more pronounced for smaller
  problem sizes. A theme running through this work is that, as the
  problem size increases, memory access patterns come to dominate the
  runtime, and the disparity in runtimes between the faster and slower
  neighborhoods becomes less pronounced. To see this, compare the
  start of each graph in the top-left of the plots, and their ends in
  the bottom-right (again, see fig.\@ \ref{fig:time-vs-error}). We can
  observe, e.g., that the maximum horizontal distance between starting
  points and ending points has decreased significantly, which confirms
  this observation.
\item Our high accuracy algorithms allow us to obtain a better
  solution on rough grids: this is helpful since opportunities to
  refine the mesh are limited in 3D. Discretizing $\Omega = [-1, 1]^2$
  in each direction into $N = 2^{14} + 1$ nodes requires about as much
  memory as discretizing $\Omega = [-1, 1]^3$ with $N = 2^{9} + 1$,
  which leads to $h$ being 32 times smaller in 2D than in 3D.
\item In general, \texttt{olim26} and \texttt{olim3d} are
  significantly more accurate than \texttt{olim6} and \texttt{olim18}.
\end{itemize}

\subsection{A linear speed function}\label{ssec:slotnick}

\begin{figure}
  \centering
  \includegraphics[width=\linewidth]{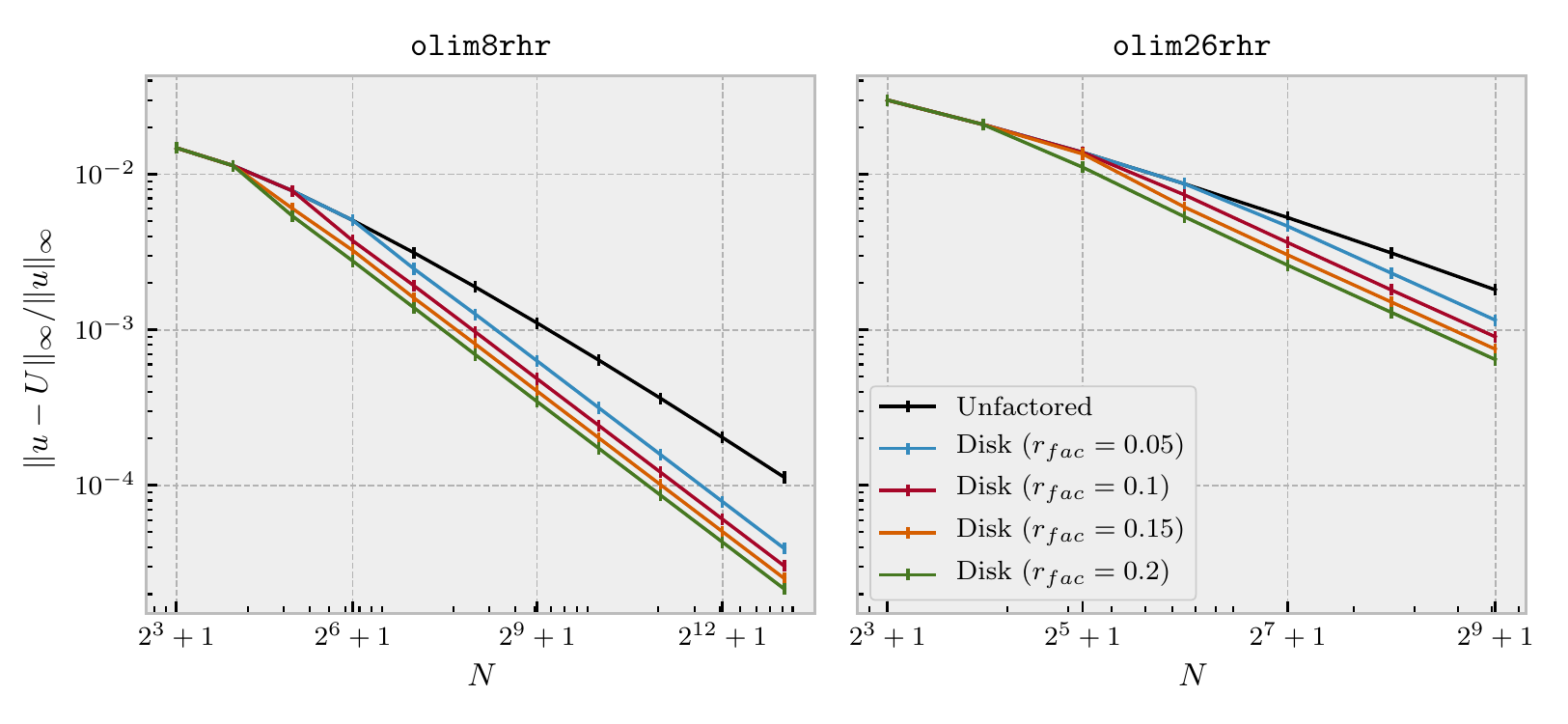}%
  \caption{\emph{Comparing different ways of selecting factored
      nodes.} For the test problem, $\Omega = [-1, 1]^n$, with $n = 2$
    (left) and $n = 3$ (right). The domain is descretized into $N^3$
    nodes, where $N = 2^p + 1$, so that $h = 2/(N - 1)$. The slowness
    function is constant ($s \equiv 1$). For the 2D problem,
    \texttt{olim8\_rhr} is used; \texttt{olim26\_rhr} is used for the
    3D problem. Solutions for the unfactored problem are plotted,
    along with solutions using a disk/sphere neighborhood with
    constant factoring radius given by $\rfac = 0.05, 0.1, 0.15,
    0.2$. We note that for this problem the choice $\rfac = \sqrt{n}$
    results in an exact solution. This only applies to the constant
    slowness function, $s \equiv 1$.
  }\label{fig:factoring-error-example}
\end{figure}

We consider a problem that has a known analytical solution and has
been used as a test problem for other factored eikonal equation
solvers before\footnote{We thank D.\ Qi for helpful discussions
  regarding this
  problem.}~\cite{slotnick1959lessons,fomel2009fast,qi2018corner}. For
a single point source at $x_i$ and a vector $v$, we define:
\begin{equation}
  \label{eq:slotnick-single-source}
  \frac{1}{s(x)} = \frac{1}{s(x_i)} + v^\top {(x - x_i)},,
\end{equation}
where $s_i = s(x_i)$. The analytic solution to eq.\@
\ref{eq:eikonal} for a single source and slowness function given by
eq.\@ \ref{eq:slotnick-single-source}
is~\cite{slotnick1959lessons}:
\begin{equation}
  \label{eq:slotnick-single-source-solution}
  u_i(x) = \frac{1}{\norm{v}} \cosh^{-1} \parens{1 + \frac{s_i}{2} s(x) \norm{v}^2 \norm{x - x_i}^2}.
\end{equation}
If we shift the point source from $x_i$ to another location $x_j$, we
find:
\begin{equation}
  \label{eq:slotnick-slowness-shift}
  \frac{1}{s_i} + v^\top {(x - x_j + x_j - x_i)} = \frac{1}{s_i} + v^\top {(x_j - x_i)} + v^\top {(x - x_j)} = \frac{1}{s_j} + v^\top {(x - x_j)}.
\end{equation}
That is, the slowness function $s$ remains unchanged as it is
rewritten with respect to a different source.

\begin{figure}
  \centering
  \includegraphics[width=\linewidth]{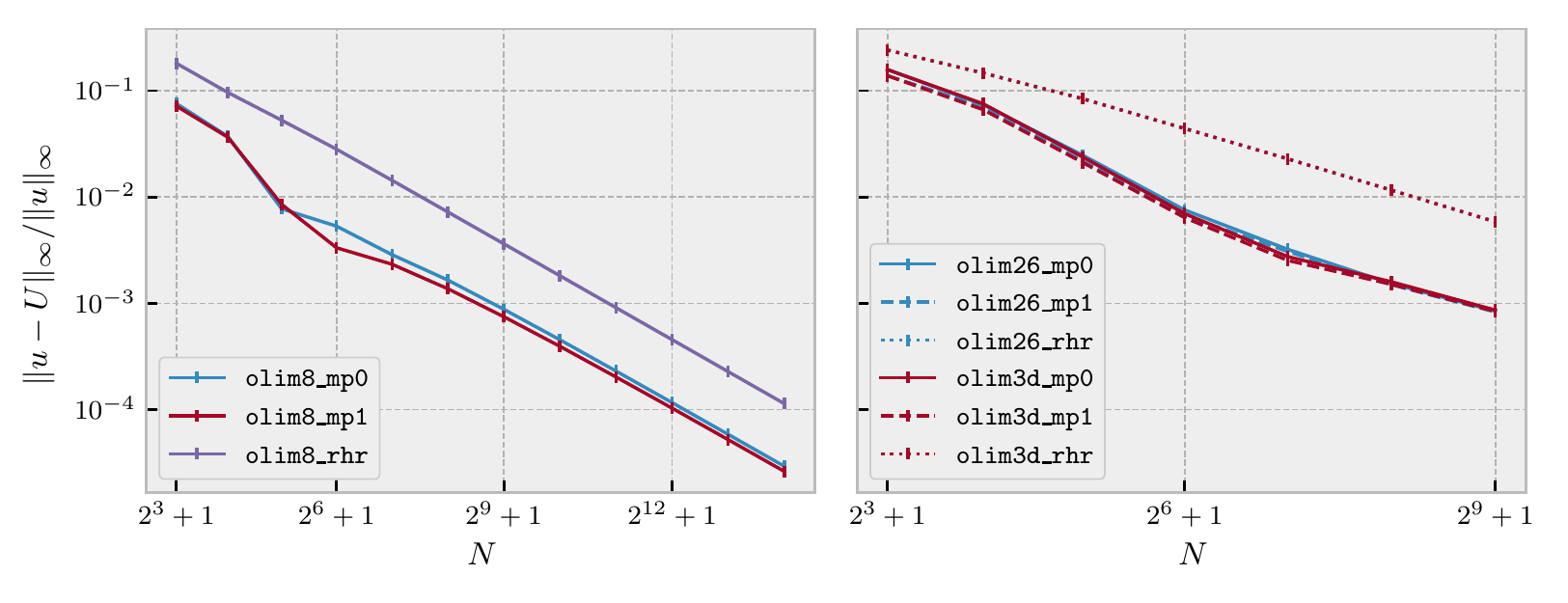}
  \includegraphics[width=\linewidth]{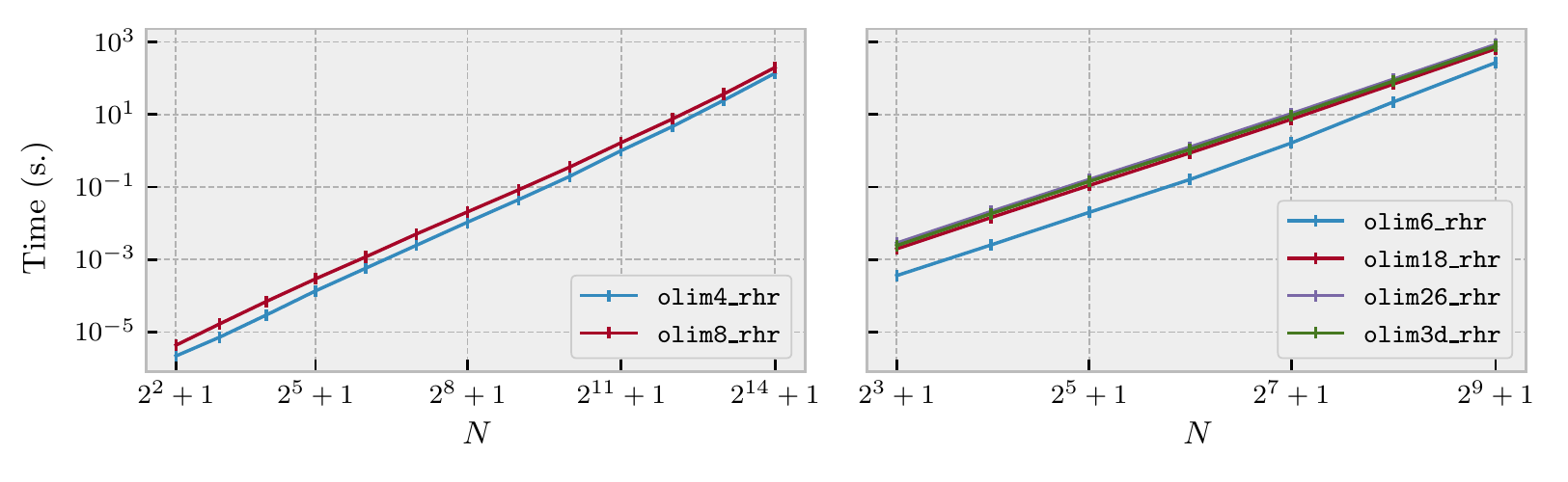}
  \caption{\emph{Numerical results for the linear speed function of
      section\@ \ref{ssec:slotnick}.} Problem sizes are $N = 2^p + 1$,
    where $p = 3, \hdots, 14$ in 2D and $p = 3, \hdots, 9$ in 3D. The
    total number of nodes is $N^n$, where $n = 2, 3$. See section\@
    \ref{ssec:slotnick} for least squares fits. Top row: relative
    $\ell_\infty$ error plotted versus $N$ in 2D (left) and 3D
    (right). Bottom row: wall clock time plotted versus $N$ in 2D
    (left) and 3D (right).}\label{fig:slotnick-plots}
\end{figure}

If $\set{x_i}$ is a set of point sources and $u_i(x)$ is the solution
of the eikonal equation for the single point source problem with point
source given by $x_i$, then the solution for the multiple point source
problem with sources $\set{x_i}$ is:
\begin{equation}
  u(x) = \min_i u_i(x).
\end{equation}
We use this formula to compare relative $\ell_\infty$ errors for each
of our OLIMs in 2D and \texttt{olim26} and \texttt{olim3d} in 3D for
this slowness function with a pair of point sources, $x_1 = (0, 0)$
and $x_2 = (0.8, 0)$ in 2D, and $x_1 = (0, 0, 0)$ and
$x_2 = (0.8, 0, 0)$ in 3D. We set the domain of the problem to be
$\Omega = [0, 1]^n$ and discretize it into $N = 2^p + 1$ points, so
that $h = (N-1)^{-1}$.

For this choice of slowness function, we plot the CPU runtime versus
$N$ (see figure \ref{fig:slotnick-plots}), along with the relative
$\ell_\infty$ error versus $N$ (see figure
\ref{fig:slotnick-plots}). We also do least squares fits for these
plots to get an overall sense of the accuracy and speed (see
\ref{table:qv-least-squares}).

We can see that our conclusions from section\@
\ref{ssec:point-source-problems} also hold for the multiple point
source problem. Additionally, our least-squares fits (table\@
\ref{table:qv-least-squares}) indicate to us that our algorithms'
runtimes are accurately described by the fit $T_N \sim C_T N^{\alpha}$
with $\alpha \approx n$, and the error by $E_N \sim C_E h^{\beta}$,
with $\beta \approx 1$ (here, $E_N$ is the relative $\ell_\infty$
error). In fact, for \texttt{olim26} and \texttt{olim3d} with
\texttt{mp0} or \texttt{mp1}, the power $\beta$ is improved beyond $1$
to $\beta \approx 1.3$.

\begin{table}
  \begin{subtable}{0.5\textwidth}
    \centering
    {
      \small
      \begin{tabular}{ccc}
        Neighborhood & $C_T$ & $\alpha$ \\
        \hline \noalign{\vskip 0.2em}
        \texttt{olim4} & $7.779\times 10^{-8}$ & 1.0785 \\
        \texttt{olim8} & $1.971\times 10^{-7}$ & 1.0515 \\
        \hline \noalign{\vskip 0.2em}
        \texttt{olim6} & $2.968\times 10^{-7}$ & 1.085 \\
        \texttt{olim18} & $2.984\times 10^{-6}$ & 1.018 \\
        \texttt{olim26} & $4.649\times 10^{-6}$ & 1.0103 \\
        \texttt{olim3d} & $3.923\times 10^{-6}$ & 1.013 \\
      \end{tabular}
    }
    \caption{$T_N \sim C_T N^{\alpha n}$}
  \end{subtable}%
  \begin{subtable}{0.4999\textwidth}
    \centering
    {
      \small
      \begin{tabular}{ccc}
        Neighborhood & $C_E$ & $\beta$ \\
        \hline \noalign{\vskip 0.2em}
        \texttt{olim8\_mp0} & 0.4077 & 0.98744 \\
        \texttt{olim8\_mp1} & 0.3683 & 0.993 \\
        \texttt{olim8\_rhr} & 1.511 & 0.9728 \\
        \hline \noalign{\vskip 0.2em}
        \texttt{olim26\_mp0} & 2.328 & 1.3135 \\
        \texttt{olim26\_mp1} & 1.949 & 1.2888 \\
        \texttt{olim26\_rhr} & 1.772 & 0.90394 \\
        \hline \noalign{\vskip 0.2em}
        \texttt{olim3d\_mp0} & 2.268 & 1.3141 \\
        \texttt{olim3d\_mp1} & 1.865 & 1.2885 \\
        \texttt{olim3d\_rhr} & 1.77 & 0.90353
      \end{tabular}
    }
    \caption{$E_N \sim C_E h^\beta$}
  \end{subtable}
  \caption{\emph{Least-squares fits of the runtime and relative
      $\ell_\infty$ error for OLIMs in 2D and 3D.} We denote the time
    for a given $N$ by $T_N$; likewise, $E_N$ denotes the relative
    $\ell_\infty$ error for a specific $N$. We fit $T_N$ to a power
    $C_T N^\alpha$. In 2D, we expect $\alpha \approx 2$; in 3D,
    $\alpha \approx 3$. In 3D, we fit $E_N$ to $C_E h^\beta$, and
    expect $\beta \approx -1$ in all cases, due to the use of local
    factoring. In fact, for \texttt{olim26} and \texttt{olim3d} using
    either \texttt{mp0} or \texttt{mp1}, we find that the situation is
    better than expected, with
    $\beta \approx -1.3$.}\label{table:qv-least-squares}
\end{table}

\subsection{Marmousi velocity model tests}

\begin{figure}[t]
  \centering
  \includegraphics[width=\linewidth]{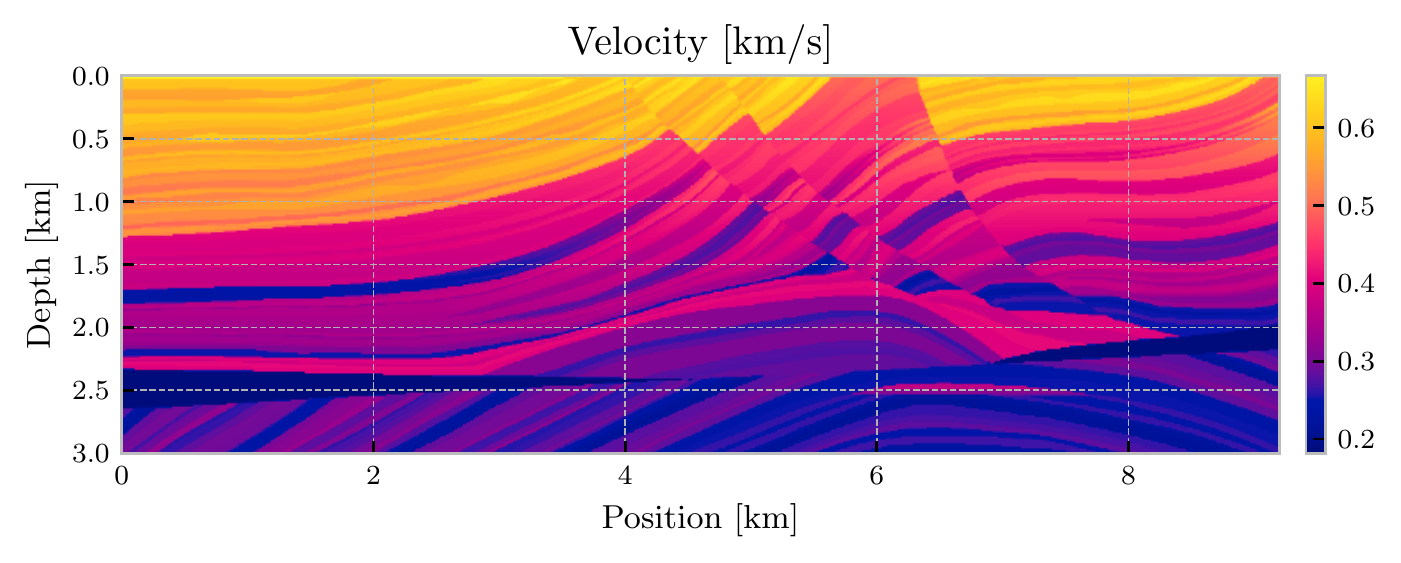}
  \caption{\emph{Marmousi slowness model.} The original Marmousi model
    is given as a velocity model $c$. We work with $s = 1/c$, so we
    plot this here. We also extrude this slowness model in the $y$
    direction to create a 3D model.}\label{fig:marmousi-slowness}
\end{figure}

\begin{figure}[t]
  \centering
  \includegraphics[width=0.3333\linewidth]{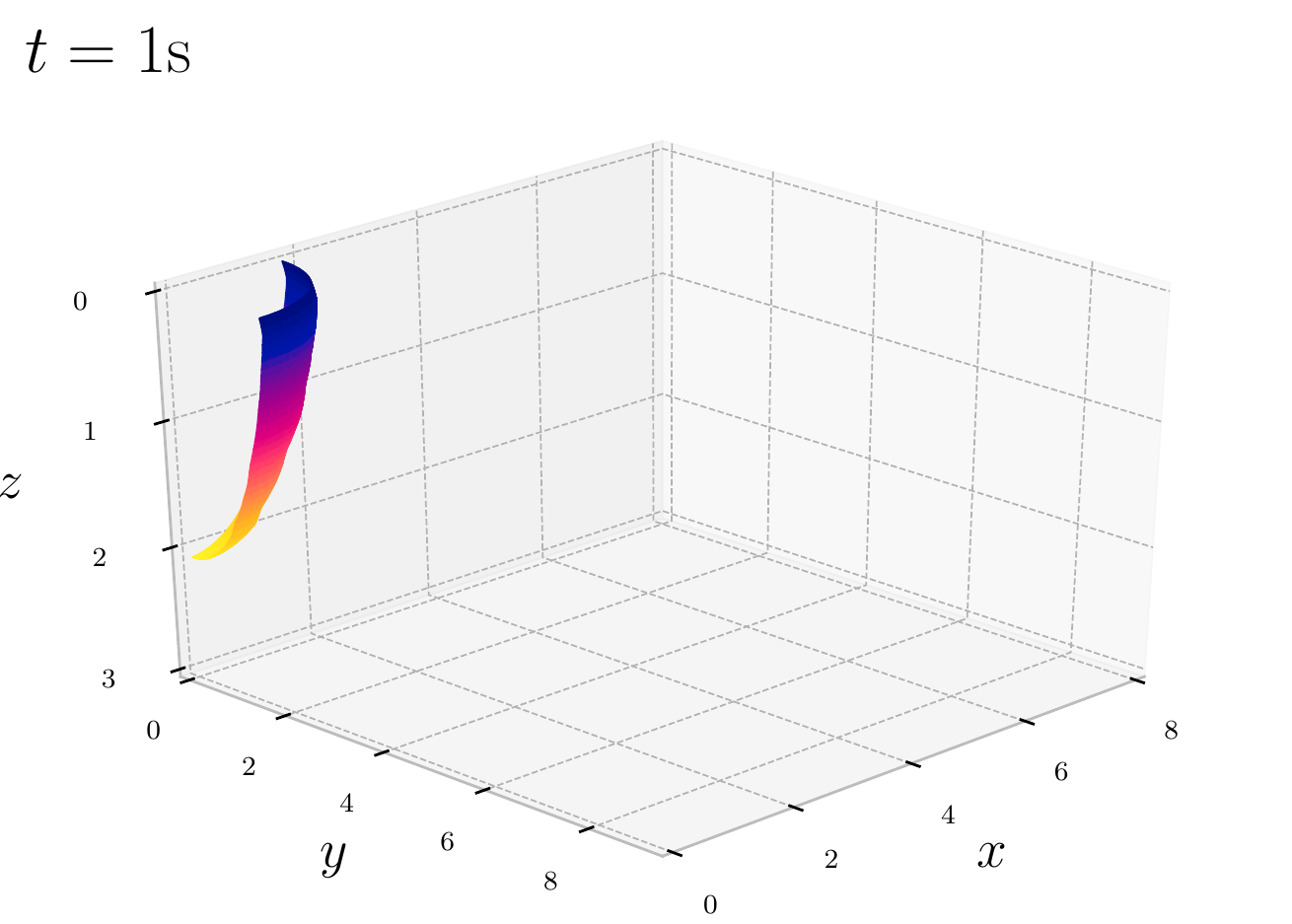}%
  \includegraphics[width=0.3333\linewidth]{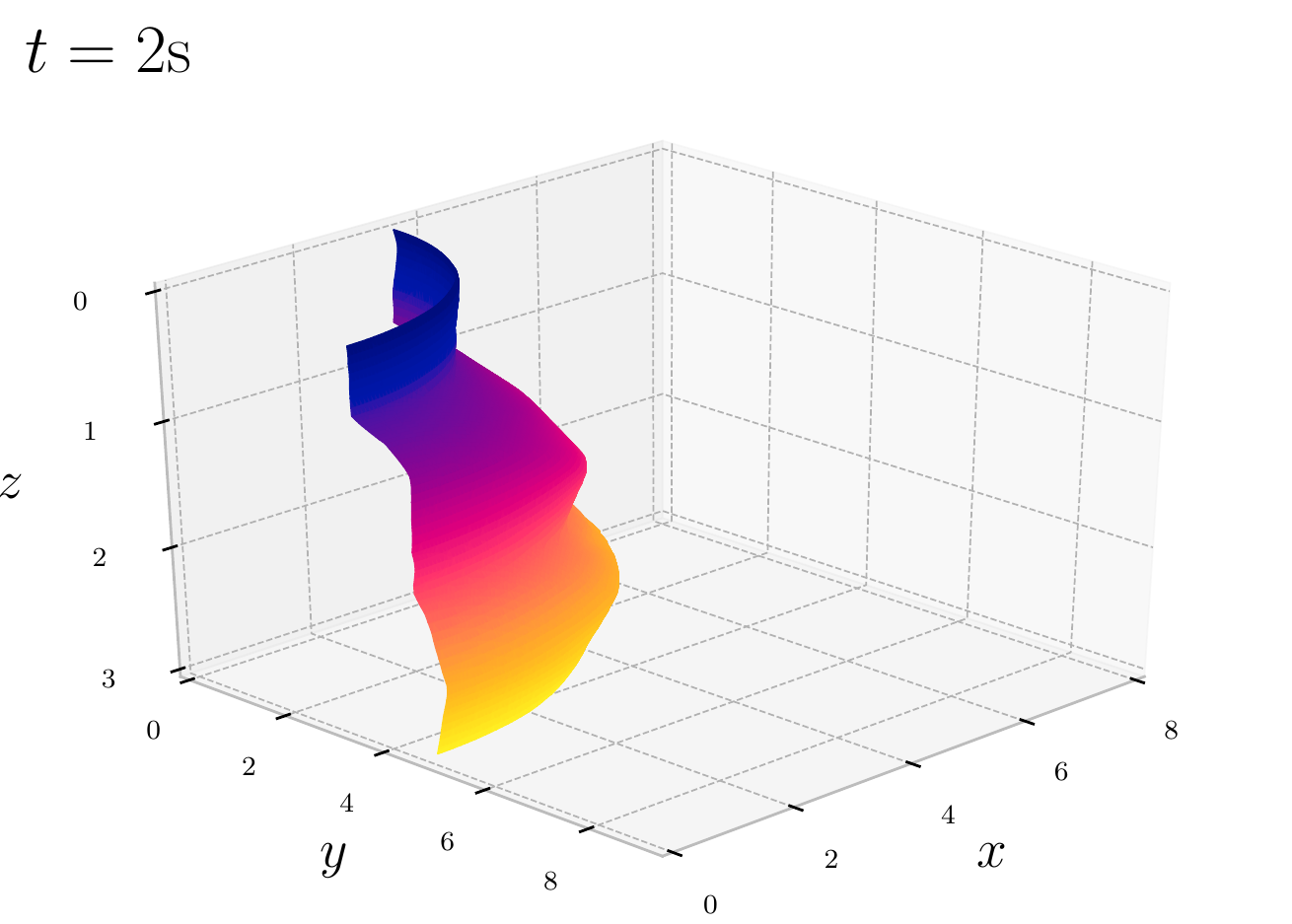}%
  \includegraphics[width=0.3333\linewidth]{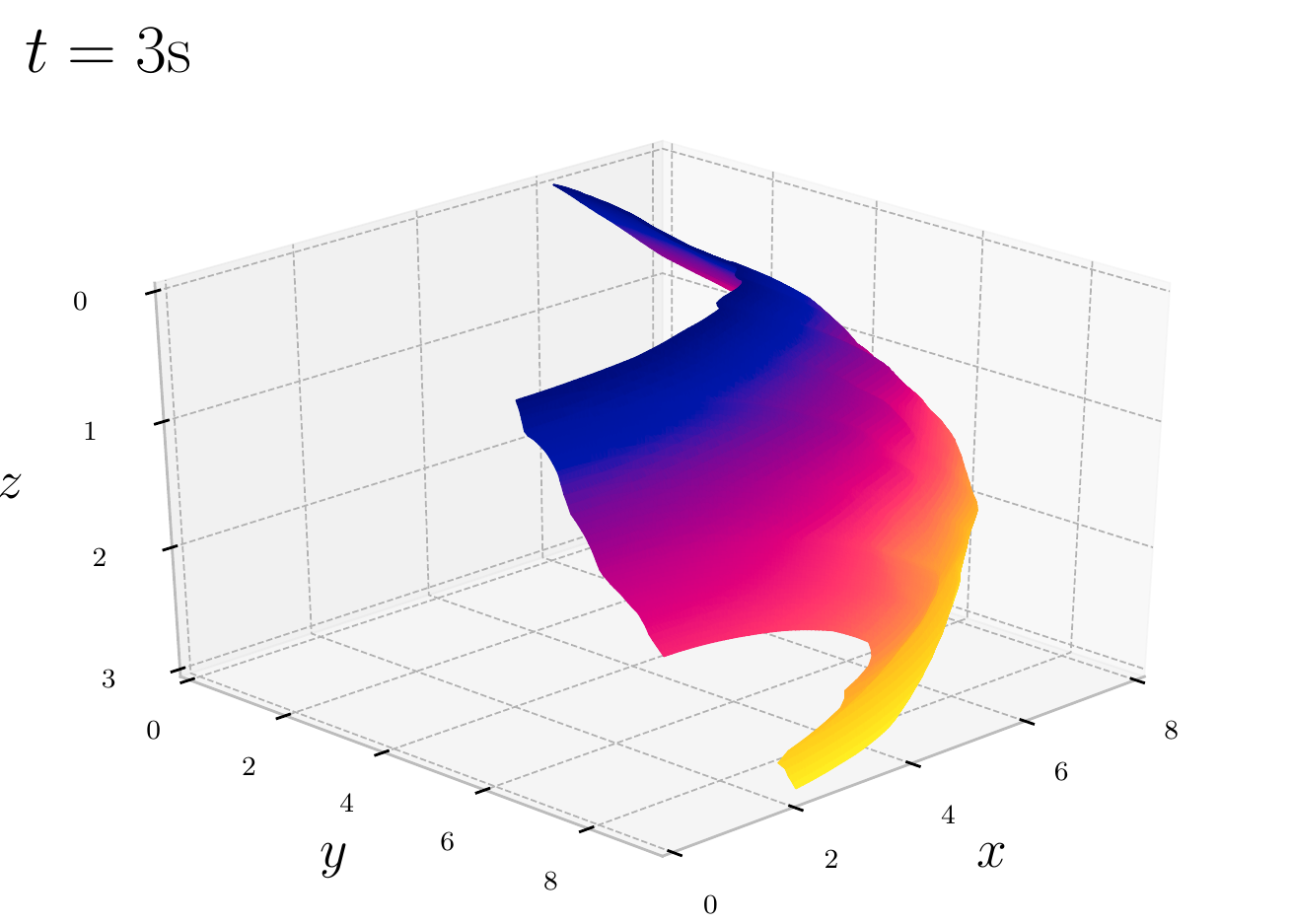}
  \caption{ Level sets of $U$ (first arrival time) for
    $t = 1\si{s}, 2\si{s}, 3\si{s}$ for the extruded 3D
    Marmousi slowness model computed on a $94 \times 288 \times 288$
    grid using \texttt{olim3d\_mp0}. Distances are in \si{km}. The
    level sets were computed using Lewiner's version of the marching
    cubes algorithm~\cite{lewiner2003efficient} as implemented in
    scikit-image~\cite{van2014scikit}.}\label{fig:marmousi-3d-t}
\end{figure}

\begin{figure}[t]
  \centering
  \includegraphics[width=\linewidth]{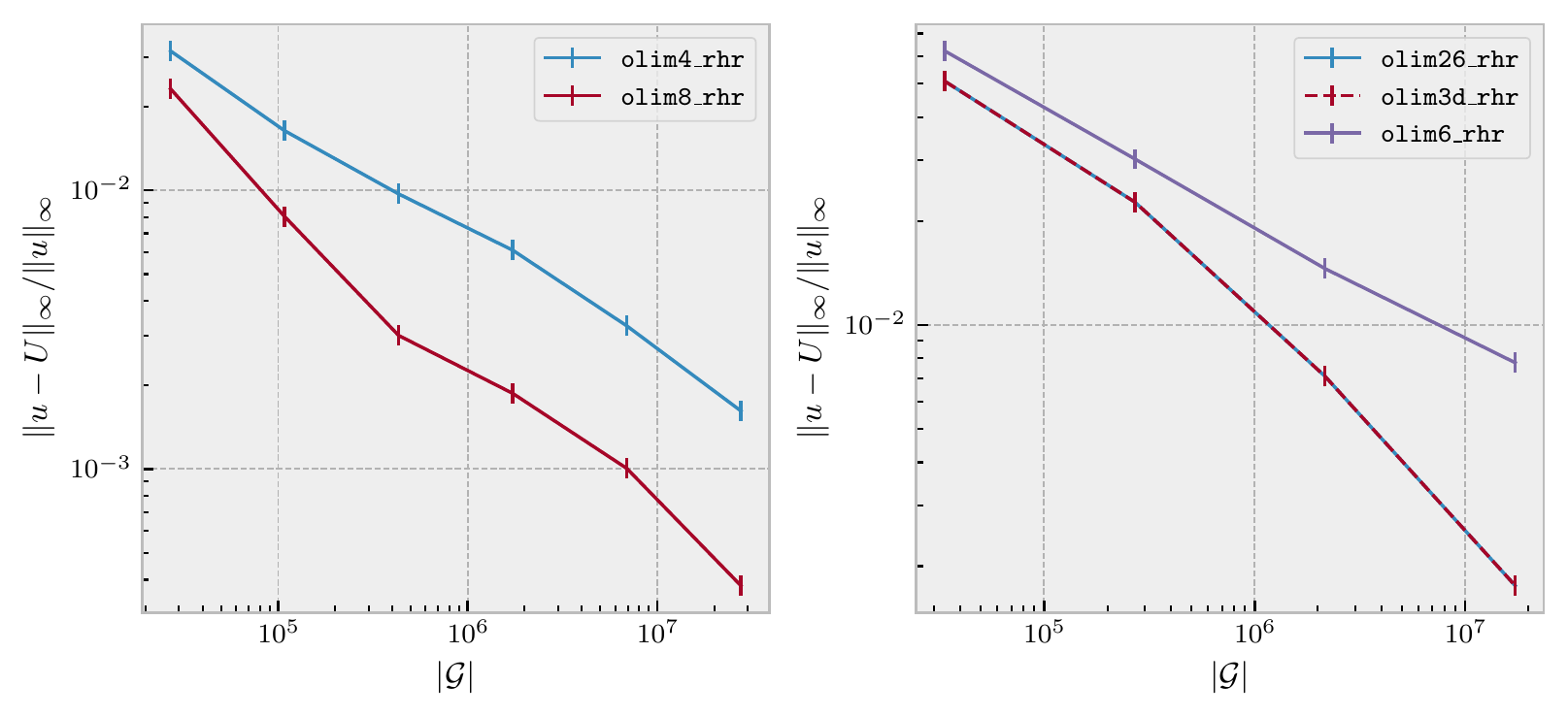}
  \caption{\emph{Relative $\ell_\infty$ errors for (unsmoothed)
      Marmousi point source problems in 2D and 3D using $\Frhr$.} The
    horizontal axis for each plot is the total number of grid nodes,
    $|\mathcal{G}|$. Left: 2D plots for \texttt{olim4\_rhr} and
    \texttt{olim8\_rhr}. Right: 3D plots for \texttt{olim6\_rhr},
    \texttt{olim26\_rhr}, and \texttt{olim3d\_rhr}. The 3D slowness
    model is obtained by extruding the original model in the $y$
    direction. 2D Note that the plots for \texttt{olim26\_rhr} and
    \texttt{olim3d\_rhr} overlap.}\label{fig:marmousi-rhr}
\end{figure}

A standard test problem in exploration geophysics is the Marmousi
velocity model, which is a synthetic velocity model based on the North
Quenguela Trough---see the following citation for more background
information~\cite{versteeg1994marmousi}. The model consists of a
number of stratified layers in the downward $z$ direction that are
roughly piecewise constant. This model has been used frequently as a
stress test for eikonal solvers, since the solution of the eikonal
equation estimates the first arrival time of a seismic P-wave.

For 2D tests, if $c_{\si{m}}(x, z)$ is the standard Marmousi velocity
model in $\si{m/s}$, we first convert to $\si{km/s}$ and set
$s(x, z) = 1/c_{\si{km}}(x, z) = 1000/c_{\si{m}}$. For the 3D tests,
we just extrude the model in the $z$ direction, setting
$s(x, y, z) \equiv s(x, z)$. We plot our slowness model in fig.\
\ref{fig:marmousi-slowness}. The domain used for 2D problems is
$\Omega = [0, 9.2] \times [0, 3]$, and for 3D problems we set
$\Omega = [0, 9.2] \times [0, 9.2] \times [0, 3]$ (all distances in
$\si{km}$). For each test, we set $\boundary = \{0\}$. To get a sense
of how a P-wave propagates in the extruded model, see fig.\
\ref{fig:marmousi-3d-t}, where we plot several level sets at one
second increments.

If we compare the quadrature rules defined in eq.\@
\ref{eq:quadrature-rules}, we can see that if:
\begin{equation}
  \hat{s} = s_0 = \cdots = s_d = \mbox{constant},
\end{equation}
i.e., if an update is performed in a region that is locally constant,
then $\Frhr \equiv \Fmpzero \equiv \Fmpone$. For the Marmousi model,
we can expect this to occur in most regions, with the exception of
updates that straddle two adjacent (approximately) piecewise constant
regions. Naturally, our assumption that $s$ is Lipschitz breaks down
for the Marmousi model.

In our numerical experiments, we have found $\Frhr$ to handle this
situation better than $\Fmpzero$ or $\Fmpone$, although all of our
solvers converge reasonably well for this model. We found that
increased directional coverage does lead to signficantly improved
accuracy. To demonstrate this phenomenon, we create a sequence of
scaled Marmousi slowness models in 2D and 3D. In 2D, we used the sizes
(288, 94), (575, 188), (1150, 376), (2301, 751), (4602, 1502), (9204,
3004), and (18408, 6008); in 3D, we used the sizes (144, 5, 47), (288,
10, 94), (575, 20, 188), (1150, 40, 376), (2301, 80, 751). In 2D, for
our ``ground truth'' solution, we solve the problem at the finest
resolution using \texttt{olim8\_rhr}; for each of the smaller problem
sizes, we downsample the groundtruth solution and compare with the
solution computed using the smaller size to estimate the relative
$\ell_\infty$ error. We do the same in 3D, but using
\texttt{olim3d\_rhr}. See fig.\@ \ref{fig:marmousi-rhr}.

Additionally, we have included supplemental numerical experiments
examining the behavior of different choices of quadrature rules on our
solvers' performance in 2D
online~\cite{marmousi-2d-online-supplement}.

\section{Conclusion}

We have presented a family of fast and accurate direct solvers for the
eikonal equation. The \emph{top-down} algorithm relies on enumerating
valid update simplexes, while the \emph{bottom-up} algorithm employs a
fast search for the first arrival characteristic. For each of these
solvers, one can use different quadrature rules: a simplified midpoint
rule (\texttt{mp0}), a midpoint rule (\texttt{mp1}), and a righthand
rule (\texttt{rhr}).

We have analyzed the relationship between these quadrature rules,
showing that the \texttt{mp0} rule can be used to compute an
approximate local characteristic direction, and $\hat{U}$ evaluated
using this direction, while incurring only $O(h^3)$ error per update,
which justifies its use.

We have conducted extensive numerical experiments that show that
\texttt{olim3d\_mp0} provides the best overall trade-off between
runtime and error. We also compare the speed of the standard fast
marching method in 3D with the equivalent \texttt{olim6\_rhr}
(equivalent in the sense that they compute the same solution to
machine precision). We demonstrate that \texttt{olim6\_rhr} incurs
only a very modest overhead, suggesting that the \emph{top-down}
approach is an efficient way of generalizing the fast marching method;
it also suggests that the \emph{bottom-up} approach is a viable
approach to speeding up Dijkstra-like algorithms in 3D, and should be
viable for other types of algorithms that solve related equations
(indeed, this has already been demonstrated for the
quasipotential~\cite{yang2019computing}).

To determine the relative time spent on different tasks, we have
profiled our C++ implementation using Valgrind, separating time spent
into several coarse-grained categories. From this, we show that for
practical problem sizes, the runtime of Dijkstra-like algorithms
behaves like $C N^n$, where $n = 2, 3$, and $N^n$ is the total number
of gridpoints (even if this is not strictly true from a computational
complexity viewpoint); we also emphasize that memory access patterns
play a large role in algorithm runtime, especially for large $N$.

We conclude that ordered line integral methods are a powerful approach
to obtaining a higher degree of accuracy when solving the eikonal
equation in 3D. With an appropriate choice of quadrature rule, we are
able to exploit improved directional coverage to drive down the error
constant. The improved accuracy more than makes up for the modest
price paid in speed, and we fully expect it to be possible to find
ways to optimize this family of algorithms further. We have also
attempted to demonstrate that memory access patterns dominate both
update time and time spent maintaining the front data structure, from
which we can conclude two things: 1) the exact time spent updating a
node is important but not paramount (improving accuracy is more
important than improving speed), 2) using memory optimally will lead
to a substantial speed-up for large problems.

\section{Acknowledgements}

We thank Prof.\ A.\ Vladimirsky for valuable discussions during the
course of this project.

\appendix

\section[Minimum action integral]{Minimum actional integral for the
  eikonal equation}\label{sec:minimum-action-integral} The eikonal
equation (eq.\@ \ref{eq:eikonal}) is a Hamilton-Jacobi equation for
$u$. If we let each fixed characteristic (ray) of the eikonal equation
be parametrized by some parameter $\sigma$ and denote
$p \equiv \nabla u$, the corresponding Hamiltonian is:
\begin{equation}
  \label{eq:eikonal-hamiltonian}
  H{(p, x)} = \frac{\norm{p}^2}{2} - \frac{s(x)^2}{2} = 0.
\end{equation}
Since $H = 0$, eq.\@ \ref{eq:eikonal-hamiltonian} implies
$L = \sup_p (\langle p, x' \rangle - H) = s(x) \norm{x'}$. Since
$x' = \partial_p H = p$ and $\norm{p} = s(x)$ can be expressed as:
\begin{equation}
  \label{eq:eikonal-lagrangian}
  L(x, x') = \langle p, x'\rangle = \langle x', x'\rangle = \langle \nabla u, x' \rangle = \frac{du}{d\sigma}.
\end{equation}

Let $x(\sigma)$ be a characteristic arriving at
$\hat{x} = x(\hat\sigma)$ from $x_0 = x(0)$, which lies on the
expanding front. Integrating from $0$ to $\hat\sigma$ and letting
$\hat u = u(\hat x)$ and $u_0 = u(x_0)$:
\begin{equation}
  \label{eq:minimum-action-on-a-ray}
  \hat u - u_0 = \int_{0}^{\hat\sigma} L(x, x') d\sigma = \int_{0}^{\hat\sigma} s(x) \norm{x'} d\sigma = \int_0^L s(x) dl,
\end{equation}
where $L$ is the length of the characteristic from $x_0$ to $\hat{x}$
and $dl$ is the length element. A characteristic of eq.\@
\ref{eq:eikonal} minimizes eq.\@ \ref{eq:minimum-action-on-a-ray}
over admissible paths. Then, if $\hat{x}$ is fixed and $\alpha$ is an
arc-length parametrized curve with $\alpha(L) = \hat{x}$, eq.\@
\ref{eq:minimum-action-on-a-ray} is equivalent to:
\begin{equation}\label{eq:eikonal-minimum-action-path}
  \hat{u} = u(\hat{x}) = \min_\alpha \curlyb{u(\alpha(0)) + \int_\alpha s(x) dl}.
\end{equation}
Our update procedure is based on eq.\@
\ref{eq:eikonal-minimum-action-path}. This problem may have multiple
local minima---$\hat{u}$ above corresponds to the first arrival, which
is what interests us primarily in this work.

\section{Skipping updates in the \emph{bottom-up} family of
  algorithms}\label{sec:kkt-skipping}

In this section, we describe how to use the KKT conditions to skip
updates in the \emph{bottom-up} algorithms. In this section, we write:
\begin{equation}
  A = \begin{bmatrix}
    -1 & & \\
    & \ddots & \\
    & & -1 \\
    1 & \cdots & 1
  \end{bmatrix} \in \mathbb{R}^{d + 1 \times d}, \qquad b = \begin{bmatrix}
    0 \\ \vdots \\ 0 \\ 1
  \end{bmatrix} \in \mathbb{R}^{d + 1}
\end{equation}
Using these, the set $\Delta^d$ can be written as a linear matrix
inequality:
\begin{equation}
  \lambda \in \Delta^d \iff A\lambda \leq b
\end{equation}
Let $\mu \in \mathbb{R}^{d + 1}$ be the vector of Lagrange
multipliers. Then, the Lagrangian function for eq.\@
\ref{eq:constrained-minimization} is:
\begin{equation}
  L(\lambda, \mu) = F(\lambda) + (A\lambda - b)^\top \mu.
\end{equation}
Since $F_0$ is strictly convex and since we assume $h$ is small enough
for $F_1$ to be strictly convex, if $\lambda^*$ lies on the boundary
of $\Delta^d$, we only need to check that the optimum Lagrange
multipliers $\mu^*$ are dual feasible; i.e., whether $\mu^* \geq 0$
(this follows directly from the standard KKT
conditions~\cite{bertsekas1999nonlinear,nocedal2006numerical}). For a
fixed $\lambda \in \Delta^d$, define the set of indices of active
constraints:
\begin{equation}
  \mathcal{I} = \set{i : (A\lambda - b)_i = 0}
\end{equation}
That is, $i \in \mathcal{I}$ if the $i$th inequality holds with
equality (``is active''). Stationarity then requires:
\begin{equation}\label{eq:stationarity}
  A^\top_{\mathcal{I}} \mu_{\mathcal{I}}^* = \nabla F_i(\lambda).
\end{equation}
If $i \notin \mathcal{I}$, we set $\mu_i^* = 0$. If $\mu^*_i \geq 0$
for all $i$, then the update may be skipped.

When implementing this, since $A$ is sparse, it is simplest and most
efficient to write out the system given by eq.\@ \ref{eq:stationarity}
and write a specialized function to solve it. Note that since we
always start with a lower-dimensional interior point solution lying on
the boundary of a higher-dimensional problem, we only have to compute
one Lagrange multiplier.

\section{Proofs for section\@
  \ref{ssec:minimization-problem}}\label{sec:minimization-proofs}

\begin{proof}[Proof of proposition \ref{prop:F0-grad-and-Hess}]
  For the gradient, we have:
  \begin{equation*}
    \nabla F_0(\lambda) = \delU + \frac{s^{\theta} h}{2 \|p_\lambda\|} \nabla p_\lambda^\top p_\lambda = \delU + \frac{s^{\theta} h}{\|p_\lambda\|} \delP^\top p_\lambda,
  \end{equation*}
  since
  $\nabla p_\lambda^\top p_\lambda = 2 \delP^\top
  p_\lambda$. For the Hessian:
  \begin{align*}
    \nabla^2_\lambda F_0(\lambda) &= \nabla \parens{\frac{s^{\theta} h}{\|p_\lambda\|} p_\lambda^\top \delP} = s^{\theta} h \parens{\nabla \frac{1}{\|p_\lambda\|} p_\lambda^\top \delP + \frac{1}{\|p_\lambda\|} \nabla p_\lambda^\top \delP} \\
    &= \frac{s^{\theta} h}{\|p_\lambda\|} \parens{\delP^\top \delP - \frac{\delP^\top p_\lambda p_\lambda^\top \delP}{p_\lambda^\top p_\lambda}} = \frac{s^{\theta} h}{\|p_\lambda\|} \delP^\top \parens{I - \frac{p_\lambda p_\lambda^\top}{p_\lambda^\top p_\lambda}} \delP,
  \end{align*}
  from which the result follows.
\end{proof}

\begin{proof}[Proof of proposition \ref{prop:F1-grad-and-Hess}]
  Since
  $F_1(\lambda) = u_\lambda + h s^{\theta}_\lambda \|p_\lambda\|$,
  for the gradient we have:
  \begin{equation*}
    \nabla F_1(\lambda) = \delU + h \parens{\theta \|p_\lambda\| \dels + \frac{s^{\theta}_\lambda}{2\|p_\lambda\|} \nabla p_\lambda^\top p_\lambda} = \delU + \frac{h}{\|p_\lambda\|} \parens{\theta p_\lambda^\top p_\lambda \dels + s^{\theta} \delP^\top p_\lambda},
  \end{equation*}
  and for the Hessian:
  \begin{equation*}
    \begin{aligned}
      \nabla^2 F_1(\lambda) = \frac{h}{2 \|p_\lambda\|} \Bigg(\theta \Big(\nabla p_\lambda^\top p_\lambda \dels^\top &+\; \dels {(\nabla p_\lambda^\top p_\lambda)}^\top\Big) \;+ \\
      &s^{\theta}_\lambda \parens{\frac{1}{2 p_\lambda^\top p_\lambda} \nabla p_\lambda^\top p_\lambda {(\nabla p_\lambda^\top p_\lambda)}^\top - \nabla^2_\lambda p_\lambda^\top p_\lambda} \Bigg).
    \end{aligned}
  \end{equation*}
  Simplifying this gives us the result.
\end{proof}

\begin{proof}[Proof of lemma \ref{lemma:dPt-cprojp-dP-pd}]

  Let $\nu_\lambda = p_\lambda/\|p_\lambda\| \in \mathbb{R}^n$ be the unit
  vector in the direction of $p_\lambda$, and assume that
  $Q = \begin{bmatrix} \nu_\lambda & U \end{bmatrix} \in \mathbb{R}^{n
    \times n}$ is orthonormal. Then:
  \begin{equation}
    \delP^\top \proj^\perp_{p_\lambda} \delP = \delP^\top {(I - \nu_\lambda \nu_\lambda^\top)} \delP = \delP^\top {(QQ^\top - \nu_\lambda \nu_\lambda^\top)} \delP = \delP^\top U U^\top \delP.
  \end{equation}
  Hence, $\delP^\top \proj^\perp \delP$ is a Gram matrix
  and positive semidefinite.

  Next, since $\Delta^n$ is nondegenerate, the vectors $p_i$ for
  $i = 0, \hdots, n - 1$ are linearly independent. Since the $i$th
  column of $\delP$ is $\delp_i = p_i - p_0$, we can see that
  the vector $p_0$ is not in the range of $\delP$; hence, there is
  no vector $\mu$ such that $\delP \mu = \alpha p_\lambda$, for any
  $\alpha \neq 0$. What's more, by definition,
  $\text{ker}(\proj_{p_\lambda}^\perp) = \langle p_\lambda
  \rangle$. So, we can see that
  $\proj^\perp_{p_\lambda} \delP \mu = 0$ only if $\mu = 0$,
  from which we can conclude
  $\delP^\top \proj^\perp_{p_\lambda} \delP \succ
  0$. Altogether, bearing in mind that $\smin$ is assumed to be
  positive, we conclude that $\nabla^2 F_0$ is positive definite.
\end{proof}

\begin{proof}[Proof of lemma \ref{lemma:F-strictly-convex}]
  To show that $\nabla^2 F_1$ is positive definite for $h$ small
  enough, note from eq.\@ \ref{eq:hess-F1} that
  $\nabla^2 F_1 = A + B$, where $A$ is positive definite and $B$ is
  small relative to $A$ and indefinite. To use this fact, note that
  since $\delP^\top \proj^\perp_\lambda \delP$ is symmetric positive
  definite, it has an eigenvalue decomposition $Q \Lambda Q^\top$
  where $\Lambda_{ii} > 0$ for all $i$. Since
  $\delP^\top \proj^\perp_\lambda \delP$ doesn't depend on $h$, for a
  fixed set of vectors $p_0, \hdots, p_n$, its eigenvalues are
  constant with respect to $h$. So, defining:
  \begin{equation}
    A = \frac{s^\theta_\lambda h}{\|p_\lambda\|} \delP^\top \proj^\perp_\lambda \delP = Q \parens{\frac{s^\theta_\lambda h}{\|p_\lambda\|} \Lambda} Q^\top
  \end{equation}
  we can expect this matrix's eigenvalues to be $\Theta(h)$; in
  particular, $\lambda_{\min} \geq C h$ for some constant $C$,
  provided that $s > \smin > 0$, as assumed. This gives us a bound for
  the positive definite part of $\nabla F_1^2$.

  The perturbation $B = \set{\delP^\top \nu_\lambda, \theta h \dels}$
  is indefinite. Since $\norm{\dels} = O(h)$, we find that:
  \begin{equation}
    |\lambda_{\max}(B)| = \norm{\anticom{\delP^\top \nu_\lambda,
        \theta h \dels}}_2 \leq \theta h \sqrt{n} \norm{\anticom{\delP^\top \nu_\lambda, \dels}}_\infty = O(h^2),
  \end{equation}
  where we use the fact that the Lipschitz constant of $s$ is
  $K \leq C$, so that:
  \begin{equation}
    |\dels_i| = |s_i - s_0| \leq K |x_i - x_0| \leq K h \sqrt{n}
    \leq Ch \sqrt{n},
  \end{equation}
  for each $i$. Letting $z \neq 0$, we compute:
  \begin{equation}
    z^\top \nabla^2 F_1 z = z^\top A z + z^\top B z \geq \lambda_{\min}(A) z^\top z + z^\top B z \geq Ch z^\top z + z^\top B z.
  \end{equation}
  Now, since
  $\abs{z^\top B z} \leq \abs{\lambda_{\max}(B)} z^\top z \leq D h^2
  z^\top z$, where $D$ is some positive constant, we can see that for
  $h$ small enough, it must be the case that
  $Ch z^\top z + z^\top B z > 0$; i.e., that $\nabla^2 F_1$ is
  positive definite; consequently, $F_1$ is strictly convex in this case.
\end{proof}

\section{Proofs for
  section\@ \ref{ssec:validation}}\label{app:validation-proofs} In this section, we establish some technical lemmas that we will use
to validate the use of \texttt{mp0}. Lemmas
\ref{lemma:bounded-inv-hess-F1}, \ref{lemma:bounded-first-step}, and
\ref{lemma:hess-F1-lipschitz} set up the conditions for theorem
\ref{thm:stoer-bulirsch} of Stoer and
Bulirsch~\cite{stoer2013introduction}, from which theorem
\ref{thm:mp0-newton} readily follows.

\begin{lemma}\label{lemma:bounded-inv-hess-F1}
  There exists $\beta = O(h^{-1})$ s.t.
  $\norm{\nabla^2 F_1(\lambda)^{-1}} \leq \beta$ for all
  $\lambda \in \Delta^n$.
\end{lemma}

\begin{proof}[Proof of lemma \ref{lemma:bounded-inv-hess-F1}]
  To simplify eq.\@ \ref{eq:hess-F1}, we temporarily define:
  \begin{equation}
    A = \frac{s^\theta_\lambda h}{\|p_\lambda\|} \delP^\top \proj^\perp_\lambda \delP \mbox{ and } B = \frac{\theta h}{\|p_\lambda\|} \anticom{\delP^\top p_\lambda, \dels}.
  \end{equation}
  Observe that $\|A\| = O(h)$ and $\|B\| = O(h^2)$, since
  $\|\delta s\| = O(h)$ and since all other factors involved in $A$
  and $B$ (excluding $h$ itself) are independent of $h$. Hence:
  \begin{equation}
    \norm{A^{-1} B} = \frac{\theta}{s^\theta_\lambda} \norm{\parens{\delP^\top \proj^\perp_\lambda \delP}^{-1} \anticom{\delP^\top p_\lambda, \dels}} = O(h),
  \end{equation}
  since $\norm{\dels} = O(h)$. Hence, $\norm{A^{-1} B} < 1$ for $h$
  small enough, and we can Taylor expand:
  \begin{equation}
    \begin{aligned}
      \nabla^2 F_1(\lambda)^{-1} &= \parens{A + B}^{-1} = {(I + A^{-1} B)}^{-1} A^{-1} \\
      &= \parens{I - A^{-1} B + {(A^{-1} B)}^2 - \cdots} A^{-1} \\
      &= A^{-1} - A^{-1} B A^{-1} + {(A^{-1} B)}^2 A^{-1} - \,\cdots,
    \end{aligned}
  \end{equation}
  which implies $\norm{\nabla^2 F_1(\lambda)^{-1}} = O(h^{-1})$. Note
  that when we Taylor expand, $\|A^{-1} B\| = O(h)$, so that
  $\|A^{-1} B\| < 1$ for $h$ small enough. To define $\beta$, let:
  \begin{equation}
    \beta = \max_{\lambda \in \Delta^n} \norm{\nabla^2 F_1(\lambda)^{-1}} = O(h^{-1}),
  \end{equation}
  completing the proof.
\end{proof}

\begin{lemma}\label{lemma:bounded-first-step}
  There exists $\alpha = O(h)$ s.t.
  $\norm{\nabla^2 F_1(\lambda_{0}^*)^{-1} \nabla F_1(\lambda_{0}^*)}
  \leq \alpha$.
\end{lemma}

\begin{proof}[Proof of lemma \ref{lemma:bounded-first-step}]
  From lemma \ref{lemma:bounded-inv-hess-F1} we have
  $\norm{F_1(\lambda_0^*)^{-1}} = O(h^{-1})$, so to establish the
  result we only need to show that
  $\norm{\nabla F_1(\lambda_0^*)} = O(h^2)$. To this end, let
  $\underline{\lambda} = {(n + 1)}^{-1} \m{1}_{n \times 1}$ (i.e., the
  centroid of $\Delta^n$, where $s^\theta$ is evaluated). Then,
  recalling figure \ref{fig:simplex-diagrams},
  $s^\theta_\lambda = s^\theta + \dels^\top (\lambda -
  \underline{\lambda})$ so that, for a general $\lambda$:
  \begin{equation}\label{eq:grad-F1-in-terms-of-grad-F0}
    \begin{aligned}
      \nabla F_1(\lambda) &= \|p_\lambda\| h \dels + \delU + \frac{s^\theta + \dels^\top (\lambda - \underline{\lambda})}{\|p_\lambda\|} h \delP^\top p_\lambda \\
      &= \|p_\lambda\| h \dels + \nabla F_0(\lambda) + \frac{\dels^\top {(\lambda - \underline{\lambda})}}{\|p_\lambda\|} h \delP^\top p_\lambda.
    \end{aligned}
  \end{equation}
  Since $\nabla F_0(\lambda_0^*) = 0$ by optimality, we can conclude
  using eq.\@ \ref{eq:grad-F1-in-terms-of-grad-F0} and
  $\norm{\dels} = O(h)$ that:
  \begin{equation}
    \norm{\nabla F_1(\lambda_0^*)} = h \norm{\|p_{\lambda_0^*}\| \dels + \frac{\dels^\top {(\lambda - \underline{\lambda})}}{\|p_{\lambda_0^*}\|} \delP^\top p_\lambda} = O(h^2),
  \end{equation}
  which proves the result.
\end{proof}

\begin{lemma}\label{lemma:hess-F1-lipschitz}
  The Hessian $\nabla^2 F_1$ is Lipschitz continuous with $O(h)$
  Lipschitz constant. That is, there is some constant $\gamma = O(h)$
  so that for two points $\lambda$ and $\lambda'$:
  \begin{align*}
    \norm{\nabla^2 F_1(\lambda) - \nabla^2 F_1(\lambda')} \leq \gamma \norm{\lambda - \lambda'}.
  \end{align*}
\end{lemma}

\begin{proof}[Proof of lemma \ref{lemma:hess-F1-lipschitz}]
  If we restrict our attention to $\Delta^n$, we see that
  $\|p_\lambda\|^{-1} \delP^\top \proj_\lambda^\perp \delP$ is
  Lipschitz continuous function of $\lambda$ with $O(1)$ Lipschitz
  constant and $\theta \{{\delP}^\top p_\lambda, \dels\} /\|p_{\lambda}\|$
  is Lipschitz continuous with $O(h)$ Lipschitz constant since
  $\|\delta s\| = O(h)$. Then, since $s^\theta_\lambda$ is $O(1)$
  Lipschitz, it follows that:
  \begin{equation}
    A(\lambda) = \tfrac{s^\theta_\lambda h}{\|p_\lambda\|} \delP^\top
    \proj^\perp_\lambda \delP
  \end{equation}
  has a Lipschitz constant that is $O(h)$ for $\lambda \in \Delta^n$,
  using the notation of lemma
  \ref{lemma:bounded-inv-hess-F1}. Likewise,
  \begin{equation}
    B(\lambda) = \tfrac{\theta h}{\|p_\lambda\|} \anticom{\delP^\top
      p_\lambda, \dels} = O(h^2),
  \end{equation}
  since it is a sum of two terms involving products of $h$ and
  $\delta s$. Since $\nabla^2 F_1(\lambda) = A(\lambda) + B(\lambda)$,
  we can see immediately that it is also Lipschitz on $\Delta^n$ with
  a constant that is $O(h)$.
\end{proof}

\begin{proof}[Proof of theorem \ref{thm:mp0-newton}]
  Our proof of theorem \ref{thm:mp0-newton} relies on the following
  theorem on the convergence of Newton's method, which we present for
  convenience.

  \begin{theorem}[Theorem 5.3.2, Stoer and Bulirsch]\label{thm:stoer-bulirsch}
    Let $C \subseteq \R^n$ be an open set, let $C_0$ be a convex set
    with $\overline{C}_0 \subseteq C$, and let
    $f : C \to \mathbb{R}^n$ be differentiable for $x \in C_0$ and
    continuous for $x \in C$. For $x_0 \in C_0$, let
    $r, \alpha, \beta, \gamma$ satisfy
    $S_r(x_0) = \set{x : \norm{x - x_0} < r} \subseteq C_0$,
    $\mu = \alpha\beta\gamma < 2$, $r = \alpha(1 - \mu)^{-1}$, and let
    $f$ satisfy:
    \begin{enumerate}[label=(\alph*)]
    \item for all $x, y \in C_0$,
      $\norm{D f(x) - D f(y)} \leq \gamma \norm{x - y}$,
    \item for all $x \in C_0$, $(D f(x))^{-1}$ exists and satisfies
      $\norm{(Df(x))^{-1}} \leq \beta$,
    \item and $\norm{(Df(x_0))^{-1} f(x_0)} \leq \alpha$.
    \end{enumerate}
    Then, beginning at $x_0$, each iterate:
    \begin{equation}
      x_{k+1} = x_k - Df(x_k)^{-1} f(x_k), \qquad k = 0, 1, \hdots,
    \end{equation}
    is well-defined and satisfies $\norm{x_k - x_0} < r$ for all
    $k \geq 0$. Furthermore, $\lim_{k \to \infty} x_k = \xi$ exists and
    satisfies $\norm{\xi - x_0} \leq r$ and $f(\xi) = 0$.
  \end{theorem}

  For our situation, Theorem 5.3.2 of Stoer and
  Bulirsch~\cite{stoer2013introduction} indicates that if:
  \begin{align}
    \|\nabla F_1(\lambda)^{-1}\| &\leq \beta, \mbox{where } \beta = O(h^{-1}) \label{enum:sb-newton-1}, \\
    \|\nabla F_1(\lambda_0^*)^{-1} \nabla F_1(\lambda_0^*)\| &\leq \alpha, \mbox{where } \alpha = O(h), \label{enum:sb-newton-2} \mbox{ and} \\
    \|\nabla F_1(\lambda) - \nabla F_1(\lambda')\| &\leq \gamma \norm{\lambda - \lambda'} \mbox{ for each } \lambda, \lambda' \in \Delta^n, \mbox{where } \gamma = O(h), \label{enum:sb-newton-3}
  \end{align}
  then with $\lambda_0 = \lambda_0^*$, the iteration eq.\@
  \ref{eq:lam0-iter-to-lam1} is well-defined, with each iterate
  satisfying $\norm{\lambda_k - \lambda_0} \leq r$, where
  $r = \alpha/(1 - \alpha\beta\gamma/2)$. Additionally, the limit of
  this iteration exists, and the iteration converges to it
  quadratically; we note that since $F_1$ is strictly convex for $h$
  small enough, the limit of the iteration must be $\lambda_1^*$, so
  the theorem also gives us
  $\norm{\dellam^*} = \norm{\lambda_1^* - \lambda_0^*} \leq r$.

  Now, we note that items \ref{enum:sb-newton-1},
  \ref{enum:sb-newton-2}, and \ref{enum:sb-newton-3} correspond
  exactly to lemma \ref{lemma:bounded-inv-hess-F1},
  \ref{lemma:bounded-first-step}, and \ref{lemma:hess-F1-lipschitz},
  respectively which gave us values for $\alpha, \beta$, and
  $\gamma$. All that remains is to compute $r$. Since the preceding
  lemmas imply $\alpha\beta\gamma = O(h)$, hence
  $\alpha\beta\gamma/2 < 1$ for $h$ small enough. We have:
  \begin{equation}
    r = \frac{\alpha}{1 - \frac{\alpha\beta\gamma}{2}} = \alpha \parens{1 + \frac{\alpha\beta\gamma}{2} + \frac{\alpha^2\beta^2\gamma^2}{4} + \cdots} = O(h),
  \end{equation}
  so that $\norm{\dellam^*} = O(h)$, and the result follows.

  To obtain the $O(h^3)$ error bound, from theorem
  \ref{thm:mp0-newton}, we have $\norm{\dellam^*} = O(h)$. Then,
  Taylor expanding $F_1(\lambda_0^*)$, we get:
  \begin{equation*}
    F_1(\lambda_0^*)
    = F_1(\lambda_1^* - \dellam^*) = F_1(\lambda_1^*) - \nabla F_1(\lambda_1^*)^\top \dellam^* + \frac{1}{2} \dellam^* \nabla F_1^2(\lambda_1^*) \dellam^* + R,
  \end{equation*}
  where $\abs{R} = O(\norm{\dellam^*}^3)$. Since $\lambda_1^*$
  is optimum, $\nabla F_1(\lambda_1^*) = 0$. Hence:
  \begin{equation*}
    \abs{F_1(\lambda_1^*) - F_1(\lambda_0^*)} \leq \frac{1}{2} \norm{\nabla F_1^2(\lambda_1^*)} \norm{\dellam^*}^2 + O(\norm{\dellam^*}^3) = O(h^3),
  \end{equation*}
  which proves the result.
\end{proof}

\section{Proofs for section\@
  \ref{ssec:exact-soln}}\label{sec:exact-soln-proofs}

\begin{figure}
  \vspace{1em}
  \centering \includegraphics[width=0.8\linewidth]{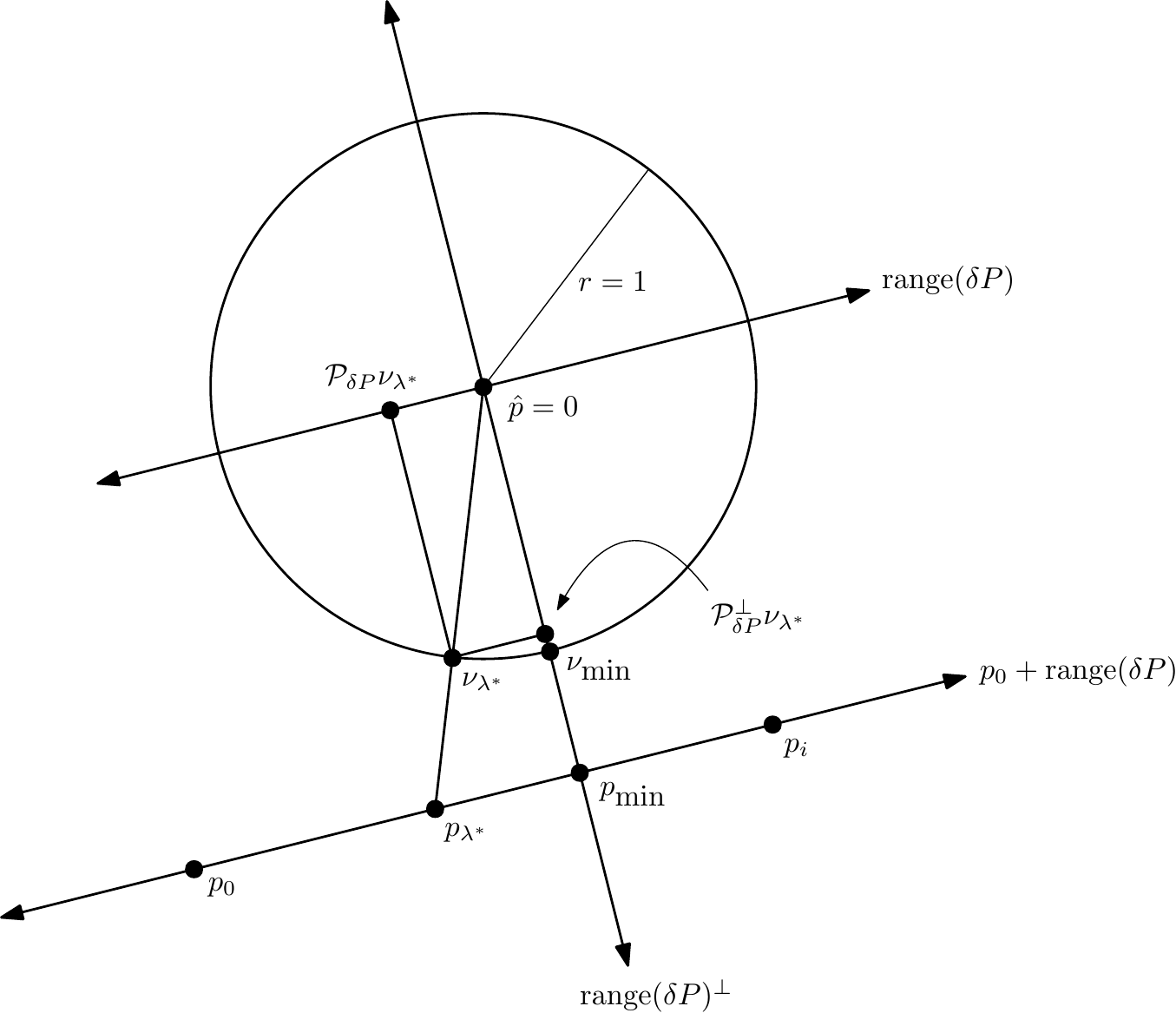}
  \caption{A schematic depiction of the proof of theorem
    \ref{thm:f0-exact}.}\label{fig:f0-exact}
\end{figure}

\begin{proof}[Proof of theorem \ref{thm:f0-exact}]
  We proceed by reasoning geometrically; figure \ref{fig:f0-exact}
  depicts the geometric setup. First, letting $\delP = QR$ be the
  reduced QR decomposition of $\delP$, and writing
  $\nu_{\lambda^*} = p_{\lambda^*}/\|p_{\lambda^*}\|$, we note that since:
  \begin{equation}
    \nabla F_0(\lambda^*) = \delU + s^\theta h \delP^\top \nu_{\lambda^*} = 0,
  \end{equation}
  the optimum $\lambda^*$ satisfies:
  \begin{equation}\label{eq:Qt-n}
    - R^{-\top} \frac{\delU}{s^\theta h} = Q^\top \nu_{\lambda^*}
  \end{equation}
  Let $\proj_{\delP} = QQ^\top$ denote the orthogonal
  projector onto $\range(\delP)$, and
  $\proj^\perp_{\delP} = I - QQ^\top$ the projector onto its
  orthogonal complement. We can try to write $p_{\lambda^*}$ by
  splitting it into a component that lies in $\range(\delP)$ and
  one that lies in $\range(\delP)^\perp$. Letting $\pmin$ be the
  point in $p_0 + \range(\delP)$ with the smallest 2-norm, we
  write:
  \begin{equation}\label{eq:plam-decomp}
    p_{\lambda^*} = (p_{\lambda^*} - \pmin) + \pmin,
  \end{equation}
  where $p_{\lambda^*} - \pmin \in \range(\delP)$ and
  $\pmin \in \range(\delP)^\perp$. The vector $\pmin$ corresponds to
  $p_{\lambdamin}$ where $\lambdamin$ satisfies:
  \begin{equation}
    0 = \delP^\top (\delP \lambdamin + p_0) = R^\top R \lambdamin + R^\top Q^\top p_0,
  \end{equation}
  hence $\lambdamin = -R^{-1} Q^\top p_0$, giving us:
  \begin{equation}\label{eq:pmin}
    \pmin = p_0 + \delP \lambdamin = \proj^\perp_{\delP} p_0.
  \end{equation}
  This vector is easily obtained. For $p_{\lambda^*} - \pmin$, we note
  that $\proj_{\delP} \nu_{\lambda^*}$ is proportional to
  $p_{\lambda^*} - \pmin$, suggesting that we determine the ratio
  $\alpha$ satisfying
  $p_{\lambda^*} - \pmin = \alpha \proj_{\delP}
  \nu_{\lambda^*}$. In particular, from the similarity of the
  triangles
  $(\hat{p}, \nu_{\lambda^*}, \proj^\perp_{\delP}
  \nu_{\lambda^*})$ and $(\hat{p}, p_{\lambda^*}, \pmin)$ in figure
  \ref{fig:f0-exact}, we have, using eqs.\@ \ref{eq:Qt-n} and
  \ref{eq:pmin}:
  \begin{equation}\label{eq:alpha-solve}
    \alpha = \frac{\norm{\pmin}}{\norm{\proj^\perp_{\delP} \nu_{\lambda^*}}} = \sqrt{\frac{p_0^\top \proj^\perp_{\delP} p_0}{1 - \norm{Q^\top \nu_{\lambda^*}}^2}} = \sqrt{\frac{p_0^\top \proj^\perp_{\delP} p_0}{1 - \norm{R^{-\top} \frac{\delU}{s^\theta h}}^2}}.
  \end{equation}
  At the same time, since:
  \begin{equation}
    \nu_{\lambda^*}^\top \proj^\perp_{\delP} \nu_{\lambda^*} = \frac{{(\proj^\perp_{\delP} p_{\lambda^*})}^\top {(\proj^\perp_{\delP} p_{\lambda^*})}}{\|p_{\lambda_*}\|^2} = \frac{\pmin^\top \pmin}{\|p_{\lambda^*}\|^2} = \frac{p_0^\top \proj^\perp_{\delP} p_0}{\|p_{\lambda^*}\|^2}
  \end{equation}
  we can conclude that:
  \begin{equation}
    \|p_{\lambda^*}\| = \alpha = \sqrt{\frac{p_0^\top \proj^\perp_{\delP} p_0}{1 - \norm{R^{-\top} \frac{\delU}{s^\theta h}}^2}},
  \end{equation}
  giving us eq.\@ \ref{eq:l-star-expression}, proving the first
  part of theorem \ref{thm:f0-exact}.

  Next, combining eqs.\@ \ref{eq:Qt-n}, \ref{eq:plam-decomp},
  \ref{eq:pmin}, and \ref{eq:alpha-solve}, we get:
  \begin{equation}\label{eq:plam-final}
    p_{\lambda^*} = \proj^\perp_{\delP} p_0 - \sqrt{\frac{p_0^\top \proj^\perp_{\delP} p_0}{1 - \norm{R^{-\top} \frac{\delU}{s^\theta h}}^2}} Q R^{-\top} \frac{\delU}{s^\theta h}.
  \end{equation}
  This expression for $p_{\lambda^*}$ can be computed from our problem
  data and $\delP$. Now, note that
  $p_{\lambda^*} = p_0 + \delP \lambda^*$ implies:
  \begin{equation}\label{eq:lambda-1}
    \lambda^* = R^{-1} Q^\top (p_{\lambda^*} - p_0).
  \end{equation}
  Substituting eq.\@ \ref{eq:plam-final} into eq.\@
  \ref{eq:lambda-1}, we obtain eq.\@ \ref{eq:f0-exact-lambda} after
  making appropriate cancellations, establishing the second part of
  theorem \ref{thm:f0-exact}.

  To establish eq.\@ \ref{eq:f0-exact}, we note that by optimality
  of $\lambda^*$, our expression for $\nabla F_0$ (eq.\@
  \ref{eq:F0-grad} of proposition \ref{prop:F0-grad-and-Hess}) gives:
  \begin{equation}
    \delU = -s^\theta h \frac{\delP^\top p_{\lambda^*}}{\|p_{\lambda^*}\|}.
  \end{equation}
  This lets us write:
  \begin{equation}\label{eq:delta-U-dot-lambda-star}
    \delU^\top \lambda^* = -\frac{s^\theta h}{\|p_{\lambda^*}\|} p_{\lambda^*}^\top \delP^\top \lambda^* = \frac{s^\theta h}{\|p_{\lambda^*}\|} p_{\lambda^*}^\top {(p_0 - p_{\lambda^*})}.
  \end{equation}
  Combining eq.\@ \ref{eq:delta-U-dot-lambda-star} with our
  definition of $F_0$ yields:
  \begin{equation}
    \hat{U} = F_0(\lambda^*) = U_0 + \delU^\top \lambda^* + s^\theta h \|p_{\lambda^*}\| = U_0 + \frac{s^\theta h}{\|p_{\lambda^*}\|} p_{\lambda^*}^\top {(p_0 - p_{\lambda^*})} + \frac{s^\theta h}{\|p_{\lambda^*}\|} p_{\lambda^*}^\top p_{\lambda^*},
  \end{equation}
  which gives eq.\@ \ref{eq:f0-exact}, completing the final part of
  the proof.
\end{proof}

\section{Proofs for section\@ \ref{ssec:equivalence}}\label{sec:equivalence-proofs}

\begin{proof}[Proof of theorem \ref{thm:equivalence}]
  We assume that $U$ is a linear function in the update simplex;
  hence, $\nabla U$ is constant. By stacking and subtracting eq.\@
  \ref{eq:finite-differences} for different values of $i$, we obtain,
  for $i = 0, \hdots, n - 1$:
  \begin{equation}\label{eq:finite-diff-eq}
    \begin{bmatrix}
      \delP^\top \\
      p_i^\top
    \end{bmatrix} \nabla U = \begin{bmatrix}
      \delU \\
      U_0 - \hat{U}
    \end{bmatrix}.
  \end{equation}
  The inverse of the matrix in the left-hand side of eq.\@
  \ref{eq:finite-diff-eq} is:
  \begin{equation}
    \begin{bmatrix}
      \parens{I - \frac{\numin p_i^\top}{\numin^\top p_i}} Q R^{-\top}, &
      \frac{\numin}{\numin^\top p_i}
    \end{bmatrix},
  \end{equation}
  which can be checked. This gives us:
  \begin{equation}
    \nabla U = \parens{I - \frac{\numin p_i^\top}{\numin^\top p_i}} Q R^{-\top} \delU + \frac{U_i - \hat{U}}{\numin^\top p_i} \numin.
  \end{equation}
  Hence, $\|\nabla U\|^2$ is a quadratic equation in
  $\hat{U} - U_i$. Expanding $\|\nabla U\|^2$, a number of
  cancellations occur since $Q^\top \numin = 0$. We have:
  \begin{equation}
    \delU^\top R^{-1} Q^\top \hspace{-0.25em} \parens{I - \frac{\numin p_i^\top}{\numin^\top p_i}}^{\hspace{-0.25em}\top} \hspace{-0.4em} \parens{I - \frac{\numin p_i^\top}{\numin^\top p_i}} \hspace{-0.1em} Q R^{-\top} \delU = \norm{R^{-\top} \delU}^2 + \frac{\parens{p_i^\top Q R^{-\top} \delU}^2}{\norm{\pmin}^2},
  \end{equation}
  so that, written in standard form:
  \begin{equation}
    \begin{aligned}
      {(\hat{U} - U_i)}^2 + 2 p_i^\top Q R^{-\top} \delU {(\hat{U} - U_i)} \,&+\, \parens{p_i^\top Q R^{-\top} \delU}^2 + \\
       &\norm{\pmin}^2 \parens{\norm{R^{-\top} \delU}^2 - \parens{s^\theta h}^2} = 0.
    \end{aligned}
  \end{equation}
  Solving for $\hat{U} - U_i$ gives:
  \begin{equation}
    \hat{U} = U_i - p_i^\top Q R^{-\top} \delU + \norm{\pmin} \sqrt{\parens{s^\theta h} - \|R^{-\top} \delU\|^2},
  \end{equation}
  establishing eq.\@ \ref{eq:U-finite-diff}.

  Next, to show that $\hat{U}' = \hat{U}$, we compute:
  \begin{align*}
    \hat{U}'
    &= U_0 + \delU^\top \lambda^* + s^\theta h \|p_{\lambda^*}\| & \\
    &= U_0 - \parens{Q^\top p_0 + \|p_{\lambda^*}\| R^{-\top} \frac{\delU}{s^\theta h}}^\top R^{-\top} \delU + s^\theta h \|p_{\lambda^*}\| & \mbox{(eq.\@ \ref{eq:f0-exact-lambda})} \\
    &= U_0 - p_0^\top Q R^{-\top} \delU + s^\theta h \|p_{\lambda^*}\| \parens{1 - \norm{R^{-\top} \frac{\delU}{s^\theta h}}^2} \\
    &= U_0 - p_0^\top Q R^{-\top} \delU + \norm{\pmin} \sqrt{\parens{s^\theta h}^2 - \norm{R^\top \delU}^2} = \hat{U}. & \mbox{(eq.\@ \ref{eq:l-star-expression})}
  \end{align*}
  To establish eq.\@ \ref{eq:U-from-Ui-exact}, first note that
  $-R^{-\top} \delU = s^\theta h Q^\top \nu_{\lambda^*}$ by
  optimality. Substituting this into eq.\@ \ref{eq:U-finite-diff},
  we first obtain:
  \begin{equation}
    \hat{U} = U_i + \frac{s^\theta h}{\|p_{\lambda^*}\|} \parens{p_i^\top \proj_{\delP} p_{\lambda^*} + \norm{\pmin} \sqrt{p_{\lambda^*}^\top \proj^\perp_{\delP} p_{\lambda^*}}}.
  \end{equation}
  Now, using the notation for weighted norms and inner products, we have:
  \begin{equation}\label{eq:weighted-norms}
    p_i^\top \proj_{\delP} p_{\lambda^*} + \norm{\pmin} \sqrt{p_{\lambda^*}^\top \proj^\perp_{\delP} p_{\lambda^*}} = \langle p_i, p_{\lambda^*} \rangle_{\proj_{\delP}} + \norm{p_i}_{\proj^\perp_{\delP}} \norm{p_{\lambda^*}}_{\proj^\perp_{\delP}}.
  \end{equation}
  Since $\proj^\perp_{\delP}$ orthogonally projects onto
  $\range(\delP)^\perp$, and since the dimension of this subspace is
  1, $\proj^\perp_{\delP} p_i$ and
  $\proj^\perp_{\delP} p_{\lambda^*}$ are multiples of one
  another and their directions coincide (see figure
  \ref{fig:f0-exact}); furthermore, the angle between them is since
  our simplex is nondegenerate. So, by Cauchy-Schwarz:
  \begin{equation}\label{eq:cauchy-schwarz}
    \norm{p_i}_{\proj^\perp_{\delP}} \norm{p_{\lambda^*}}_{\proj^\perp_{\delP}} = \langle p_i, p_{\lambda^*} \rangle_{\proj^\perp_{\delP}}.
  \end{equation}
  Combining eq.\@ \ref{eq:cauchy-schwarz} with eq.\@
  \ref{eq:weighted-norms} and cancelling terms yields:
  \begin{equation}
    p_i^\top \proj_{\delP} p_{\lambda^*} + \norm{\pmin} \sqrt{p_{\lambda^*}^\top \proj_{\delP} p_{\lambda^*}} = p_i^\top p_{\lambda^*}.
  \end{equation}
  Eq.\@ \ref{eq:U-from-Ui-exact} follows.

  To parametrize the characteristic found by solving the finite
  difference problem, first note that the characteristic arriving at
  $\hat{p}$ is colinear with $\nabla \hat{U}$. If we let $\tilde{\nu}$
  be the normal pointing from $\hat{p}$ in the direction of the
  arriving characteristic, let $\tilde{p}$ be the point of
  intersection between
  $p_0 + \range(\delP)$ and
  $\operatorname{span}\hspace{-0.15em}{(\tilde{\nu})}$, and let
  $\tilde{l} = \norm{\tilde{p}}$, then, since
  $\tilde{p} - p_0 \in \range(\delP)$:
  \begin{equation}
    \numin^\top (\tilde{p} - p_0) = 0.
  \end{equation}
  Rearranging this and substituting
  $\tilde{p} = \tilde{l} \tilde{\nu}$, we get:
  \begin{equation}
    \tilde{l} = \frac{\numin^\top p_0}{\numin^\top \tilde{\nu}}.
  \end{equation}
  Now, if we assume that we can write
  $\tilde{p} = \delP \tilde{\lambda} + p_0$ for some
  $\tilde{\lambda}$, then:
  \begin{equation}
    \tilde{\lambda} = R^{-1} Q^\top \parens{\tilde{p} - p_0} = -R^{-1} Q^\top \parens{I - \frac{\tilde{\nu} \numin^\top}{\tilde{\nu}^\top \numin}} p_0.
  \end{equation}

  To see that $\tilde{p} = p_{\lambda^*}$, note that since
  $\tilde{\nu} = -\nabla \hat{U}/\norm{\nabla \hat{U}} = -\nabla
  \hat{U}/(s^\theta h)$:
  \begin{equation}
    \proj_{\delP} \tilde{\nu} = \frac{-\proj_{\delP} \nabla \hat{U}}{s^\theta h} = \frac{-QR^{-\top} \delU}{s^\theta h} = \proj_{\delP} \nu_{\lambda^*}.
  \end{equation}
  Since $\tilde{\nu}$ and $\nu_{\lambda^*}$ each lie in the unit
  sphere on the same side of the hyperplane spanned by $\delP$, and
  since $\proj_{\delP}$ orthogonally projects onto
  $\range(\delP)$, we can see that in fact
  $\tilde{\nu} = \nu_{\lambda^*}$. Hence,
  $\tilde{p} = p_{\lambda^*} \in p_0 + \range(\delP)$. The second and
  third parts of theorem \ref{thm:equivalence} follow.
\end{proof}

\section{Proofs for section\@
  \ref{ssec:causality}}\label{sec:causality-proofs}

\begin{proof}[Proof of theorem \ref{thm:causality}]
  For causality of $F_0$, we want $\hat{U} \geq \max_i U_i$, which is
  equivalent to $\min_i(\hat{U} - U_i) \geq 0$. From eq.\@
  \ref{eq:U-finite-diff}, we have:
  \begin{equation}
    \min_i \parens{\hat{U} - U_i} = s^\theta h \min_i \min_{\lambda \in \Delta^n} \frac{\nu_i^\top \nu_{\lambda}}{\norm{p_\lambda}} = s^\theta h \min_{i, j} \frac{\nu_i^\top \nu_j}{\norm{p_i}} \geq 0.
  \end{equation}
  The last equality follows because minimizing the cosine between two
  unit vectors is equivalent to maximizing the angle between them;
  since $\lambda$ is restricted to lie in $\Delta^n$, this clearly
  happens at a vertex since the minimization problem is a linear
  program.

  For $F_1$, first rewrite $s_\lambda^\theta$ as follows:
  \begin{equation}
    s_\lambda^\theta = s^\theta + \theta (s_0 + \dels^\top \lambda - \overline{s}),
  \end{equation}
  where $\overline{s} = n^{-1} \sum_{i=0}^{n-1} s_i$. If
  $\lambda_0^\star$ and $\lambda_1^\star$ are the minimizing arguments
  for $F_0$ and $F_1$, respectively, and if
  $\dellam^* = \lambda_1^* - \lambda_0^*$, then we have:
  \begin{equation}\label{eq:F1-in-terms-of-F0}
    F_1(\lambda_1^*) = F_0(\lambda_1^*) + \theta \parens{s_0 + \dels^\top \lambda_1^* - \overline{s}} h \|p_{\lambda_1^\star}\|.
  \end{equation}
  By the optimality of $\lambda_0^*$ and strict convexity of
  $F_0$ (lemma \ref{lemma:F-strictly-convex}), we can Taylor
  expand and write:
  \begin{equation}
    F_0(\lambda_1^*) = F_0(\lambda_0^*) + \nabla F_0(\lambda_0^*)^\top \dellam^* + \frac{1}{2} {\dellam^*}^\top \nabla^2 F_0(\lambda_0^*) \dellam^* + R \geq R,
  \end{equation}
  where $\abs{R} = O(h^3)$ by theorem \ref{thm:mp0-newton}. Let
  $\hat{U} = F_1(\lambda_1^*)$. Since $F_0$ is causal,
  we can write:
  \begin{equation}
    \hat{U} \geq \max_i U_i + R + \theta \parens{s_0 + \dels^\top \lambda_1^* - \overline{s}} h \|p_{\lambda_1^*}\|.
  \end{equation}
  Since $s$ is Lipschitz, the last term is $O(h^2)$---in particular,
  $\norm{\dels} = O(h)$ and $\norm{s_0 - \overline{s}} = O(h)$ since
  $s_0$ and $\overline{s}$ lie in the same simplex. So, because the
  gap $\min_i(\hat{U} - U_i)$ is $O(h)$, we can see that
  $\hat{U} \geq \max_i U_i$ for $h$ sufficiently small.
\end{proof}

\bibliographystyle{plain}
\bibliography{eikonal}{}

\end{document}